\documentclass[11pt]{article}

\usepackage[a4paper,margin=29mm]{geometry}
\usepackage[T1]{fontenc}
\usepackage[utf8]{inputenc}
\usepackage{lmodern}
\usepackage{microtype}
\usepackage{amsmath,amssymb,amsthm,mathtools}
\usepackage{array,booktabs,longtable}
\usepackage{arydshln}
\usepackage{enumitem}
\usepackage{tikz-cd}
\usepackage{graphicx}
\usepackage{xcolor}
\usepackage{hyperref}
\usepackage{aliascnt}
\usepackage{cleveref}
\hypersetup{
  colorlinks=true,
  linkcolor=blue!45!black,
  citecolor=blue!45!black,
  urlcolor=blue!45!black
}

\newcommand{\Fp}{\mathbb F_p}
\newcommand{\Ffive}{\mathbb F_5}
\newcommand{\Fnine}{\mathbb F_9}

\newcommand{\Ftwentyfive}{\mathbb F_{25}}
\newcommand{\Fthree}{\mathbb F_3}

\newcommand{\mgen}{\mathsf{X}}

\newcommand{\Cl}{\operatorname{Cl}}
\newcommand{\Gal}{\operatorname{Gal}}
\newcommand{\Hom}{\operatorname{Hom}}
\newcommand{\cd}{\operatorname{cd}}
\newcommand{\AV}{\mathrm{AV}}
\newcommand{\HT}{\operatorname{HT}}
\newcommand{\rk}{\operatorname{rk}}
\newcommand{\coker}{\operatorname{coker}}
\newcommand{\tr}{\operatorname{tr}}
\newcommand{\Spec}{\operatorname{Spec}}

\newcommand{\ab}{\mathrm{ab}}

\newtheorem{theorem}{Theorem}[section]
\newtheorem*{theoremA}{Theorem A}
\newtheorem*{theoremB}{Theorem B}
\newtheorem*{theoremC}{Theorem C}
\newtheorem*{theoremD}{Theorem D}

\newaliascnt{proposition}{theorem}
\newtheorem{proposition}[proposition]{Proposition}
\aliascntresetthe{proposition}
\crefname{proposition}{proposition}{propositions}
\Crefname{proposition}{Proposition}{Propositions}

\newaliascnt{lemma}{theorem}
\newtheorem{lemma}[lemma]{Lemma}
\aliascntresetthe{lemma}
\crefname{lemma}{lemma}{lemmas}
\Crefname{lemma}{Lemma}{Lemmas}

\newaliascnt{corollary}{theorem}
\newtheorem{corollary}[corollary]{Corollary}
\aliascntresetthe{corollary}
\crefname{corollary}{corollary}{corollaries}
\Crefname{corollary}{Corollary}{Corollaries}

\theoremstyle{definition}
\newaliascnt{definition}{theorem}
\newtheorem{definition}[definition]{Definition}
\aliascntresetthe{definition}
\crefname{definition}{definition}{definitions}
\Crefname{definition}{Definition}{Definitions}

\newaliascnt{algo}{theorem}
\newtheorem{algo}[algo]{Algorithm}
\aliascntresetthe{algo}
\crefname{algo}{algorithm}{algorithms}
\Crefname{algo}{Algorithm}{Algorithms}

\theoremstyle{remark}
\newaliascnt{remark}{theorem}
\newtheorem{remark}[remark]{Remark}
\aliascntresetthe{remark}
\crefname{remark}{remark}{remarks}
\Crefname{remark}{Remark}{Remarks}

\newaliascnt{example}{theorem}

\aliascntresetthe{example}
\crefname{example}{example}{examples}
\Crefname{example}{Example}{Examples}
\crefname{appendix}{Appendix}{Appendices}
\Crefname{appendix}{Appendix}{Appendices}

\title{Mild $p$-Class Tower Groups of Imaginary Quadratic Fields}
\author{Denis Vogel}
\date{}

\newcommand{\blfootnote}[1]{%
  \begingroup
  \renewcommand{\thefootnote}{}%
  \footnote{#1}%
  \addtocounter{footnote}{-1}%
  \endgroup
}

\begin{document}
\maketitle

\vspace{-2.5em}
\begin{abstract}
Let $K$ be an imaginary quadratic number field, let $p$ be an
odd prime, and let
$G=G_{\varnothing}(K)(p)$ be the Galois
group of the maximal everywhere unramified pro-$p$ extension of
$K$.  To each mod-$p$ character $x$ of $G$ we associate a
linear map $D_x$ from $\Cl(K)[p]$ to $\Cl(K)/p$;
a formula of Ahlqvist and Carlson expresses it through the
class of a norm ideal in the unramified cyclic degree-$p$
extension attached to $x$.  These maps determine all triple Massey products
on $H^1(G,\Fp)$, and with them the cubic initial relations of
$G$.

Suppose that the $p$-class rank $d=\dim_{\Fp}\Cl(K)/p$ is at
least three.  The $(d-1)\times(d-1)$ minors of the family
$x\mapsto D_x$ define a subscheme $\Sigma_D$ of
$\mathbb P^{d-1}_{\Fp}$, an invariant of $K$, the
norm-degeneracy scheme.  We prove: if the rank condition
$\rk D_x=d-2$ holds transversally at a point of $\Sigma_D$,
over some finite extension of $\Fp$, then $G$ is mild, and
hence of cohomological dimension $2$.  For $p>3$ transversality
means that $\Sigma_D$ is smooth of dimension $d-3$ at the
point; at $p=3$ the kernel of the Bockstein map enters as an
additional constraint.

We treat every imaginary quadratic field of $p$-class rank at
least three with $|D_K|<2^{30}$.  Only $p=3$, $5$, and $7$ occur.  The criterion decides $206$ of the $207$ fields at $p=5$ and $7$,
and $505$ of the $12\,750$ fields at $p=3$, where the Bockstein
condition restricts its reach.  A direct computation with the cubic
initial relations settles the remaining field at $p=5$ and a
further $11\,765$ at $p=3$, $26$ of them not mild, while $480$
remain undecided.  In all, the $p$-class tower group is proved mild for
$12\,451$ of the $12\,957$ fields.  These appear to be the
first number fields for which the full maximal everywhere
unramified pro-$p$ Galois group is proved to be mild, and hence of
cohomological dimension~$2$.
\end{abstract}

{\small
\setcounter{tocdepth}{1}
\tableofcontents}

\section{Introduction}\label{sec:introduction}

Let $K$ be a number field and let $p$ be a prime.  Write
$G_{\varnothing}(K)(p)$ for the $p$-class tower group, the Galois
group of the maximal everywhere unramified pro-$p$ extension of
$K$.  It is finite exactly when the $p$-class field tower of $K$
terminates.  Whether that always happens was a classical
question.  Golod and Shafarevich answered it negatively in 1964,
giving a criterion for infinitude and exhibiting fields that
satisfy it.  Since then the group has been studied from
several sides: the type of presentation it admits,
classifications in small rank, explicit candidates for tower
groups, and heuristics for their distribution
(see for instance
\cite{VenkovKoch,McLeman,BartholdiBush,BostonBushHajir,AhlqvistPink}).

The cohomological dimension of the group is an invariant whose
value has structural consequences.  For a
pro-$p$ group it is at most one precisely for the free groups,
it is finite only for torsion-free groups, and
it does not increase on passing to closed subgroups.  For an
infinite $p$-class tower group it has not been determined.
Hajir--Maire--Ramakrishna record the computation as an open
problem \cite[\S3.4]{HMR}.  The partial results, due to Maire,
are exclusions: the maximal unramified pro-$2$ group of an
imaginary quadratic field is almost never of cohomological
dimension two, a special case of
\cite[Theorem~4.6]{Maire2018}.

One way to reach the invariant is through the defining relations
of the group.  There is a well-developed theory of mild pro-$p$
groups, initiated by Labute.  For its genesis see his own account \cite{LabuteGenesis}.
A pro-$p$ group is mild if it admits a presentation whose initial defining
relations (the lowest nonzero homogeneous parts of the relators with respect
to the Zassenhaus filtration) form a strongly free sequence.  Here strong
freeness is a condition on the Hilbert series of the quotient of the free
associative algebra on the generators of the presentation by the two-sided
ideal generated by these initial forms.  The precise definition is recalled in
\cref{sec:relations}.  This graded-algebra condition implies cohomological dimension $2$.
For initial forms lying in the free Lie algebra this is Labute's
theorem \cite{Labute2006}.  G\"artner's extension \cite{Gaertner}
covers arbitrary homogeneous initial forms, in particular the
restricted $p$-th powers that occur at $p=3$.  In arithmetic the theory has been developed for Galois groups
with \emph{restricted ramification}.  For a number field $k$ and a finite set $S$ of primes of $k$, let $G_S(k)$
denote the Galois group of the maximal pro-$p$ extension of $k$ unramified
outside $S$.  Labute
constructed mild groups $G_S(\mathbb Q)$ for suitable finite tame sets $S$
\cite{Labute2006}.  Schmidt gave a cohomological criterion for mildness in
terms of cup products and applied Labute's theory in this arithmetic setting
\cite[Theorem~5.5]{Schmidt2007}.  G\"artner later extended this viewpoint to
higher Massey products \cite{Gaertner}.  Other subsequent work includes
Labute--Minac \cite{LabuteMinac}, Salle \cite{Salle}, Maire
\cite{Maire2014}, and Bouazzaoui--El Habibi, who treat imaginary quadratic
base fields with nonempty tame ramification set \cite{BouazzaouiElHabibi}.
On the tame side,
Feuerpfeil--Hamza--Lim show that, starting from a finite tame set,
one can add two suitably chosen tame primes and obtain cohomological dimension
$2$ \cite{FeuerpfeilHamzaLim}.  In all these constructions the set $S$ is nonempty and can be
arranged so that the initial relations satisfy a mildness
criterion.

In the everywhere unramified case $S$ is empty, and the group
is determined by $K$ alone.  The
purpose of this paper is to close the gap there in one concrete
case.  From here on $K$ is an imaginary quadratic number
field and $p$ an odd prime.  For $p=3$ we exclude the single
field $K=\mathbb Q(\sqrt{-3})$.  We write $G=G_{\varnothing}(K)(p)$,
and $d$ for the $p$-class rank $\dim_{\Fp}\Cl(K)/p$.

\subsection*{Massey products and the arithmetic problem}

Complex conjugation forces the cup product in the group
cohomology of $G_{\varnothing}(K)(p)$,
\[
  H^1(G_{\varnothing}(K)(p),\Fp)\times H^1(G_{\varnothing}(K)(p),\Fp)
  \longrightarrow H^2(G_{\varnothing}(K)(p),\Fp)
\]
to vanish (Koch--Venkov \cite{VenkovKoch}).  Here and throughout, the cohomology of a profinite group is
understood in the usual sense, using continuous cochains and discrete
coefficients, and $\Fp$ always carries the trivial action.  Higher cohomological
information therefore begins with the triple Massey products.  The
author's interest in these products goes back to his dissertation
\cite[Corollary~2.3.2]{Vogel} (see also \cite{VogelCrelle}).  There the triple
Massey products are defined without obstruction in the everywhere
unramified case and give a trilinear
pairing on $(\Cl(K)/p)^*$, whose traces against the defining relations
are the cubic Magnus coefficients \cite[Theorem~1.2.6]{Vogel}.  What
interpretation that pairing admits was raised there as a question, and no
satisfactory intrinsic answer was available at the time.

Ahlqvist and Carlson supplied the missing arithmetic input.  Their explicit
formula for triple Massey products in the étale cohomology of rings of
integers expresses the relevant evaluations through norm ideals in cyclic
unramified degree-$p$ extensions \cite{AhlqvistCarlson}.  With this
formula they produced, for odd $p$, the first
known imaginary quadratic fields of $p$-class rank two with infinite
$p$-class tower.
More recently, Ahlqvist--Pink made explicit the comparison with the
cohomology of the $p$-class tower group and translated the
presentation formula in Proposition~1.3.3 of \cite{Vogel} into étale
cohomology \cite[\S7, Propositions~7.5 and~7.8]{AhlqvistPink}.

The present paper uses the triple Massey products differently.  We
assemble all of them into a single trilinear map, recover from this map
the cubic initial relation space, and then
test that relation space for strong freeness.

For each character
\[
  x\in H^1(G_{\varnothing}(K)(p),\Fp),
\]
we introduce an $\Fp$-linear map
\[
  D_x\colon \Cl(K)[p]\longrightarrow \Cl(K)/p,
\]
which we call the \emph{secondary norm operator attached to $x$}.  For
$x\ne0$, the Artin map identifies $x$ with a character of $\Cl(K)/p$,
and class field theory associates to the kernel of that character the
cyclic unramified degree-$p$ extension $L_x| K$.  Ahlqvist--Carlson's formula expresses $D_x$ through
the class of a norm ideal from this extension.  The operators $D_x$ record
that term for all characters at once.

The secondary norm operators vary quadratically with the
character.  More precisely, their failure to be additive is bilinear.  We call
the resulting bilinear map the \emph{polarization} of the secondary norm
operators.  It follows that, if the $p$-class rank is $d$, the operators
attached to a basis of the character space and to its pairwise
sums (altogether $d(d+1)/2$ of them) determine all secondary
norm operators.

We show that the triple Massey product
\[
  \langle x,y,z\rangle\in H^2(G_{\varnothing}(K)(p),\Fp),
  \qquad x,y,z\in H^1(G_{\varnothing}(K)(p),\Fp),
\]
is recovered from the polarization in the two outer variables $x$ and $z$,
with the middle character $y$ providing the corresponding evaluation.

\begin{theoremA}[Reconstruction]
Let $K$ be an imaginary quadratic number field, let $p$ be an odd
prime ($K\neq\mathbb Q(\sqrt{-3})$ if $p=3$), and put
\[
  G=G_{\varnothing}(K)(p),\qquad d=\dim_{\Fp}\Cl(K)/p.
\]
For a basis $\chi_1,\ldots,\chi_d$ of $H^1(G,\Fp)$, the
$d(d+1)/2$ secondary norm operators
\[
  D_{\chi_i},\qquad D_{\chi_i+\chi_j}\quad(1\le i<j\le d)
\]
determine the complete triple Massey product
\[
  M\colon H^1(G,\Fp)\times H^1(G,\Fp)\times H^1(G,\Fp)
  \longrightarrow H^2(G,\Fp).
\]
More precisely, for all $x,y,z\in H^1(G,\Fp)$ and
$e\in\Cl(K)[p]$,
\[
  \bigl\langle M(x,y,z),e\bigr\rangle_{\AV}
  =y\bigl((D_x+D_z-D_{x+z})(e)\bigr).
\]
Here $\langle\cdot\,,\cdot\rangle_{\AV}$ is the perfect
Artin--Verdier pairing between $H^2(G,\Fp)$ and $\Cl(K)[p]$,
recalled in \cref{subsec:comparison}.
\end{theoremA}

Theorem~A is proved in numbered form as \cref{thm:reconstruction}.  Combined
with the Ahlqvist--Carlson norm formula, it turns a finite collection of
unramified cyclic class-field computations into the complete cubic Massey
data needed for the relation theory below.

We call the quadratic map $x\mapsto D_x$ the \emph{secondary norm
family}.  The $d(d+1)/2$ operators of Theorem~A determine its
bilinear polarization.  Being bilinear, the polarization extends
by scalars to every finite extension of $\Fp$, and the family is
recovered from it.  This is the canonical scalar extension of the
family; the quadratic map itself is not simply tensored.  From
here on we assume that the $p$-class rank $d$ is at least three.
The following criterion extracts mildness directly from the
secondary norm family.

\begin{theoremB}[Transversality criterion]
Let $K$ be an imaginary quadratic number field, let $p$ be an odd prime, and
assume
\[
  d=\dim_{\Fp}\Cl(K)/p\ge3.
\]
Let $k|\Fp$ be a finite extension and let
\[
  0\ne x\in
  H^1(G_{\varnothing}(K)(p),\Fp)\otimes_{\Fp}k.
\]
Extend the secondary norm family $D$ to $k$ by polarization and suppose
that
\[
  D_x\colon \Cl(K)[p]\otimes_{\Fp}k
  \longrightarrow
  \Cl(K)/p\otimes_{\Fp}k
\]
has rank $d-2$ and that the diagonal Massey product vanishes at $x$,
a condition that holds for every $x$ when $p>3$ and that for $p=3$
places $x$ in the kernel of the Bockstein map.  Using
\[
  \Cl(K)/p\otimes_{\Fp}k
  \simeq
  \bigl(H^1(G_{\varnothing}(K)(p),\Fp)\otimes_{\Fp}k\bigr)^*,
\]
the polarization induces a linear map
\[
  \Theta_x\colon 
  \bigl(H^1(G_{\varnothing}(K)(p),\Fp)\otimes_{\Fp}k\bigr)/kx
  \longrightarrow
  \Hom_k\!\left(\ker D_x,\,x^\perp/\operatorname{im}D_x\right),
\]
Here $x^\perp\subseteq\Cl(K)/p\otimes_{\Fp}k$ is the annihilator
of $x$ under the evaluation pairing.
If $\Theta_x$ is surjective, we call $x$ a \emph{transverse
element}.  In that case $G_{\varnothing}(K)(p)$ is mild, and
hence
\[
  \cd G_{\varnothing}(K)(p)=2.
\]
\end{theoremB}

Theorem~B is proved in numbered form as
\cref{thm:transverse-rank-one}.  From a transverse element
the proof constructs new generators $X_1,\ldots,X_d$ of the free
associative algebra carrying the initial relations, chosen so that
the $d$ cubic initial relations have leading words
\[
  X_1X_1X_3,\ \ldots,\ X_1X_1X_d,\qquad
  X_1X_2X_{k_1},\qquad X_1X_2X_{k_2}
\]
for two indices $2\le k_1<k_2\le d$, in the degree-lexicographic
order with $X_1>\cdots>X_d$.  These $d$ words are combinatorially
free, and Anick's criterion applies.  Independently of
this choice, Theorem~A always yields the full coefficient matrix
of the $d$ initial relations in all $d^3$ cubic words, which may
be used with other strong-freeness criteria when the
transversality criterion does not apply.

The $(d-1)\times(d-1)$ minors of the secondary norm family
$x\mapsto D_x$ define a closed subscheme
$\Sigma_D\subseteq\mathbb P\bigl(H^1(G,\Fp)\bigr)
\simeq\mathbb P^{d-1}_{\Fp}$, the
\emph{norm-degeneracy scheme}.  It does not depend on any of the choices
made, and is therefore an invariant of $K$ itself.  For every
finite extension $k|\Fp$, its $k$-points are the classes $[x]$
of the characters $x\in H^1(G,\Fp)\otimes_{\Fp}k$ with
$\rk D_x\le d-2$.  The vanishing locus of the diagonal map
$x\mapsto M(x,x,x)$ of the Massey product is a second invariant, the
\emph{Bockstein cone} $C_\beta$.
For $p>3$ this map vanishes identically, and $C_\beta$ is
the whole space.  For $p=3$ it is the Bockstein map, which is
linear in $x$, and $C_\beta$ is the projectivized kernel, a
linear subspace whose codimension is the number of cyclic
factors of order exactly $3$ in the $3$-part of $\Cl(K)$.
At a point of $C_\beta$ with $\rk D_x=d-2$, the map $\Theta_x$ of
Theorem~B detects whether this rank condition holds
transversely there.  For $p>3$, the map $\Theta_x$ is surjective
if and only if $\Sigma_D\otimes_{\Fp}k$ is smooth of dimension
$d-3$ at $[x]$; for $d=3$ and every odd $p$, if and only if
$[x]$ is a reduced isolated point of $\Sigma_D\otimes_{\Fp}k$
(\cref{prop:rank-drop-geometry}, spelled out in
\cref{rem:two-equations}).  This is the geometric content of
the word \emph{transverse}, and it turns Theorem~B into an
equivalent geometric statement, proved as
\cref{cor:geometric-mildness}:

\begin{theoremC}[Geometric mildness criterion]
Let $K$ be an imaginary quadratic number field and let $p$ be
an odd prime.  Assume $d=\dim_{\Fp}\Cl(K)/p\ge3$, and for
$p=3$ assume $d=3$.  Then a finite extension $k|\Fp$ and a
transverse element $x$ as in Theorem~B exist if and only if
\begin{enumerate}[label=\textup{(\alph*)}]
\item for $p>3$: the norm-degeneracy scheme $\Sigma_D$ has a
closed point at which it is smooth of dimension $d-3$;
\item for $p=3$: the scheme $\Sigma_D$ has a reduced isolated
closed point lying on the Bockstein cone $C_\beta$.
\end{enumerate}
In either case $G_{\varnothing}(K)(p)$ is mild, and hence
$\cd G_{\varnothing}(K)(p)=2$; for $k$ one may take the residue
field of the point.
\end{theoremC}

In this formulation the auxiliary pair $(k,x)$ of Theorem~B
acquires an intrinsic meaning.  Mildness is detected on an invariant
of $K$, the scheme $\Sigma_D$.  This has two practical
consequences.  First, the computational search for transverse elements
(\cref{subsec:rank-one-computation}) is conducted on this
invariant.
The candidate directions, over every finite extension alike, are
the solutions of one fixed system of equations over $\Fp$, the
minors that define $\Sigma_D$.  And the extensions are
genuinely needed.  At $p=5$, $52$ of the fields treated below
have no transverse element over $\Ffive$ itself, the relevant
closed point having degree between $2$ and $6$.  Second, failure of
the criterion becomes a well-defined statement about $K$ rather
than about an unsuccessful search, and this is what makes the
field-by-field bookkeeping of \cref{sec:example-p3} possible.

Several
questions suggested by the construction remain open.  Is the
norm-degeneracy scheme of dimension $d-3$ for every field of
$p$-class rank $d\ge3$, or at least under natural arithmetic
hypotheses?  Can the existence and the residue degrees of the
points at which the criterion applies be controlled
arithmetically?  These are smooth closed points of dimension
$d-3$, and for $d=3$ reduced isolated closed points.  And does the geometric
characterization of transversality extend to $p=3$ in rank
$d>3$?  The description of the tangent spaces of $\Sigma_D$
carries over.  The proof of the smoothness part needs the
identity $\operatorname{im}D_x\subseteq x^\perp$ at every point
of a chart, and $p=3$ grants it only on the Bockstein cone;
whether the smoothness part itself extends is open.

The class-group tables of Mosunov--Jacobson are
unconditional to $|D_K|<2^{40}$ \cite{MosunovJacobson}.  Within
$|D_K|<2^{30}$, fields of $p$-class rank three occur only for
$p=3$, $5$, and $7$, with $12\,749$, $204$, and three fields.
A single further field at $p=3$ has rank four
(\cref{subsec:exceptional-fields}).  We treat every one of these
fields, with the following outcome.  The
proof is computer-assisted, and \cref{app:verification} specifies
what is computed and how it is checked.

\begin{theoremD}[Mild $p$-class tower groups for $|D_K|<2^{30}$]
Let $K$ be an imaginary quadratic number field of $p$-class rank
at least three for an odd prime $p$, with $|D_K|<2^{30}$.  Then the
$p$-class tower group
\[
  G_{\varnothing}(K)(p)=\Gal(K_{\varnothing}(p)| K)
\]
is an infinite mild FAb pro-$p$ group, and in particular
\[
  \cd G_{\varnothing}(K)(p)=2,
\]
in each of the following cases:
\begin{itemize}
\item $p=5$;
\item $p=7$;
\item $p=3$ and $K$ is one of the $12\,243$ fields of $3$-class
  rank three listed in Parts~W6 and~W7 of \cite{Repository};
\item $p=3$ and $K$ is the field of $3$-class rank four,
  $D_K=-653329427$.
\end{itemize}
For $26$ further fields at $p=3$ the cubic initial relation space
is not strongly free, so their tower groups are not mild with
respect to the Zassenhaus filtration.
\end{theoremD}

Here FAb means that every open subgroup has finite abelianization.
We say ``mild FAb'' rather than Labute's \emph{fabulous}
\cite{LabuteFabulous}, because fabulous groups are in addition quadratic,
while our defining relations are cubic.

\Cref{sec:results} counts the fields decided by each route, and
Parts~W5--W7 of
\cite{Repository} record for every field the datum that
decides it.
The transversality criterion decides $711$ of the $12\,451$ mild fields,
and a direct verification of strong freeness of the cubic initial
forms decides the others.  That direct verification takes one of two forms,
both applied after an invertible linear change of variables, over
$\Fthree$ or a finite extension (\cref{rem:flag}).  An \emph{Anick
witness} is such a change of variables after which the high terms of
the cubic initial forms are combinatorially free in the sense of
Anick \cite{Anick}.  Such a witness proves strong freeness, and hence
mildness, directly (\cref{rem:direct-anick}).  A \emph{word count} computes,
degree by degree, the dimensions of the quotient of the free algebra
by the ideal of the cubic initial forms and compares them with the
coefficients of $1/(1-3z+3z^3)$.  Agreement in every degree is strong
freeness (\cref{cor:word-counts}).  The dimensions are read off from
a Gr\"obner basis of the ideal, and once its completion terminates,
finitely many degrees decide (Appendix~\ref{app:word-counts}).
Theorem~D is proved in numbered form as \cref{thm:mainexample},
\cref{thm:mild-p3}, \cref{thm:p7}, and \cref{thm:rank-four}.  The assertion about
the $26$ non-mild fields is \cref{prop:nonmild-p3}.  Being FAb of
cohomological dimension $2$, the groups of Theorem~D are in
particular not $p$-adic analytic (\cref{cor:nonanalytic}), in
accordance with the unramified Fontaine--Mazur conjecture
\cite{FontaineMazur}.  \Cref{rem:FM} compares this with the older
route to non-analyticity through the Golod--Shafarevich property,
and \cref{sec:consequences} collects the consequences common to all
these groups.

The outcome differs by prime.  At $p=5$ the criterion decides all
fields but one.  For the exception, the norm-degeneracy scheme has
no reduced isolated closed point at all, so the criterion provably
cannot apply.  That field is decided by the direct route
(\cref{sec:example}).  At $p=7$ it decides
every field, with the same arithmetic, unchanged
(\cref{sec:example-p7}).  At $p=3$ the same chain operates with
the two changes of \cref{rem:p3}: the Ahlqvist--Carlson formula
carries an additional ideal factor, and the cubic initial forms lie in an
eleven-dimensional space of shuffle-symmetric forms,
reconstructed from four of the secondary norm operators.  The
diagonal of the Massey product becomes the Bockstein map, and
the criterion operates on its kernel, the Bockstein cone.  On
more than half of the fields the cone is empty and the
criterion cannot apply at all.  Indeed, at $p=3$ the
criterion is far from necessary, accounting for $505$ of the
$12\,243$ mild fields of rank three.  The direct route settles the
rest, one field at a time.  Its failure proves nothing, and $480$
fields of the range remain undecided.

The fields of Theorem~D appear to be the first number fields for
which the full maximal everywhere unramified pro-$p$ Galois group
is proved to be mild.  Closest to the
present $p=3$ setting is the work of Ahlqvist--Pink
\cite[\S\S7--11]{AhlqvistPink}.  They use triple Massey products
to study rank-two $3$-class tower groups, determining their
quotient by the fourth Zassenhaus subgroup.  In several cases
they obtain finiteness or $3$-adic analyticity.
Recent work on mildness and Massey products concerns abstract criteria,
absolute or tame Galois groups, or groups with nonempty ramification set
\cite{Efrat2026,FeuerpfeilHamzaLim,MaireMinacRamakrishnaTan}.  As far as
we are aware, none establishes mildness or cohomological dimension $2$
for the $p$-class tower group of any number
field.

The data are computed first and then verified by exact
arithmetic, and only the verification enters the proofs.
Computations that may depend on a generalized Riemann hypothesis
occur before the verification, so the results are unconditional
(Appendix~\ref{app:verification}).

\subsection*{Organization and notation}

The paper is organized as follows.  In \cref{sec:secondary-norms} we set up
the relevant cohomology and duality, define the secondary norm operators,
prove the reconstruction theorem, and relate the operators to the arithmetic
formula of Ahlqvist--Carlson.  \Cref{sec:relations} identifies the reconstructed
Massey product with the cubic initial relations and proves the
transversality criterion for every $p$-class rank $d\ge3$,
together with its geometric formulation through the
norm-degeneracy scheme (two general results on determinantal loci are proved in
Appendix~\ref{app:rank-one-lemmas}).  It also records the more general
direct Anick test
available when that criterion does not apply.  \Cref{sec:computation} presents the two algorithms behind the
proof, the evaluation of the secondary norms and the
transversality search, together with the exact checks.  The case
$p=3$ follows the same chain, and each section states what
changes.  The Bockstein cone in the criterion is the one object
with no counterpart at $p>3$ (\cref{def:bockstein-cone}).
\Cref{sec:results} collects the results: the $204$ fields at $p=5$
(\cref{sec:example}), the three fields at $p=7$
(\cref{sec:example-p7}), and the $12\,243$ mild $3$-class tower
groups of rank three (\cref{sec:example-p3}).  The boundary cases,
among them the field of rank four, are gathered in
\cref{subsec:exceptional-fields}.
\Cref{sec:consequences} records the consequences shared by all the
groups obtained: non-analyticity, the relation-rank formula along
the tower, and the graded growth.  Appendix~\ref{app:word-counts}
develops the word counts that decide strong freeness, with the
completion algorithm behind them.  Verification and
reproducibility information, for all three primes, is collected in
Appendix~\ref{app:verification}.

\paragraph{Fixed identifications.}
The computations use explicit class-group elements, so we fix and name
the identifications used to pass between them and cohomology:
$H^1(G,\Fp)$ with $(\Cl(K)/p)^*$ in \eqref{eq:H-classgroup},
$H^2(G,\Fp)$ with $\Cl(K)[p]^*$ in \eqref{eq:H2-classgroup}, and
$\operatorname{Ext}^1_{\Spec\mathcal O_K}(\Fp,\mathbb G_m)$
with $\Cl(K)[p]$ in \cref{prop:comparison}\ref{comp:dual}.  The comparison maps between
group and étale cohomology are those of \cref{prop:comparison}, and
\cref{lem:sign} relates the sign conventions of \cite{Vogel} and
\cite{AhlqvistCarlson}.  The secondary-norm matrices displayed in
Appendix~\ref{app:verification} use these identifications and the
class-group bases fixed there.

\begin{remark}[Notation and comparison with Ahlqvist--Carlson]\label{rem:ACnotation}
Wherever practical, we follow the notation of Ahlqvist--Carlson
\cite{AhlqvistCarlson}, with one exception: identities among
fractional ideals are written multiplicatively throughout, where
they write divisors additively.  In particular,
$L_x| K$ denotes the unramified cyclic degree-$p$ extension
associated with a character $x$, and $\sigma_x$ the generator of
$\Gal(L_x| K)$ corresponding to $x$ under Artin reciprocity.  We
write $i_x$ for extension of fractional ideals from $K$ to $L_x$,
and $N_x$ for the norm from $L_x$ to $K$, on elements or ideals
as the context dictates.
For the maximal everywhere
unramified pro-$p$ extension and its Galois group we use the
$S$-ramification notation
\[
  K_{\varnothing}(p),\qquad
  G_{\varnothing}(K)(p)=\Gal(K_{\varnothing}(p)| K),
\]
corresponding to $S=\varnothing$.  All additional notation used for the
reconstruction and mildness arguments is introduced at first use.
\end{remark}

\subsection*{Acknowledgements}

The passage from strong freeness to cohomological dimension two, on which
the main theorem rests, is John Labute's theorem.  The long genesis of the
theory of mild pro-$p$ groups, from a question of Serre to the first
pro-$p$ groups of cohomological dimension two with finite abelianized
derived factors, he has recently recounted himself \cite{LabuteGenesis}.
The present paper carries that theory into the everywhere unramified case.

The strategy pursued here has its origin in
\cite[Corollary~2.3.2]{Vogel}, where the triple Massey products of
the everywhere unramified case give the trilinear pairing on
$(\Cl(K)/p)^*$ whose interpretation was left as a question.  I
discussed the plan of combining it with a criterion for
mildness with Jochen G\"artner in the early 2010s, when his work on mild
pro-$p$ groups and higher Massey products \cite{Gaertner} supplied a
criterion of the kind needed.  I thank him for those conversations.  What
was still missing was the arithmetic.  It became available with
\cite{AhlqvistCarlson}, where the question is answered by the norm-ideal
formula, and the present paper is the result of putting the pieces
together.

I thank Eric Ahlqvist for making his \texttt{Massey-pari} program publicly
available.  The computations reported here were carried out with an extended
version of it, and having a working implementation of the class-field and
Artin-symbol arithmetic to build on saved a great deal of work.  The
arithmetic formula evaluated in these computations is that of
Ahlqvist--Carlson
\cite{AhlqvistCarlson}, and the passage to the relations of the tower group
uses the comparison of Ahlqvist--Pink \cite{AhlqvistPink}.  The
boundary between his code and the extensions written for this paper
is recorded in the \texttt{NOTICE} file of \cite{Repository}.

\paragraph{AI declaration.}  Large language model assistants (OpenAI
GPT Sol; Anthropic Claude Fable and Claude Opus) were used throughout
this work: for proof exploration, for implementing and running the
code of the accompanying repository, and in the exposition.  Their
part in the searches that produced the stored arithmetic data is
described in \cref{subsec:searches}.  The conception of this paper,
its overall strategy, and all final mathematical judgements are the
author's.  Every computational claim rests on the public verification
procedure of Appendix~\ref{app:verification} rather than on model
output, and the author assumes full responsibility for the
mathematical content.

\section{Cohomology, secondary norms, and reconstruction}
\label{sec:secondary-norms}\label{sec:setup}\label{sec:secondary}

Let $K$ be an imaginary quadratic number field and let $p$ be an odd
prime.  For $p=3$ we exclude the single field $K=\mathbb Q(\sqrt{-3})$.
Put
\[
  X=\Spec\mathcal O_K
\]
and let $K_{\varnothing}(p)$ be the maximal everywhere unramified pro-$p$ extension of
$K$.  We write
\[
  G=G_{\varnothing}(K)(p)=\Gal(K_{\varnothing}(p)| K).
\]
Throughout this section, $H^i(G,\Fp)$ denotes the cohomology of the
profinite group $G$ with coefficients in the discrete trivial $G$-module
$\Fp$.  Thus the cochain complex consists of continuous cochains.  In
particular,
\[
  H^1(G,\Fp)=\Hom_{\mathrm{cont}}(G,\Fp).
\]
For a finite-dimensional vector space $V$ over a field $k$ we write
\[
  V^*=\Hom(V,k)
\]
for its linear dual, and we write the evaluation
pairing as $V^*\times V\to k$.  The canonical map $V\to V^{**}$ is
an isomorphism, and we use it to identify $V$ with $V^{**}$.  For $\Fp$-vector spaces this agrees
with the dual $\Hom(V,\mathbb Q/\mathbb Z)$ of \cite{NSW} and
\cite{Vogel} under the identification
$(\mathbb Q/\mathbb Z)[p]\simeq\Fp$,
$\tfrac{a}{p}\bmod\mathbb Z\mapsto a\bmod p$, which we fix once and
for all.  It will be used again only to normalize the Artin--Verdier
pairing in \eqref{eq:AV-etale}.
Class field theory identifies $H^1(G,\Fp)$ with
$(\Cl(K)/p)^*$.  We fix the identification.  The Artin map of the maximal unramified
abelian $p$-extension of $K$ induces an isomorphism
\[
  \operatorname{Artin}_K\colon\Cl(K)/p\xrightarrow{\ \sim\ }
  G^{\mathrm{ab}}/p,
\]
and every $x\in H^1(G,\Fp)=\Hom_{\mathrm{cont}}(G,\Fp)$ factors
through $G^{\mathrm{ab}}/p$.  Composition with $\operatorname{Artin}_K$
therefore gives an isomorphism
\begin{equation}\label{eq:H-classgroup}
  H^1(G,\Fp)\xrightarrow{\ \sim\ }\bigl(\Cl(K)/p\bigr)^*,
  \qquad x\longmapsto x\circ\operatorname{Artin}_K.
\end{equation}
Throughout, we use \eqref{eq:H-classgroup} to read a class
$y\in H^1(G,\Fp)$ as a linear form on $\Cl(K)/p$, and we write
$y(c)$ for its value at $c\in\Cl(K)/p$.  Thus $y(c)$ means
$y(\operatorname{Artin}_K(c))$, the value of the character $y$ on the
Artin image of $c$.

The same argument, applied over every finite subextension of the
tower, shows that every open subgroup of $G$ has finite
abelianization, since the abelianization is the $p$-primary part of a
class group.  This FAb property holds for every number field and
every prime and is used in \cref{sec:consequences}.

\begin{proposition}\label{prop:FAb}
For every number field $K$ and every prime $p$, with
$K_{\varnothing}(p)| K$ unramified at all places, the archimedean
ones included, the group $G_{\varnothing}(K)(p)$ is FAb: every open
subgroup has finite abelianization.
\end{proposition}

\begin{proof}
Let $U\subseteq G$ be open and let $L=(K_{\varnothing}(p))^U$.  Then $L| K$ is finite
and everywhere unramified.  The quotient $U^{\ab}$ is the Galois group of the
maximal abelian pro-$p$ subextension of $K_{\varnothing}(p)| L$.  Every finite
subextension of this abelian extension is everywhere unramified over $L$.
By global class field theory it is therefore contained in the Hilbert
$p$-class field of $L$, whose Galois group is the finite $p$-primary part of
$\Cl(L)$.  Hence $U^{\ab}$ is finite.
\end{proof}

\subsection{Comparison with étale cohomology and Artin--Verdier duality}
\label{subsec:comparison}

The secondary norm formula of Ahlqvist--Carlson is formulated in the étale
cohomology of $X$, whereas the defining-relation arguments later in the paper
use the group cohomology of the profinite Galois group $G$.  We therefore
recall the comparison between these two cohomology theories.  The maximal
unramified pro-$p$ cover of $X$ gives natural maps
\begin{equation}\label{eq:kappa}
  \kappa_i\colon H^i(G,\Fp)\longrightarrow H^i(X,\Fp).
\end{equation}
Ahlqvist--Pink
recall that $\kappa_1$ is an isomorphism and $\kappa_2$ is
injective, and that Bockstein, cup products, and Massey products are
compatible with these maps
\cite[\S7, especially (7.3)]{AhlqvistPink}.  The comparison is
developed in the present setting in Ahlqvist--Carlson
\cite[\S3]{AhlqvistCarlson}, and we cite these two papers because
they collect the statements in the form used below.  For an imaginary
quadratic number field, Ahlqvist--Pink prove that $\kappa_i$ is an
isomorphism for
$i=1,2$, except for the exceptional case $p=3$ and
$K=\mathbb Q(\sqrt{-3})$ \cite[Proposition~7.8]{AhlqvistPink}.  Outside
the excluded case, we henceforth identify the two cohomologies in
degrees one and two.

Artin--Verdier duality pairs $H^2(X,\Fp)$ perfectly with
$H^1(X,D(\Fp))$, with values in
$H^3(X,\mathbb G_m)\simeq\mathbb Q/\mathbb Z$.  Both factors are
killed by $p$, so the values lie in
$(\mathbb Q/\mathbb Z)[p]$, and the identification
$\tfrac{a}{p}\bmod\mathbb Z\mapsto a\bmod p$ fixed
above turns this into a perfect pairing
\begin{equation}\label{eq:AV-etale}
  H^2(X,\Fp)\times H^1\bigl(X,D(\Fp)\bigr)\longrightarrow\Fp,
  \qquad
  D(\Fp)=R\mathcal Hom_X(\Fp,\mathbb G_m).
\end{equation}

In the notation of Ahlqvist--Carlson, the classes of
$H^1(X,D(\Fp))$ are represented by pairs of an element and a
fractional ideal of $K$,
\begin{equation}\label{eq:pair-aJ}
  (a',J),\qquad
  a'\in K^\times,\quad
  (a')\,J^{\,p}=\mathcal O_K,
\end{equation}
and every class arises from such a pair, since
$H^1(X,D(\Fp))=\operatorname{Ext}^1_X(\Fp,\mathbb G_m)$ by the
definition of $D(\Fp)$, and these
$\operatorname{Ext}$-groups are computed in
\cite[Corollary~2.15]{AhlqvistCarlsonRing}.
These pairs sit in an exact sequence
\[
  0\longrightarrow\mathcal O_K^\times/p
  \longrightarrow H^1\bigl(X,D(\Fp)\bigr)
  \longrightarrow\Cl(K)[p]\longrightarrow0
  \qquad\text{\cite[\S4.1]{AhlqvistCarlson}},
\]
whose second map sends the class of a pair $(a',J)$ to $[J]$.
For an imaginary quadratic number field the unit group is finite of
order $2$, $4$, or $6$, so for odd $p$ the first term vanishes except
precisely in the excluded case ($\mathcal O_K^\times$ has order
divisible by an odd prime only for $K=\mathbb Q(\sqrt{-3})$); outside
it, that map is an isomorphism,
\[
  H^1\bigl(X,D(\Fp)\bigr)\xrightarrow{\ \sim\ }\Cl(K)[p],
  \qquad (a',J)\longmapsto[J].
\]
We represent elements of $\Cl(K)[p]$ by pairs accordingly, writing
\begin{equation}\label{eq:e-is-J}
  e=[J]\in\Cl(K)[p]
\end{equation}
for the class represented by $(a',J)$.  Since $J^{\,p}$ is the
principal ideal $(a')^{-1}$, $p[J]=0$.
Under these identifications, the Artin--Verdier pairing becomes a perfect
pairing of $\Fp$-vector spaces
\begin{equation}\label{eq:AVpair}
  H^2(G,\Fp)\times \Cl(K)[p]\longrightarrow\Fp,
  \qquad (\alpha,e)\longmapsto \langle\alpha,e\rangle_{\AV}.
\end{equation}
Equivalently, taking adjoints of \eqref{eq:AVpair},
\begin{equation}\label{eq:H2-classgroup}
  H^2(G,\Fp)\xrightarrow{\ \sim\ }\Cl(K)[p]^*,
  \qquad
  \alpha\longmapsto\langle\alpha,\cdot\,\rangle_{\AV}.
\end{equation}
We use \eqref{eq:H-classgroup} and \eqref{eq:H2-classgroup} with
these normalizations throughout.  They come from different sources,
class field theory and Artin--Verdier duality.

In particular, if
\[
  d_p\Cl(K):=\dim_{\Fp}\Cl(K)/p
\]
denotes the $p$-class rank, then
\begin{equation}\label{eq:ranks}
  \dim_{\Fp}H^1(G,\Fp)=\dim_{\Fp}H^2(G,\Fp)=d_p\Cl(K).
\end{equation}
Equivalently, the generator and relation ranks of $G$ agree in the present
setting.

The quotient $\Cl(K)/p$ and the subgroup $\Cl(K)[p]$ play different
roles.  By \eqref{eq:H-classgroup}, degree-one characters are linear forms
on $\Cl(K)/p$, whereas the Artin--Verdier pairing \eqref{eq:AVpair}
evaluates degree-two cohomology classes on $\Cl(K)[p]$.  We will later
choose bases of these two $\Fp$-vector spaces arising from the same
invariant-factor decomposition of $\Cl(K)$, but there is no canonical
identification between them.

We record the properties of the comparison used in the sequel.

\begin{proposition}[Comparison package]\label{prop:comparison}
\begin{enumerate}[label=\textup{(\roman*)}]
\item\label{comp:iso} The maps $\kappa_1$ and $\kappa_2$ of
\eqref{eq:kappa} are isomorphisms.
\item\label{comp:products} They are compatible with Bockstein maps,
cup products, and triple Massey products together with their
defining systems.
\item\label{comp:dual}
$H^1(X,D(\Fp))=\operatorname{Ext}^1_X(\Fp,\mathbb G_m)
\simeq\Cl(K)[p]$, induced by $(a',J)\mapsto[J]$.
\item\label{comp:pairing} Under \ref{comp:iso} and \ref{comp:dual},
Artin--Verdier duality \eqref{eq:AV-etale} becomes the perfect
pairing \eqref{eq:AVpair}.  Equivalently, \eqref{eq:H2-classgroup}
holds.
\item\label{comp:equivariant} The identifications in
\ref{comp:iso}--\ref{comp:pairing}, as well as the Artin
identification \eqref{eq:H-classgroup}, are functorial and
equivariant for the action of complex conjugation.
\end{enumerate}
\end{proposition}

\begin{proof}
Parts \ref{comp:iso} and \ref{comp:products} are
\cite[\S7, especially (7.3) and Proposition~7.8]{AhlqvistPink},
with the comparison developed in \cite[\S3]{AhlqvistCarlson}.
Part \ref{comp:dual} is the description of
\cite[\S\S4--5]{AhlqvistCarlson} recalled above, using the
vanishing of the unit term.  Part \ref{comp:pairing} restates
\eqref{eq:AVpair} and \eqref{eq:H2-classgroup}.  Part
\ref{comp:equivariant} holds because every map involved is defined
functorially, in particular compatibly with the automorphisms
induced by complex conjugation on $G$, on $X$, and on $\Cl(K)$.
\end{proof}

We fix once and for all which cohomology carries which
statement.  Three statements about $G$ live in group cohomology and
are quoted there: the Koch--Venkov vanishing of the
cup product (\cref{lem:cupzero} below), the presentation theory of
\cref{sec:relations}, and the relation formula of Proposition~1.3.3 of
\cite{Vogel}.  The
explicit Massey formulas of Ahlqvist--Carlson and Artin--Verdier
duality live on $X$.  Under the identifications of
\cref{prop:comparison}, exactly two statements cross between the
sides: the vanishing of the cup product passes from $G$ to $X$,
where it makes the Ahlqvist--Carlson formula applicable
(\cref{prop:AC}), and the Artin--Verdier evaluations pass back
from $X$ to $G$, where they define and compute the secondary norm
operators (\cref{def:D}).

\subsection{Vanishing of cup products}

\begin{lemma}\label{lem:cupzero}
\begin{enumerate}[label=\textup{(\roman*)}]
\item\label{cupzero:group} The cup product
\[
  H^1(G,\Fp)\times H^1(G,\Fp)\longrightarrow H^2(G,\Fp)
\]
is identically zero.  Hence every triple Massey product of classes in
$H^1(G,\Fp)$ is uniquely defined and the map
\[
\begin{aligned}
  M\colon H^1(G,\Fp)\times H^1(G,\Fp)\times H^1(G,\Fp)
  &\longrightarrow H^2(G,\Fp),\\
  (x,y,z)&\longmapsto\langle x,y,z\rangle,
\end{aligned}
\]
is trilinear.  We refer to the map $M$ itself as the \emph{triple
Massey product}.
\item\label{cupzero:etale} The cup product
\[
  H^1(X,\Fp)\times H^1(X,\Fp)\longrightarrow H^2(X,\Fp)
\]
is likewise identically zero.
\end{enumerate}
\end{lemma}

\begin{proof}
Part \ref{cupzero:group} is the complex-conjugation
argument of Koch--Venkov \cite{VenkovKoch}, who used it to force
the relations of the tower group into Zassenhaus depth at least
three.  Complex conjugation acts by inversion on the ideal class
group.  On $H^1(G,\Fp)\simeq(\Cl(K)/p)^*$
\eqref{eq:H-classgroup} it therefore acts by $-1$, by class field
theory alone.  To argue on $H^2(G,\Fp)$ one needs some description
of that group.  The description used
throughout this paper is \eqref{eq:H2-classgroup}, which is equivariant by
\cref{prop:comparison}\ref{comp:equivariant}, so conjugation acts
by $-1$ on $H^2(G,\Fp)\simeq\Cl(K)[p]^*$ as well (any
functorial description would serve).  Equivariance of the
cup product now makes a product of two degree-one classes
invariant, while the action on $H^2(G,\Fp)$ is by $-1$.  Since $p$
is odd, the product must vanish.  With all cup products zero,
\cite[Lemma~1.2.5]{Vogel}, applied with $m=3$, shows that every
triple Massey product of classes in $H^1(G,\Fp)$ is uniquely
defined and that the map $M$ is trilinear.
Part \ref{cupzero:etale} follows from part \ref{cupzero:group}
through the comparison isomorphisms of \cref{prop:comparison}.
\end{proof}

We shall use the shuffle identities
\begin{equation}\label{eq:shuffle}
  M(x,y,z)=M(z,y,x),\qquad
  M(x,y,z)+M(y,z,x)+M(z,x,y)=0.
\end{equation}
In group cohomology these are the degree-three case of the shuffle relations
in \cite[Corollary~1.2.10 and Example~1.2.11(b)]{Vogel}.  By
\cref{prop:comparison}\ref{comp:products} the same identities hold for the
étale Massey products used in the arithmetic formulas.  Taking all three
arguments equal in the cyclic identity gives
\begin{equation}\label{eq:diagonal-triple}
  3\,M(x,x,x)=0
  \qquad (x\in H^1(G,\Fp)).
\end{equation}
Hence for $p>3$ the diagonal $x\mapsto M(x,x,x)$ vanishes identically,
while at $p=3$ it need not.  The diagonal at $p=3$ is determined in
\cref{lem:diagonal-p3}.

Two cochain conventions for the triple Massey product occur in the
sources used here: that of \cite{Vogel}, in which the relation
coefficients of \cref{sec:relations} are computed, and that of
Ahlqvist--Carlson, in which the arithmetic evaluations of
\cref{prop:AC} take place.  Write $M^{\mathrm V}$ and
$M^{\mathrm{AC}}$ for the two products after the identification of
group and étale cohomology by the maps \eqref{eq:kappa}.

\begin{lemma}[The global Massey sign]\label{lem:sign}
We have
\begin{equation}\label{eq:global-sign}
  M^{\mathrm{AC}}=-M^{\mathrm V}
\end{equation}
as maps $H^1(G,\Fp)\times H^1(G,\Fp)\times H^1(G,\Fp)\to
H^2(G,\Fp)$.
\end{lemma}

\begin{proof}
Let $x,y,z\in H^1(G,\Fp)$.  The defining sets of
\cite[Definition~1.2.1]{Vogel} use cochains $a_{12},a_{23}$ with
\[
  \partial a_{12}=x\smile y,
  \qquad
  \partial a_{23}=y\smile z,
\]
the signs being fixed there so that $\langle u_1,u_2\rangle=u_1\smile
u_2$ in the twofold case, and represent the product by
$x\smile a_{23}+a_{12}\smile z$.  Ahlqvist--Carlson \cite[\S2]{AhlqvistCarlson}, following
\cite{Dwyer}, choose $k_{xy},k_{yz}$ with
\[
  dk_{xy}=-x\smile y,
  \qquad
  dk_{yz}=-y\smile z,
\]
and use the same formal ordering
$x\smile k_{yz}+k_{xy}\smile z$.
Taking $k_{xy}=-a_{12}$ and $k_{yz}=-a_{23}$ negates the
representative.  Since all cup products vanish, the Massey
indeterminacy is zero.  Thus the sign is global, with no
permutation-dependent factor and no exception on the diagonal.
\end{proof}

Throughout, $M$ denotes $M^{\mathrm{AC}}$.  It is the product that
the arithmetic formulas of \cref{prop:AC,prop:AC-p3} evaluate.

\subsection{Secondary norm operators}

Define
\[
  Q\colon H^1(G,\Fp)\times H^1(G,\Fp)\longrightarrow H^2(G,\Fp),
  \qquad
  (x,y)\longmapsto M(x,x,y).
\]

\begin{lemma}\label{lem:Q}
\begin{enumerate}[label=\textup{(\roman*)}]
\item\label{Q:linear} For fixed $x$, the map
$H^1(G,\Fp)\to H^2(G,\Fp)$, $y\mapsto Q(x,y)$, is linear.
\item\label{Q:quadratic} $Q(0,y)=0$ and
$Q(\lambda x,y)=\lambda^2\,Q(x,y)$ for every $\lambda\in\Fp$.
\item\label{Q:polarization} For all $x,y,z\in H^1(G,\Fp)$,
\begin{equation}\label{eq:polarization}
  M(x,y,z)=Q(x,y)+Q(z,y)-Q(x+z,y).
\end{equation}
\end{enumerate}
\end{lemma}

\begin{proof}
Parts \ref{Q:linear} and \ref{Q:quadratic} are immediate from
trilinearity of $M$ (\cref{lem:cupzero}\ref{cupzero:group}).  For
\ref{Q:polarization}, trilinearity gives
\[
  Q(x+z,y)=M(x+z,x+z,y)
  =Q(x,y)+M(x,z,y)+M(z,x,y)+Q(z,y).
\]
By the outer symmetry in \eqref{eq:shuffle}, $M(x,z,y)=M(y,z,x)$,
and the cyclic relation in \eqref{eq:shuffle} gives
$M(y,z,x)+M(z,x,y)=-M(x,y,z)$.  Substituting proves
\ref{Q:polarization}.
\end{proof}

For fixed $x\in H^1(G,\Fp)$ and $e\in\Cl(K)[p]$, the map
\[
  H^1(G,\Fp)\longrightarrow\Fp,\qquad
  y\longmapsto\langle Q(x,y),e\rangle_{\AV},
\]
where $\langle\cdot,\cdot\rangle_{\AV}$ is the Artin--Verdier
pairing \eqref{eq:AVpair}, is linear by \cref{lem:Q}\ref{Q:linear}.
Under the identification $H^1(G,\Fp)\simeq(\Cl(K)/p)^*$ of
\eqref{eq:H-classgroup}, it is therefore an element of the bidual
$(\Cl(K)/p)^{**}=\Cl(K)/p$, that is, evaluation at a unique
element of $\Cl(K)/p$.

\begin{definition}[Secondary norm operator]\label{def:D}
For $x\in H^1(G,\Fp)$, define
\[
  D_x\colon \Cl(K)[p]\longrightarrow \Cl(K)/p
\]
as the map sending $e$ to the unique element
$D_x(e)\in\Cl(K)/p$ with
\begin{equation}\label{eq:Ddef}
  y(D_x(e))
  =\langle Q(x,y),e\rangle_{\AV}
  \qquad\text{for all } y\in H^1(G,\Fp).
\end{equation}
\end{definition}

Since the right-hand side of \eqref{eq:Ddef} is also linear in $e$
(the Artin--Verdier pairing is bilinear), the map $D_x$ is
$\Fp$-linear.  The name \emph{secondary norm operator} will be
justified by the Ahlqvist--Carlson formula in \cref{prop:AC}.

For $x\in H^1(G,\Fp)$, viewed via \eqref{eq:H-classgroup} as a
linear form on $\Cl(K)/p$, its \emph{annihilator} under the
evaluation pairing is
\[
  x^\perp:=\{c\in\Cl(K)/p:x(c)=0\},
\]
an $\Fp$-subspace of $\Cl(K)/p$, of codimension one for $x\ne0$.

\begin{lemma}\label{lem:D-image-orthogonal}
Let $p>3$.  For every $x\in H^1(G,\Fp)$,
\[
  \operatorname{im}D_x\subseteq x^\perp.
\]
\end{lemma}

\begin{proof}
For $e\in\Cl(K)[p]$, the definition of $D_x$ gives
\[
  x(D_x(e))=\langle M(x,x,x),e\rangle_{\AV}.
\]
Since $p>3$, the diagonal vanishes by \eqref{eq:diagonal-triple},
and hence $x(D_x(e))=0$.
\end{proof}

At $p=3$ only the last step fails.  The identity
$x(D_x(e))=\langle M(x,x,x),e\rangle_{\AV}$ holds for every odd $p$,
so, the Artin--Verdier pairing being perfect,
$\operatorname{im}D_x\subseteq x^\perp$ holds precisely for those
characters $x$ with $M(x,x,x)=0$.  Which characters these are is
determined in \cref{lem:diagonal-p3}, stated below after the
arithmetic formulas that its proof uses.

\subsection{Polarization and reconstruction}

\begin{theorem}[Reconstruction theorem]\label{thm:reconstruction}
The map
\[
  D\colon H^1(G,\Fp)\longrightarrow
  \Hom_{\Fp}\bigl(\Cl(K)[p],\Cl(K)/p\bigr),\qquad x\longmapsto D_x,
\]
is quadratic.  More precisely,
\begin{equation}\label{eq:quadratic}
  D_0=0,\qquad D_{\lambda x}=\lambda^2D_x
  \quad(\lambda\in\Fp).
\end{equation}
Define
\begin{equation}\label{eq:DeltaD}
\begin{aligned}
  \Delta D\colon H^1(G,\Fp)\times H^1(G,\Fp)&\longrightarrow
  \Hom_{\Fp}\bigl(\Cl(K)[p],\Cl(K)/p\bigr),\\
  (x,z)&\longmapsto D_x+D_z-D_{x+z}.
\end{aligned}
\end{equation}
Then
\begin{equation}\label{eq:MfromD}
  \langle M(x,y,z),e\rangle_{\AV}
  =y\bigl(\Delta D(x,z)(e)\bigr)
\end{equation}
for all $x,y,z\in H^1(G,\Fp)$ and $e\in\Cl(K)[p]$.  In particular,
the map $\Delta D$ is $\Fp$-bilinear.  We call it the
\emph{polarization} of $D$.

Consequently, if $\chi_1,\ldots,\chi_d$ is a basis of $H^1(G,\Fp)$, then the
\[
  \frac{d(d+1)}2
\]
operators
\begin{equation}\label{eq:minimal-family}
  D_{\chi_i},\qquad D_{\chi_i+\chi_j}\quad(1\le i<j\le d)
\end{equation}
determine $D_x$ for every $x\in H^1(G,\Fp)$ and therefore determine the complete
triple Massey product $M$.
\end{theorem}

\begin{proof}
For \eqref{eq:quadratic}, let $x\in H^1(G,\Fp)$,
$\lambda\in\Fp$, and $e\in\Cl(K)[p]$.  For every
$y\in H^1(G,\Fp)$,
\[
  y\bigl(D_{\lambda x}(e)\bigr)
  =\langle Q(\lambda x,y),e\rangle_{\AV}
  =\lambda^2\langle Q(x,y),e\rangle_{\AV}
  =\lambda^2\,y\bigl(D_x(e)\bigr)
  =y\bigl(\lambda^2D_x(e)\bigr),
\]
using \eqref{eq:Ddef}, then \cref{lem:Q}\ref{Q:quadratic}, then
\eqref{eq:Ddef} again, and the linearity of $y$.  Thus
$\lambda^2D_x(e)$ satisfies the property characterizing
$D_{\lambda x}(e)$, and the uniqueness in \eqref{eq:Ddef} gives
$D_{\lambda x}=\lambda^2D_x$.  The same argument with $Q(0,y)=0$
gives $D_0=0$.

For \eqref{eq:MfromD}, evaluate both sides of the polarization
identity \eqref{eq:polarization} against $e$ under the
Artin--Verdier pairing.  By linearity of
$\langle\cdot,e\rangle_{\AV}$ in the first argument,
\eqref{eq:Ddef} applied to each term, and linearity of $y$,
\[
\begin{aligned}
  \langle M(x,y,z),e\rangle_{\AV}
  &=\langle Q(x,y),e\rangle_{\AV}
   +\langle Q(z,y),e\rangle_{\AV}
   -\langle Q(x+z,y),e\rangle_{\AV}\\
  &=y\bigl(D_x(e)\bigr)+y\bigl(D_z(e)\bigr)
   -y\bigl(D_{x+z}(e)\bigr)
  =y\bigl(\Delta D(x,z)(e)\bigr),
\end{aligned}
\]
which is \eqref{eq:MfromD}.  The same formula shows that $\Delta D$
is bilinear, since the left-hand side is trilinear in $x$, $y$, $z$ and
the evaluation pairings separate the points of
$\Hom_{\Fp}\bigl(\Cl(K)[p],\Cl(K)/p\bigr)$.

Finally, write $x=\sum_i x_i\chi_i$.  Bilinearity of
$\Delta D$ gives
\[
  D_{u+v}=D_u+D_v-\Delta D(u,v),
\]
and repeated application of this identity, together with
$D_{x_i\chi_i}=x_i^2D_{\chi_i}$, yields
\begin{equation}\label{eq:D-reconstruct}
  D_x=\sum_i x_i^2D_{\chi_i}
   +\sum_{i<j}x_ix_j
    \bigl(D_{\chi_i+\chi_j}-D_{\chi_i}-D_{\chi_j}\bigr).
\end{equation}
Indeed, the term in parentheses is $-\Delta D(\chi_i,\chi_j)$, so the displayed
formula is the decomposition of the quadratic map into its diagonal
and mixed terms.  Thus the family \eqref{eq:minimal-family} determines all
$D_x$, and then \eqref{eq:MfromD} determines $M$.
\end{proof}

\begin{remark}[Arithmetic cost]\label{rem:arithmetic-cost}
The polarization is a formal identity, but it has an arithmetic
consequence.  By the Ahlqvist--Carlson formula of \cref{prop:AC}
below, each operator $D_x$ is computed from a single unramified
cyclic degree-$p$ extension of $K$, and by \cref{thm:reconstruction}
already $d(d+1)/2$ of these operators determine the whole Massey
product, where $d=d_p\Cl(K)$.  The number of class-field
computations therefore grows quadratically in the $p$-class rank,
whereas $M$ has $d^3$ coordinates against each basis element of
$\Cl(K)[p]$.
\end{remark}

\begin{remark}\label{rem:six-computations}
Assume $d_p\Cl(K)=3$ throughout this remark, and fix bases
$\chi_1,\chi_2,\chi_3$ of $H^1(G,\Fp)$ and $e_1,e_2,e_3$ of
$\Cl(K)[p]$.  The six operators \eqref{eq:minimal-family} determine
the $27$ coordinates
\[
  \langle M(\chi_i,\chi_j,\chi_k),e_\ell\rangle_{\AV}
  \qquad(1\le i,j,k\le3)
\]
of the Massey product against each $e_\ell$.  By \eqref{eq:MfromD}
the coordinate equals
$\chi_j\bigl(\Delta D(\chi_i,\chi_k)(e_\ell)\bigr)$.  For
$i\ne k$ the polarization \eqref{eq:DeltaD} uses $D_{\chi_i}$,
$D_{\chi_k}$ and $D_{\chi_i+\chi_k}$, and for $i=k$ the quadratic
dependence \eqref{eq:quadratic} gives
\[
  \Delta D(\chi_i,\chi_i)
   =2D_{\chi_i}-D_{2\chi_i}
   =-2D_{\chi_i}.
\]
Six class-field computations therefore suffice.

How the coordinates with $i=j=k$ behave depends on the prime.  Taking
$y=x=\chi_i$ in \eqref{eq:Ddef} identifies them as the evaluations of
$D_{\chi_i}$ against its own character,
\[
  \langle M(\chi_i,\chi_i,\chi_i),e_\ell\rangle_{\AV}
  =\chi_i\bigl(D_{\chi_i}(e_\ell)\bigr).
\]
For $p>3$ these coordinates vanish by \cref{lem:D-image-orthogonal};
at $p=3$ they are the entries of the Bockstein matrix of
\cref{lem:diagonal-p3}.  At $p=3$ the count also improves.  The four
operators
$D_{\chi_1},D_{\chi_2},D_{\chi_3},D_{\chi_1+\chi_2+\chi_3}$
already determine all $27$ coordinates, and no three characters
suffice (\cref{lem:four-evaluations}).
\end{remark}

We make the scalar extension of the secondary norm family explicit.  Define
\[
  \begin{aligned}
  b\colon H^1(G,\Fp)\times H^1(G,\Fp)
    &\longrightarrow
      \Hom_{\Fp}\!\left(\Cl(K)[p],\Cl(K)/p\right),\\
  (x,z)&\longmapsto
      D_{x+z}-D_x-D_z=-\Delta D(x,z).
  \end{aligned}
\]
By \cref{thm:reconstruction}, $b$ is symmetric and bilinear.  Moreover,
$b(x,x)=2D_x$, because $D_{2x}=4D_x$.  For a finite extension
$k|\Fp$, extend $b$ by scalars to a $k$-bilinear map $b_k$ and define
\begin{equation}\label{eq:scalar-extended-D}
  D_{k,x}:=\frac12 b_k(x,x)
  \in
  \Hom_k\!\left(
    \Cl(K)[p]\otimes_{\Fp}k,\,
    \Cl(K)/p\otimes_{\Fp}k
  \right)
  \qquad
  \left(x\in H^1(G,\Fp)\otimes_{\Fp}k\right).
\end{equation}
In coordinates, this is the reconstruction formula
\eqref{eq:D-reconstruct}, now with coefficients in $k$.
This is the unique quadratic scalar extension of the original family.  Since
$D$ is quadratic rather than linear, this construction, not a
tensor product of the map itself, is what we mean by extending $D$ by
scalars.  We suppress $k$ from the notation and write $D_x$ also after
scalar extension.  The quadratic and polarization identities and \eqref{eq:MfromD}
remain valid over $k$ for every odd $p$, since the underlying shuffle
identities are multilinear.  For $p>3$ the conclusion of
\cref{lem:D-image-orthogonal} also remains valid over $k$.  What
happens to the diagonal at $p=3$ is recorded after
\cref{lem:diagonal-p3}.

\subsection{Arithmetic realization via Ahlqvist--Carlson}

It remains to explain why the operators occurring in the
reconstruction theorem are arithmetically computable.  Let
$0\ne x\in H^1(G,\Fp)$.  The character $x$ determines an unramified cyclic extension
\[
  L_x| K
\]
of degree $p$, namely the fixed field of $\ker x$.  It therefore
depends only on the subgroup generated by $x$ and is
unchanged when $x$ is replaced by $\lambda x$ with
$\lambda\in\Fp^\times$.  The operator $D_x$ does depend on $x$
itself, by \eqref{eq:quadratic}, and so does the formula of
\cref{prop:AC} that computes it, since it refers to a generator of
$\Gal(L_x| K)$ normalized by $x$.  Let
\[
  \operatorname{Artin}_{L_x| K}\colon\Cl(K)\longrightarrow
  \Gal(L_x| K)
\]
denote the composite of the map $\operatorname{Artin}_K$ of
\eqref{eq:H-classgroup} with the projection
$G^{\mathrm{ab}}/p\twoheadrightarrow\Gal(L_x| K)$.  It is the Artin
map of the everywhere unramified extension $L_x| K$ and sends the
class of a prime ideal $\mathfrak p$ of $K$ to its Frobenius.  In
particular
\begin{equation}\label{eq:artin-compat}
  x\bigl(\operatorname{Artin}_{L_x| K}(c)\bigr)=x(\overline c)
  \qquad(c\in\Cl(K)),
\end{equation}
where on the left $x$ is read on $\Gal(L_x| K)=G/\ker x$ and on the
right as a linear form on $\Cl(K)/p$.

Since $x$ vanishes on $\ker x$ and is injective on
$\Gal(L_x| K)=G/\ker x$, it induces an isomorphism
$\Gal(L_x| K)\xrightarrow{\ \sim\ }\Fp$.  We write
\begin{equation}\label{eq:sigma-dual}
  \sigma_x\in\Gal(L_x| K),\qquad x(\sigma_x)=1,
\end{equation}
for the preimage of $1$, the generator of $\Gal(L_x| K)$ dual to
$x$.  Two consequences are used below.

\begin{lemma}\label{lem:sigma-x}
Let $0\ne x\in H^1(G,\Fp)$ and let $\sigma_x$ be as in
\eqref{eq:sigma-dual}.  Then
\begin{equation}\label{eq:sigma-character}
  \operatorname{Artin}_{L_x| K}(c)
   =\sigma_x^{\,x(\overline c)}
  \qquad\text{for all } c\in\Cl(K),
\end{equation}
where $\overline c$ is the image of $c$ in $\Cl(K)/p$ and the exponent
$x(\overline c)\in\Fp$ is read modulo $p$, and
\begin{equation}\label{eq:sigma-scaling}
  \sigma_{\lambda x}=\sigma_x^{\,\lambda^{-1}}
  \qquad(\lambda\in\Fp^\times).
\end{equation}
\end{lemma}

\begin{proof}
For $c\in\Cl(K)$,
\[
  x\bigl(\operatorname{Artin}_{L_x| K}(c)\bigr)
  \overset{\eqref{eq:artin-compat}}{=}x(\overline c)
  =x\bigl(\sigma_x^{\,x(\overline c)}\bigr),
\]
and injectivity of $x$ on $\Gal(L_x| K)$ gives
\eqref{eq:sigma-character}.  For \eqref{eq:sigma-scaling}, note that
$(\lambda x)(\sigma_x^{\,\lambda^{-1}})
 =\lambda\lambda^{-1}x(\sigma_x)=1$, so $\sigma_x^{\,\lambda^{-1}}$
is the generator dual to $\lambda x$.
\end{proof}

This is the convention of Ahlqvist--Pink, who represent a nonzero
class of $H^1(X,\Fp)$ by a pair $(L_x,\rho_x)$ with $\rho_x$ a generator of
$\Gal(L_x| K)$ and read the character through the isomorphism
$\Gal(L_x| K)\to\Fp$ sending $\rho_x$ to $1$
\cite[\S7]{AhlqvistPink}.  Thus $\rho_x=\sigma_x$.

Let $i_x$ denote extension of fractional ideals,
$J\mapsto J\mathcal O_{L_x}$, let $N_x$ denote the norm
$L_x\to K$, and let the group ring act on fractional ideals
through the Galois action, so that
\[
  I^{(1-\sigma_x)^2}
  =I\cdot\bigl(I^{\sigma_x}\bigr)^{-2}\cdot I^{\sigma_x^2}.
\]

\begin{definition}[Norm witness]\label{def:norm-witness}
Let $0\ne x\in H^1(G,\Fp)$, with $\sigma_x$ the generator
\eqref{eq:sigma-dual} dual to $x$, and let $(a',J)$ be a pair as in
\eqref{eq:pair-aJ}.  A \emph{norm witness} for $x$ and $(a',J)$ is
a pair $(t,I')$ of an element $t\in L_x^\times$ and a
fractional ideal $I'$ of $L_x$ with
\begin{equation}\label{eq:AC1}
  I'^{\,(1-\sigma_x)^2}\,(t)\,i_x(J)=\mathcal O_{L_x}
\end{equation}
and
\begin{equation}\label{eq:AC2}
  N_x(t)=a'.
\end{equation}
\end{definition}

\begin{proposition}[Ahlqvist--Carlson formula in operator form]\label{prop:AC}
Let $p>3$ and $0\ne x\in H^1(G,\Fp)$, and let $e\in\Cl(K)[p]$.
Then for every pair $(a',J)$ representing $e$ as in
\eqref{eq:pair-aJ} a norm witness for $x$ and $(a',J)$ exists, and
for every such pair and witness
\begin{equation}\label{eq:ACD}
  D_x(e)=[N_x(I')]\in\Cl(K)/p.
\end{equation}
\end{proposition}

\begin{proof}
We apply Ahlqvist--Carlson's Theorem~4.12.  Through the isomorphism
$\kappa_1$ of \cref{prop:comparison}\ref{comp:iso} we regard
$x,y\in H^1(G,\Fp)$ as classes in $H^1(X,\Fp)$.  By
\cref{prop:comparison}\ref{comp:products} the triple Massey product
appearing in their theorem is the one used in the definition of $D_x$.

We first check the hypotheses of their theorem.  Let
\[
  c_x\colon H^1(X,\Fp)\longrightarrow H^2(X,\Fp),
  \qquad u\longmapsto x\smile u,
\]
and define $c_y$ similarly.  Ahlqvist--Carlson denote by
$\widetilde c_x$ and $\widetilde c_y$ the maps dual to these cup-product
maps under Artin--Verdier duality.  Their Theorem~4.12 requires
\[
  x\smile y=0
\]
and requires the Artin--Verdier dual class represented by $(a',J)$ to lie
in
\[
  \ker(\widetilde c_x)\cap\ker(\widetilde c_y).
\]
By \cref{lem:cupzero}\ref{cupzero:etale} the cup product vanishes
on $H^1(X,\Fp)$.  Hence $c_x=c_y=0$, and their
Artin--Verdier dual maps $\widetilde c_x$ and $\widetilde c_y$ are
zero as well.  Thus the kernel conditions impose no further restriction.
Every class $e\in\Cl(K)[p]$, represented on the Artin--Verdier dual side by
a pair $(a',J)$, satisfies the cohomological hypotheses of
Theorem~4.12.

The same hypotheses yield the existence clause.  Their theorem is
stated for a given witness, but its proof constructs one.  The pair
$(a',J)$ lifts to a class $(b,a,J,I)$ in $H^1(X,D(P_x))$, the
proof of \cite[Lemma~4.11]{AhlqvistCarlson} produces
$u\in L_x^\times$ by the Hasse norm theorem and a fractional
ideal $I'$ of $L_x$ by Hilbert's theorem~90 for ideals, and the proof of
\cite[Theorem~4.12]{AhlqvistCarlson}, after adjusting the lift by
a boundary so that $N_x(a)=a'$, assembles the witness with
$t=a+(1-\sigma_x)u$.

For $p>3$, that theorem now gives
\[
  \bigl\langle\langle x,x,y\rangle,(a',J)\bigr\rangle_{\AV}
  =
  \langle y,N_x(I')\rangle
\]
for every $y\in H^1(G,\Fp)$.  The pairing on the right is the
Artin--Verdier pairing in degree one,
\[
  H^1(X,\Fp)\times H^2\bigl(X,D(\Fp)\bigr)\longrightarrow\Fp .
\]
Here $H^2(X,D(\Fp))\simeq\Cl(K)/p$ by
\cite[Corollary~2.15]{AhlqvistCarlsonRing}, so the second argument is
the class of a fractional ideal of $K$.  The pairing is
computed by Artin symbols.  By the proof of
\cite[Lemma~4.1]{AhlqvistCarlsonRing} it sends a pair
$(y,\mathfrak a)$ to the exponent $n\in\Fp$ determined by
\[
  \operatorname{Artin}_{L_y| K}(\mathfrak a)=\sigma_y^{\,n},
\]
with $\sigma_y$ the generator \eqref{eq:sigma-dual} dual to $y$,
in agreement with \cite[(7.7)]{AhlqvistPink} (their $\rho_y$ is
our $\sigma_y$).  By
\eqref{eq:sigma-character} the exponent equals $y(\overline{\mathfrak a})$,
so under \eqref{eq:H-classgroup} the degree-one pairing is the
evaluation pairing, and
\[
  \langle y,N_x(I')\rangle
  =
  y\bigl([N_x(I')]\bigr).
\]
On the other hand, since $(a',J)$ represents $e\in\Cl(K)[p]$, the defining
property \eqref{eq:Ddef} of the secondary norm operator says
\[
  \bigl\langle\langle x,x,y\rangle,(a',J)\bigr\rangle_{\AV}
  =
  y\bigl(D_x(e)\bigr).
\]
Consequently
\[
  y\bigl(D_x(e)\bigr)
  =
  y\bigl([N_x(I')]\bigr)
\]
for every
\[
  y\in H^1(G,\Fp)
  \simeq\bigl(\Cl(K)/p\bigr)^*.
\]
The evaluation pairing
\[
  \bigl(\Cl(K)/p\bigr)^*\times\Cl(K)/p\longrightarrow\Fp
\]
is perfect, and hence its characters separate the elements of
$\Cl(K)/p$.  We therefore obtain
\[
  D_x(e)=[N_x(I')]\in\Cl(K)/p,
\]
as claimed.
\end{proof}

\begin{proposition}[The case $p=3$]\label{prop:AC-p3}
In the situation of \cref{prop:AC} with $p=3$, norm witnesses
again exist, and for every representing pair $(a',J)$ and every
norm witness $(t,I')$
\begin{equation}\label{eq:ACD-p3}
  D_x(e)=[N_x(I')\,J]\in\Cl(K)/3.
\end{equation}
\end{proposition}

\begin{proof}
The hypotheses of Ahlqvist--Carlson's Theorem~4.12 are verified as
in the proof of \cref{prop:AC}.  All cup products vanish,
so every pair $(a',J)$ satisfies the kernel conditions of their
theorem.  The existence of a witness follows as at $p>3$.  The
$p=3$ case of their Lemma~4.11 runs with the extra factor $J$.
They state this at the end of that proof and leave the
verification to the reader.  At $p=3$ the theorem carries
the additional factor $J$ and gives
\[
  \bigl\langle\langle x,x,y\rangle,(a',J)\bigr\rangle_{\AV}
  =
  \langle y,N_x(I')\,J\rangle
\]
for every $y$.  The additional term is the one displayed in
\cite[(7.7)]{AhlqvistPink}.  The remaining
identifications are those of the proof of \cref{prop:AC}, and
perfection of the evaluation pairing yields \eqref{eq:ACD-p3}.
\end{proof}

\begin{remark}[Choice and construction of norm witnesses]
\label{rem:witness-choices}
Since the left-hand sides of \eqref{eq:ACD} and \eqref{eq:ACD-p3}
are defined intrinsically by \eqref{eq:Ddef}, every representing
pair $(a',J)$ for $e$ and every norm witness $(t,I')$ yield the
same class, so the choices are auxiliary.  The existence argument in
the proof of \cref{prop:AC} is not effective, and it is not used in
the computations.  For the fields of \cref{sec:results} the searches
construct the witnesses explicitly (\cref{subsec:searches}), the
stored data records them, and the verification checks the two
equations of \cref{def:norm-witness} directly for every character
and every class that enter the computation
(Appendix~\ref{app:verification}).
\end{remark}

The formula also explains the terminology in \cref{def:D}.  For fixed
$x$, the map
\[
  H^1(G,\Fp)\longrightarrow H^2(G,\Fp),
  \qquad
  y\longmapsto Q(x,y)=M(x,x,y),
\]
is a secondary cohomology operation, and by \eqref{eq:ACD} every
value $D_x(e)=[N_x(I')]$ of the linear operator
$D_x\colon\Cl(K)[p]\to\Cl(K)/p$ is the class of a norm from
$L_x$, whence the name.

The following lemma determines the map $x\mapsto M(x,x,x)$ at $p=3$,
where \eqref{eq:diagonal-triple} leaves it unconstrained.  It turns
out to be $\mathbb F_3$-linear, and its rank and its kernel enter
the $p=3$ part of the mildness criterion in
\cref{subsec:rank-drop-geometry}.

\begin{lemma}[The diagonal at $p=3$]\label{lem:diagonal-p3}
Let $p=3$ and $d=d_3\Cl(K)$.  Choose bases
$\chi_1,\ldots,\chi_d$ of $H^1(G,\mathbb F_3)$ and
$e_1,\ldots,e_d$ of $\Cl(K)[3]$, and let $B$ be the $d\times d$
matrix over $\mathbb F_3$ with entries
\begin{equation}\label{eq:bockstein-matrix}
  B_{\ell i}=\langle M(\chi_i,\chi_i,\chi_i),e_\ell\rangle_{\AV}
  \qquad(1\le \ell,i\le d).
\end{equation}
\begin{enumerate}[label=\textup{(\roman*)}]
\item\label{diag:linear} For $x=\sum_ix_i\chi_i$ with
$x_i\in\mathbb F_3$,
\begin{equation}\label{eq:diagonal-cube}
  \langle M(x,x,x),e_\ell\rangle_{\AV}
  =\Bigl(\sum_iB_{\ell i}x_i\Bigr)^{3}
  =\sum_iB_{\ell i}x_i .
\end{equation}
In particular the diagonal map
\[
  \beta\colon H^1(G,\mathbb F_3)\longrightarrow H^2(G,\mathbb F_3),
  \qquad x\longmapsto M(x,x,x),
\]
is $\mathbb F_3$-linear, with matrix $B$ relative to
$\chi_1,\ldots,\chi_d$ and the basis of $H^2(G,\mathbb F_3)$ dual to
$e_1,\ldots,e_d$ under \eqref{eq:H2-classgroup}.
\item\label{diag:dual} Under the Artin--Verdier identifications the
dual of $\beta$ is the natural map
\[
  \delta\colon\Cl(K)[3]\longrightarrow\Cl(K)/3 ;
\]
in particular $\operatorname{rk}\beta=\operatorname{rk}\delta$, and
this rank is the number of cyclic factors of order exactly $3$ in the
$3$-part of $\Cl(K)$.
\item\label{diag:bockstein} $\beta$ is the Bockstein map, that is,
the connecting homomorphism
$H^1(G,\mathbb F_3)\to H^2(G,\mathbb F_3)$ in the long exact
cohomology sequence of the coefficient sequence
$0\to\mathbb Z/3\to\mathbb Z/9\to\mathbb Z/3\to0$.
\end{enumerate}
\end{lemma}

\begin{proof}
For \ref{diag:linear}, expanding by trilinearity gives
\[
  \langle M(x,x,x),e_\ell\rangle_{\AV}
  =\sum_{i,j,k}x_ix_jx_k\,
   \langle M(\chi_i,\chi_j,\chi_k),e_\ell\rangle_{\AV}.
\]
Group the terms according to whether the indices $i,j,k$ are all
equal, exactly two are equal, or pairwise distinct.  If exactly two
agree, the orderings occurring are the three cyclic permutations of
$(\chi_i,\chi_i,\chi_j)$ with $i\ne j$, and their sum vanishes by the
cyclic identity in \eqref{eq:shuffle}.  If the three
indices are distinct, the six orderings split into two cyclic
triples, each summing to zero for the same reason.  Only the terms
with $i=j=k$ survive.  In characteristic three cubing is additive and
fixes $\mathbb F_3$, so with the entries \eqref{eq:bockstein-matrix},
\[
\begin{aligned}
  \langle M(x,x,x),e_\ell\rangle_{\AV}
  &=\sum_i x_i^{3}\,
    \langle M(\chi_i,\chi_i,\chi_i),e_\ell\rangle_{\AV}
   =\sum_iB_{\ell i}x_i^{3}\\
  &=\Bigl(\sum_iB_{\ell i}x_i\Bigr)^{3}
   =\sum_iB_{\ell i}x_i ,
\end{aligned}
\]
which is \eqref{eq:diagonal-cube}.  The right-hand side is linear in
the coordinates of $x$, and by \eqref{eq:AVpair} an element of
$H^2(G,\mathbb F_3)$ is determined by its evaluations against
$e_1,\ldots,e_d$; hence $\beta$ is linear, with the asserted matrix.

For \ref{diag:dual}, write $e_\ell=[J_\ell]$ as
in \eqref{eq:e-is-J}.  Then, by the definition
\eqref{eq:bockstein-matrix} of $B_{\ell i}$, by \eqref{eq:Ddef} with
$y=x=\chi_i$, and by \cref{prop:AC-p3},
\[
  B_{\ell i}
  =\langle M(\chi_i,\chi_i,\chi_i),e_\ell\rangle_{\AV}
  =\chi_i\bigl(D_{\chi_i}(e_\ell)\bigr)
  =\chi_i\bigl([N_{\chi_i}(I')]\bigr)+\chi_i\bigl([J_\ell]\bigr)
  =\chi_i\bigl(\delta(e_\ell)\bigr),
\]
because a norm class from $L_{\chi_i}$ has trivial Artin symbol in
$\Gal(L_{\chi_i}| K)$ and is therefore annihilated by $\chi_i$, by
\eqref{eq:artin-compat}, while $\delta(e_\ell)$ is the class of $J_\ell$ in
$\Cl(K)/3$.  Thus $B^{t}$ is the matrix of $\delta$ with respect to
$e_1,\ldots,e_d$ and the basis of $\Cl(K)/3$ dual to
$\chi_1,\ldots,\chi_d$, which is \ref{diag:dual}.

It remains to identify the rank.  The map $\delta$
is the restriction to the $3$-torsion of the quotient map
$\Cl(K)\to\Cl(K)/3$.  On a cyclic factor $\mathbb Z/3^a$ of the
$3$-part its $3$-torsion is generated by $3^{a-1}$, whose image in the
quotient by $3$ is nonzero exactly when $a=1$.  Hence
$\operatorname{rk}\delta$ is the number of factors with $a=1$.

For \ref{diag:bockstein}, write $B$ for that connecting
homomorphism.  In the notation of \cref{sec:relations} (a minimal
presentation $1\to R\to F\to G\to1$ with the Zassenhaus filtration
$F_{(n)}$ of $F$), the vanishing of the cup products gives
$R\subseteq F_{(3)}$ \cite[Corollary~1.2.9]{Vogel}, and
\cite[Proposition~1.2.15]{Vogel} computes the
threefold diagonal of $M^{\mathrm V}$ as $-B$; then
\cref{lem:sign} gives $M=M^{\mathrm{AC}}=-M^{\mathrm V}$, hence
$\beta=B$.
\end{proof}

\begin{remark}\label{rem:bockstein}
The dual description of $\beta$ by $\delta$ in
\cref{lem:diagonal-p3}\ref{diag:dual} is the analogue, in the present
setting, of \cite[Proposition~4.4]{ChungEtAl} and
\cite[Lemma~4.1]{AhlqvistCarlsonRing}.  Both are stated for base
fields containing $\mu_n$, whereas an
imaginary quadratic field contains a primitive cube root of unity
only for $K=\mathbb Q(\sqrt{-3})$, excluded here.  Compare also
\cite[Remark~4.13]{AhlqvistCarlson}.  Of the three parts, only
$\ker\beta$ and $\operatorname{rk}\beta$ are used below.  The
identification with the Bockstein map serves to name the cone of
\cref{def:bockstein-cone}.
\end{remark}

Thus at $p=3$ the image of $D_x$ need not lie in $x^\perp$, but this
failure is measured by the linear map $\beta$.  Since
$x(D_x(e))=\langle\beta(x),e\rangle_{\AV}$, the orthogonality
$\operatorname{im}D_x\subseteq x^\perp$ holds precisely for the
characters in $\ker\beta$.  This is why the cone condition enters
\cref{def:transverse-rank-one}.  It is the condition $\beta(x)=0$,
and the orthogonality it restores is what makes the map $\Theta_x$ of
\cref{lem:Theta-well-defined} well defined.  Over a finite extension
$k|\mathbb F_3$ the diagonal identity \eqref{eq:diagonal-cube}
holds in cubed form.  For $x=\sum_ix_i\chi_i$ with $x_i\in k$, the
scalar extension \eqref{eq:scalar-extended-D} gives
\[
  \langle M(x,x,x),e_\ell\rangle_{\AV}
  =\Bigl(\sum_iB_{\ell i}x_i\Bigr)^{3},
\]
so over $k$ the diagonal is \emph{not} $k$-linear.  It satisfies
$M(\lambda x,\lambda x,\lambda x)=\lambda^{3}M(x,x,x)$ for
$\lambda\in k$, and its evaluations are the cubes of the linear
forms belonging to $\beta$.  Since $t\mapsto t^{3}$ is injective on
$k$, the diagonal nevertheless vanishes at $x$ if and only if the
$k$-linear extension of $\beta$ does.  Only the vanishing locus is
used below, and the
scheme-theoretic difference is taken up in \cref{rem:reduced-cone}.

\begin{remark}[The two changes at $p=3$]\label{rem:p3}
Two independent phenomena separate $p=3$ from the larger primes.
First, Ahlqvist--Carlson's formula carries the additional
factor $J$ of \cref{prop:AC-p3}.  Second, the restricted $p$-power operation in the
Zassenhaus graded object contributes in degree three, so the cubes
join the initial forms.  This graded restricted Lie algebra is
recalled in \cref{sec:relations}.  Both changes are incorporated
there and in \cref{sec:example-p3}.  The transversality geometry
itself persists at $p=3$.  It lives along the Bockstein cone
(\cref{def:bockstein-cone}), the projectivized kernel of the linear
diagonal of \cref{lem:diagonal-p3}.
\end{remark}

\section{Mildness from transverse points of the norm-degeneracy
scheme}\label{sec:relations}

From now on write
\[
  d=d_p\Cl(K);
\]
by \eqref{eq:ranks} this is also the relation rank.  No further
assumption on $d$ is made in this section; the statements that need
one carry it explicitly.  Only the two main statements, \cref{thm:transverse-rank-one} and
\cref{cor:geometric-mildness}, repeat their hypotheses in full, so
that they can be quoted independently; every other statement is
read under the standing assumptions of this section and of
\cref{sec:setup}.  Choose a basis $\chi_1,\ldots,\chi_d$ of $H^1(G,\Fp)$ and a basis
$e_1,\ldots,e_d$ of $\Cl(K)[p]$.  Let
\begin{equation}\label{eq:minimal-presentation}
  1\longrightarrow R\longrightarrow F\longrightarrow G\longrightarrow1
\end{equation}
be a minimal pro-$p$ presentation.  Minimality means that inflation
induces an isomorphism
\begin{equation}\label{eq:inflation}
  \operatorname{inf}\colon H^1(G,\Fp)\xrightarrow{\ \sim\ }H^1(F,\Fp)
  =\Hom_{\Fp}\bigl(F/F^p[F,F],\Fp\bigr);
\end{equation}
we choose free generators $f_1,\ldots,f_d$ of $F$ dual to
$\chi_1,\ldots,\chi_d$ under \eqref{eq:inflation}, that is,
$\operatorname{inf}(\chi_j)(f_i)=\delta_{ij}$.  Let $\mgen_1,\ldots,\mgen_d$ be noncommuting variables.  The Magnus
embedding is the injective homomorphism
\[
  \mu\colon F\longrightarrow
  1+(\mgen_1,\ldots,\mgen_d)
  \subseteq
  \Fp\langle\!\langle \mgen_1,\ldots,\mgen_d\rangle\!\rangle,
  \qquad
  \mu(f_i)=1+\mgen_i.
\]
Here $\Fp\langle\!\langle \mgen_1,\ldots,\mgen_d\rangle\!\rangle$
denotes the algebra of formal power series in $d$ noncommuting variables.
We write $F_{(n)}$ for the Zassenhaus filtration of $F$: denoting by
$\Fp[[F]]$ the completed group algebra and by $I_F$ its augmentation
ideal,
\[
  F_{(n)}=\{f\in F:\ f-1\in I_F^{\,n}\},
  \qquad n\ge1,
\]
see \cite[Remark~2 following~(3.9.8)]{NSW}.  The map $\mu$ extends,
by the same letter, to an isomorphism of complete algebras
\[
  \mu\colon\Fp[[F]]\xrightarrow{\ \sim\ }
  \Fp\langle\!\langle \mgen_1,\ldots,\mgen_d\rangle\!\rangle
\]
carrying $I_F$ onto the ideal $(\mgen_1,\ldots,\mgen_d)$, hence
$I_F^{\,n}$ onto the power series with no terms of degree less than
$n$.  The filtration is therefore detected by degree, in that
\[
  f\in F_{(n)}
  \iff
  \mu(f)-1\ \text{has no terms in degrees }1,\ldots,n-1;
\]
cf.\ \cite[(1.1.24)]{Vogel}.  Writing
\[
  \operatorname{gr}F:=\bigoplus_{n\ge1}F_{(n)}/F_{(n+1)},
\]
the inclusions $[F_{(r)},F_{(s)}]\subseteq F_{(r+s)}$ and
$F_{(r)}^{\,p}\subseteq F_{(pr)}$ make $\operatorname{gr}F$ a graded
restricted Lie algebra over $\Fp$, with bracket induced by the
commutator and $p$-operation by the $p$-th power.  Let
$\mathcal A=\Fp\langle \mgen_1,\ldots,\mgen_d\rangle$ be the
free associative algebra, with homogeneous degree-$n$ part
$\mathcal A_n$.  It is a restricted Lie algebra under $[a,b]=ab-ba$
and $a^{[p]}=a^p$.  In each degree, $\mu$ induces an injective
$\Fp$-linear map
\[
  \mu_n\colon F_{(n)}/F_{(n+1)}\longrightarrow\mathcal A_n,
  \qquad
  f\bmod F_{(n+1)}\longmapsto
  \text{degree-$n$ component of }\mu(f)-1,
\]
and together the $\mu_n$ identify $\operatorname{gr}F$ with the
restricted Lie subalgebra of $\mathcal A$ generated by
$\mgen_1,\ldots,\mgen_d$, which is free on these generators and
has $\mathcal A$ as its universal enveloping algebra
\cite[Remark~2.3]{Gaertner}.

In particular, for $r\in F_{(3)}$ we may write
\[
  \mu(r)=1+\rho+\text{terms of degree at least $4$},
\]
where $\rho$ is homogeneous of degree three.  We call $\rho$ the cubic
Magnus polynomial, or cubic initial form, of $r$.  The definition
applies to any $r\in F_{(3)}$.  What we identify below is the
coefficient of each ordered word $\mgen_i\mgen_j\mgen_k$ in the cubic
initial forms of the relators of \eqref{eq:minimal-presentation}.

\subsection{The cubic word matrix and the initial forms of the
  relators}

The chosen bases turn the polarization $\Delta D$ of
\cref{thm:reconstruction} into a table of scalars, which we record as
a matrix
\begin{equation}\label{eq:T}
  T=(m_{\ell,ijk})\in M_{d\times d^3}(\Fp),
  \qquad
  m_{\ell,ijk}
   =\chi_j\bigl(\Delta D(\chi_i,\chi_k)(e_\ell)\bigr),
\end{equation}
with rows indexed by $e_1,\ldots,e_d$ and columns by the triples
$(i,j,k)$ with $1\le i,j,k\le d$.  We call $T$ the \emph{cubic word
matrix}.  By \eqref{eq:MfromD} the entries
are the values of the triple Massey product,
\begin{equation}\label{eq:T-massey}
  m_{\ell,ijk}
  =\bigl\langle M(\chi_i,\chi_j,\chi_k),e_\ell\bigr\rangle_{\AV},
\end{equation}
so that $T$ is at this stage nothing but the arithmetic of \S2 written
out in coordinates, computable from the $d(d+1)/2$ operators
\eqref{eq:minimal-family} alone (\cref{rem:six-computations} spells
this out for $d=3$).

We label the $d^3$ columns not by the triples $(i,j,k)$ but by the
words $\mgen_i\mgen_j\mgen_k$, ordered lexicographically, so that the
order begins
\[
  \mgen_1^3,\ \mgen_1^2\mgen_2,\ \mgen_1^2\mgen_3,\
  \mgen_1\mgen_2\mgen_1,\ldots .
\]
Nothing is claimed by this relabelling.  Its point is
\cref{prop:T-is-relations} below, where the $\ell$-th row of $T$ turns
out to be the coordinate vector of the cubic initial form of the
$\ell$-th relator in exactly this basis of words.  Each column then
holds, in row $\ell$, the coefficient of its own word in that initial
form.

We next explain how the rows of $T$ are related to defining relations.  Put
\[
  \overline R=R/R^p[R,F].
\]
Recall the inflation isomorphism \eqref{eq:inflation}.  Moreover,
$F$ is a free pro-$p$ group, so $H^2(F,\Fp)=0$.  The five-term
exact sequence \cite[(1.6.7)]{NSW} therefore shows that transgression
is an isomorphism
\[
  \operatorname{tg}\colon H^1(R,\Fp)^G\xrightarrow{\sim}H^2(G,\Fp);
\]
here and below, $\operatorname{tg}$ is the transgression normalized
as in \cite[(1.6.6)]{NSW}, the convention of \cite{Vogel}.
A continuous homomorphism $R\to\Fp$ is $G$-invariant precisely when it
vanishes on $[R,F]$, and it automatically vanishes on $R^p$.  Hence
\[
  H^1(R,\Fp)^G
  =\Hom_{\Fp}(\overline R,\Fp)
  =\overline R^*.
\]
Dualizing $\operatorname{tg}$ and using
$H^1(R,\Fp)^G=\overline R^*$ therefore gives the isomorphism
\[
  \operatorname{tg}^*\colon H^2(G,\Fp)^*
  \xrightarrow{\ \sim\ }\overline R^{**}=\overline R .
\]
On the other side, the perfect pairing \eqref{eq:AVpair} of
\cref{prop:comparison}\ref{comp:pairing} gives the isomorphism
\[
  \Cl(K)[p]\xrightarrow{\ \sim\ }H^2(G,\Fp)^*,
  \qquad
  e\longmapsto\lambda_e,\quad
  \lambda_e(\alpha)=\langle\alpha,e\rangle_{\AV}.
\]
The chosen basis $e_1,\ldots,e_d$ of $\Cl(K)[p]$ thus determines a basis
of $\overline R$, and we orient it explicitly: writing
$\lambda_\ell=\lambda_{e_\ell}$, set
\begin{equation}\label{eq:negative-orientation}
  \overline{r_\ell}
  =-\operatorname{tg}^*(\lambda_\ell)\in\overline R,
\end{equation}
and choose any lift $r_\ell\in R$ of $\overline{r_\ell}$.  The sign in
\eqref{eq:negative-orientation} belongs to the definition of the
classes $\overline{r_\ell}$ themselves; we do not first fix relators
and then pass to their inverses.  Since the $\lambda_\ell$ form a
basis, the normal-generation criterion \cite[(3.9.3)]{NSW}
(elements of $R$ generate $R$ as a normal subgroup if and only if their
images span $R/R^p[R,F]$) shows that $r_1,\ldots,r_d$
form a minimal system of normal generators of $R$.  Moreover
$R\subseteq F_{(3)}$, by the vanishing of the cup products
\cite[Corollary~1.2.9]{Vogel}; hence
$[R,F]\subseteq[F_{(3)},F]\subseteq F_{(4)}$ and
$R^p\subseteq F_{(3)}^{\,p}\subseteq F_{(3p)}\subseteq F_{(4)}$,
and the map
\[
  \vartheta_3\colon\overline R\longrightarrow\mathcal A_3,
  \qquad
  \overline r\longmapsto
  \text{cubic initial form of }r
  =\mu_3\bigl(r\bmod F_{(4)}\bigr),
\]
is a well-defined homomorphism of $\Fp$-vector spaces.  We set
$\rho_\ell=\vartheta_3(\overline{r_\ell})$.  It is the cubic initial
form of $r_\ell$ and is independent of the choice of the lift
$r_\ell$.

\begin{proposition}\label{prop:T-is-relations}
For every odd $p$,
\begin{equation}\label{eq:rho-from-T}
  \rho_\ell
   =\sum_{i,j,k=1}^d m_{\ell,ijk}\mgen_i\mgen_j\mgen_k,
  \qquad\ell=1,\ldots,d.
\end{equation}
Consequently, the rows of $T$ are the associative word
coordinates of the cubic initial forms of the $d$ minimal
relators.  For $p>3$ the diagonal coefficients $m_{\ell,iii}$ vanish by
\eqref{eq:diagonal-triple};
for $p=3$ they are the Bockstein values of \cref{lem:diagonal-p3}.
\end{proposition}

\begin{proof}
For $r\in R$, the trace map with respect to $r$ is
\[
  \tr_r\colon H^2(G,\Fp)\longrightarrow\Fp,
  \qquad
  \alpha\longmapsto\bigl(\operatorname{tg}^{-1}\alpha\bigr)(r);
\]
see \cite[III.9]{NSW}.  The character of $R$ transgressing to $\alpha$
lies in $H^1(R,\Fp)^G=\overline R^*$, and $\tr_r$ evaluates it at $r$.
Since that character vanishes on $R^p[R,F]$, only the class $\overline r$
enters.  Under the identification $\overline R^{**}=\overline R$ used
above, an element $\lambda\in H^2(G,\Fp)^*$ has image
$\operatorname{tg}^*(\lambda)$ characterized by
\[
  \varphi\bigl(\operatorname{tg}^*(\lambda)\bigr)
   =\lambda\bigl(\operatorname{tg}\varphi\bigr)
  \qquad\text{for all }\varphi\in\overline R^* .
\]
Applying this to $\varphi=\operatorname{tg}^{-1}\alpha$ and inserting
\eqref{eq:negative-orientation} gives
\begin{equation}\label{eq:trace-orientation}
  \tr_{r_\ell}=-\lambda_\ell .
\end{equation}
Because $R\subseteq F_{(3)}$, the coefficient theorem
\cite[Theorem~1.2.6]{Vogel} applies to every ordered multi-index of
length three, including repeated indices.  The coefficient
$\varepsilon_{ijk,3}(r_\ell)$ of the word $\mgen_i\mgen_j\mgen_k$ in
the degree-three Magnus polynomial of $r_\ell$ is
\[
  \varepsilon_{ijk,3}(r_\ell)
   =\tr_{r_\ell}M^{\mathrm V}(\chi_i,\chi_j,\chi_k).
\]
Combining this with the global sign \eqref{eq:global-sign} of
\cref{lem:sign} and with \eqref{eq:trace-orientation} (the
orientation \eqref{eq:negative-orientation} is chosen so that the
two minus signs cancel) gives the entrywise identity
\begin{equation}\label{eq:coefficient-identity}
  \varepsilon_{ijk,3}(r_\ell)
  =\bigl\langle M(\chi_i,\chi_j,\chi_k),e_\ell\bigr\rangle_{\AV}
  =m_{\ell,ijk}.
\end{equation}
Reading it for all $d^3$ ordered words gives \eqref{eq:rho-from-T}.
\end{proof}

\begin{remark}[The complete cubic initial relation space]
\label{rem:complete-space}
By the \emph{complete cubic initial relation space} $R_3$ we mean
the span of the cubic initial forms of all relations $r\in R$.
Since the cubic initial form of $r$ is $\vartheta_3(\overline r)$,
\begin{equation}\label{eq:R3}
  R_3=\operatorname{im}(\vartheta_3)
     =\bigl\langle\rho_1,\ldots,\rho_d\bigr\rangle_{\Fp},
  \qquad
  \dim_{\Fp}R_3=\rk_{\Fp}T.
\end{equation}
\end{remark}

\begin{remark}[Comparison with the commutator presentation]
The same data appear in group-theoretic form in Ahlqvist--Pink's
Proposition~7.5, which translates the presentation formula of
\cite[Proposition~1.3.3 and Example~1.2.11]{Vogel} into the
étale-cohomological setting in which Ahlqvist--Carlson compute.  The
triple Massey evaluations occur there as exponents of triple
commutators in a presentation of $G$, and here as the coefficients
\eqref{eq:T-massey} of the words $\mgen_i\mgen_j\mgen_k$.
\end{remark}

\subsection{The shuffle-symmetric space and reconstruction at
  \texorpdfstring{$p=3$}{p = 3}}

As recalled at the beginning of this section, the maps
$\mu_n$ identify $\operatorname{gr}F$ with the free restricted Lie
algebra over $\Fp$ on $\mgen_1,\ldots,\mgen_d$ inside
$\mathcal A$.  In particular
$R_3\subseteq\mu_3\bigl(F_{(3)}/F_{(4)}\bigr)$ lies in the
degree-three part of this subalgebra.  Write
$L_3\subseteq\mathcal A_3$ for the degree-three component of the
ordinary free Lie algebra on $\mgen_1,\ldots,\mgen_d$, embedded
via $[a,b]=ab-ba$.  The following lemma computes the degree-three
part.  Call an element
\[
  \sum_{i,j,k=1}^d m_{ijk}\,\mgen_i\mgen_j\mgen_k\in\mathcal A_3
\]
\emph{shuffle-symmetric} if its coefficients satisfy
\begin{equation}\label{eq:admissible}
  m_{ijk}=m_{kji},
  \qquad
  m_{ijk}+m_{jki}+m_{kij}=0
\end{equation}
for all $i,j,k$, and write $\mathcal R_3\subseteq\mathcal A_3$ for
the subspace of shuffle-symmetric elements.  The initial forms
$\rho_1,\ldots,\rho_d$ are shuffle-symmetric by the identities
\eqref{eq:shuffle} for $M$, so $R_3\subseteq\mathcal R_3$.  The
shuffle-symmetric elements are the candidate cubic initial forms,
inside which the actual relation space $R_3$ sits, and the calligraphic
notation records this.

\begin{lemma}[The shuffle-symmetric space]\label{lem:eleven-space}
The space $\mathcal R_3$ is precisely the degree-three part of the
free restricted Lie algebra
on $\mgen_1,\ldots,\mgen_d$: for $p>3$,
\[
  \mathcal R_3=L_3,\qquad\dim_{\Fp}\mathcal R_3=\frac{d^3-d}3,
\]
and for $p=3$,
\[
  \mathcal R_3=L_3\oplus
   \langle\mgen_1^3,\ldots,\mgen_d^3\rangle,
  \qquad\dim_{\Fp}\mathcal R_3=\frac{d^3-d}3+d.
\]
\end{lemma}

\begin{proof}
The proof compares three subspaces of $\mathcal A_3$: the shuffle-symmetric
space $\mathcal R_3$, defined by the linear identities
\eqref{eq:admissible}; the degree-three part
$\mu_3\bigl(F_{(3)}/F_{(4)}\bigr)$ of $\operatorname{gr}F$; and
the Lie part $L_3$.  We compute the dimension of the first, produce
a basis of the second, and obtain the equalities by comparison.

Both identities in
\eqref{eq:admissible} relate $m_{ijk}$ only to
coefficients of words with the same letters, so $\mathcal R_3$
decomposes accordingly.  For three equal letters the cyclic identity
gives $3m_{iii}=0$, so the $d$ diagonal coefficients vanish for $p>3$
and are unconstrained for $p=3$.  For the letters
$\mgen_i,\mgen_i,\mgen_j$ with $i\ne j$, reversal and the cyclic
identity express $m_{jii}$ and $m_{iji}$ through $m_{iij}$, leaving
one dimension for each of the $d(d-1)$ ordered pairs $(i,j)$.  For
three pairwise distinct letters, reversal pairs the six words formed
from them, and the one remaining cyclic relation reduces the
three parameters to two, leaving two dimensions for each of the
$\binom d3$ unordered triples.  Hence
\[
  \dim\mathcal R_3
  =d(d-1)+2\binom d3=\frac{d^3-d}3\ \text{for } p>3,
  \qquad
  \frac{d^3-d}3+d\ \text{for } p=3.
\]

Next, by
\cite[Proposition~1.3.3]{Vogel}, every element of $F_{(3)}$ is,
modulo $F_{(4)}$, uniquely a product of the basic commutators
$((f_k,f_l),f_m)$ with $k<l$ and $m\le l$, of which there are
$(d^3-d)/3$, and, for $p=3$, of the cubes $f_k^3$; their classes
therefore form a basis of $F_{(3)}/F_{(4)}$.  The injective $\mu_3$
carries this basis to the expansions
\[
  [[\mgen_k,\mgen_l],\mgen_m]
  =\mgen_k\mgen_l\mgen_m-\mgen_l\mgen_k\mgen_m
  -\mgen_m\mgen_k\mgen_l+\mgen_m\mgen_l\mgen_k
\]
with $k<l$, $m\le l$ and, for $p=3$, the cubes $\mgen_k^3$.
Together they form a basis of $\mu_3\bigl(F_{(3)}/F_{(4)}\bigr)$.

Each element of this basis is shuffle-symmetric, so
$\mu_3\bigl(F_{(3)}/F_{(4)}\bigr)\subseteq\mathcal R_3$; both
spaces have dimension $(d^3-d)/3$, respectively $(d^3-d)/3+d$, so
they are equal.
The bracket elements lie in $L_3$, which has dimension
$(d^3-d)/3$ by the Witt formula; they are therefore a basis of
$L_3$, and the basis of $\mathcal R_3$ splits into this basis of
$L_3$ and, for $p=3$, the $d$ cubes.  The displayed formulas
follow, the sum being direct.
\end{proof}

At $p=3$ the shuffle identities also degenerate off the diagonal.
Taking the outer entries equal gives
\begin{equation}\label{eq:p3-contraction}
  M(x,y,x)=-2M(x,x,y)=M(x,x,y).
\end{equation}
For $x=\sum_ix_i\chi_i\in H^1(G,\Fthree)$ define the
\emph{contraction}
\[
  \gamma_x\colon\mathcal R_3\longrightarrow\Fthree^d,
  \qquad
  \gamma_x(m)_j=\sum_{i,k=1}^dx_ix_k\,m_{ijk}.
\]
Every $e\in\Cl(K)[p]$ determines the element
$m(e)\in\mathcal R_3$ with coefficients
\[
  m(e)_{ijk}=\bigl\langle M(\chi_i,\chi_j,\chi_k),e\bigr\rangle_{\AV},
\]
shuffle-symmetric by \eqref{eq:shuffle}.  In the word
basis, the coefficients of $m(e_\ell)$ form the $\ell$-th row of
$T$, by \eqref{eq:T-massey}.  Then
\begin{equation}\label{eq:row-D}
  \chi_j(D_x(e))
  =\bigl\langle M(x,x,\chi_j),e\bigr\rangle_{\AV}
  =\bigl\langle M(x,\chi_j,x),e\bigr\rangle_{\AV}
  =\gamma_x\bigl(m(e)\bigr)_j.
\end{equation}
By \eqref{eq:row-D}, knowing the operator $D_x$, that is, all
its values $\chi_j(D_x(e))$, amounts to knowing the contraction
$\gamma_x(m(e))$ for every $e$.  In the reconstruction problem the
element $m(e)$ is the unknown, and a priori it is constrained only
to lie in $\mathcal R_3$.  If the contractions of a family of
characters are jointly injective on $\mathcal R_3$, the
corresponding operators recover $m(e)$ for every $e$, and in
particular $T$; a nonzero kernel leaves two shuffle-symmetric elements with
the same contractions.  This proves the refinement
announced in \cref{rem:six-computations}:

\begin{lemma}[Four-evaluation reconstruction at $p=3$]
\label{lem:four-evaluations}
Let $p=3$ and $d=3$.  The four operators
$D_{\chi_1},D_{\chi_2},D_{\chi_3},D_{\chi_1+\chi_2+\chi_3}$
determine $m(e)$ for every $e\in\Cl(K)[p]$, and in particular the
matrix $T$.  Equivalently, the joint contraction map
\[
  (\gamma_{\chi_1},\gamma_{\chi_2},\gamma_{\chi_3},
   \gamma_{\chi_1+\chi_2+\chi_3})
  \colon\mathcal R_3\longrightarrow(\Fthree^3)^4
\]
is injective.  No family of contractions at three or fewer
characters is injective on $\mathcal R_3$.
\end{lemma}

\begin{proof}
Let $m\in\mathcal R_3$.  At $p=3$, for indices $a\ne j$, the
identities \eqref{eq:admissible} give
\[
  m_{jaa}=m_{aaj},
  \qquad
  m_{aja}=-2\,m_{aaj}=m_{aaj}.
\]
We first evaluate the three basis characters.  For $x=\chi_a$ the
coordinates $x_i=\delta_{ia}$ reduce the contraction to
\[
  \gamma_{\chi_a}(m)_j=m_{aja};
\]
for $j=a$ this is the diagonal coefficient $m_{aaa}$, and for
$j\ne a$ the displayed identities recover $m_{aaj}$ and $m_{jaa}$.
The basis characters therefore determine every coefficient in which
an index repeats.

Next we evaluate the sum character $s=\chi_1+\chi_2+\chi_3$, for
which
\[
  \gamma_s(m)_j=\sum_{i,k=1}^3 m_{ijk}.
\]
Of the nine terms of this sum, those in which an index repeats are
determined by the first step.  The two remaining terms are $m_{ijk}$
and $m_{kji}$, where $i$ and $k$ denote the two indices other than
$j$; by reversal they are equal, so their sum is $2\,m_{ijk}$.  Hence $\gamma_s(m)_j$ and the
first step together determine the coefficients with pairwise
distinct indices, and the joint map is injective.

For the last assertion, a family of three contractions takes values
in $(\Fthree^3)^3$, of dimension nine, while
$\dim_{\Fthree}\mathcal R_3=11$ by \cref{lem:eleven-space}; no such
family is injective.
\end{proof}

\subsection{Strong freeness and mildness}
\label{subsec:strong-freeness}

Let $d,m\ge1$, let
$\mathcal A=\Fp\langle \mgen_1,\ldots,\mgen_d\rangle$, and let
$\rho_1,\ldots,\rho_m\in\mathcal A$ be homogeneous of degrees
$h_1,\ldots,h_m\ge1$.  For a connected graded $\Fp$-algebra
$C=\bigoplus_{n\ge0}C_n$ with finite-dimensional graded pieces, write
\[
  \operatorname{Hilb}_C(z)=\sum_{n\ge0}(\dim_{\Fp}C_n)z^n
\]
for its Hilbert series.  The homogeneous sequence
$\rho_1,\ldots,\rho_m$ is called \emph{strongly free} if
\begin{equation}\label{eq:hilbert-series}
  \operatorname{Hilb}_{\mathcal A/(\rho_1,\ldots,\rho_m)}(z)
  =\frac{1}{1-dz+z^{h_1}+\cdots+z^{h_m}},
\end{equation}
where $(\rho_1,\ldots,\rho_m)$ denotes the two-sided ideal that they
generate.  Note that if $\rho_1,\ldots,\rho_m$ are linearly
independent, replacing them by another homogeneous basis of their
span changes neither side of \eqref{eq:hilbert-series}, since the
two-sided ideal is the same and any two homogeneous bases of a
graded subspace have the same number of elements in each degree.
Accordingly, a finite-dimensional graded subspace
$V\subseteq\mathcal A$ (that is, a subspace spanned by finitely
many homogeneous elements) is called \emph{strongly free} if
one, equivalently every, homogeneous basis of $V$ is strongly
free.
Both notions apply verbatim over a finite extension $k|\Fp$, with
$\mathcal A$ replaced by $\mathcal A\otimes_{\Fp}k$ and
dimensions taken over $k$.
The notion is due to Anick \cite{Anick}, who defined
strongly free sets as noncommutative analogues of regular sequences
and proved that, in the free algebra, they are exactly the sequences
satisfying \eqref{eq:hilbert-series} \cite[Theorem~2.6]{Anick}.
Taking this identity as the definition follows G\"artner
\cite[Definition~2.7]{Gaertner}.  The Lie-theoretic counterpart
underlies Labute's theory of mild pro-$p$ groups \cite{Labute2006}.  \begin{remark}[Strong freeness of the relation space]
\label{rem:span-strongly-free}
For the cubic initial forms $\rho_1,\ldots,\rho_d$ of
\cref{prop:T-is-relations}, the right-hand side of
\eqref{eq:hilbert-series} is
\[
  \frac{1}{1-dz+dz^3}
  =1+dz+d^2z^2+(d^3-d)z^3+\cdots.
\]
The two-sided ideal $(\rho_1,\ldots,\rho_d)$ is spanned by the
products $u\rho_\ell v$ with words $u,v$, of degrees
$|u|+3+|v|\ge3$.  In degrees up to two it therefore vanishes, so the
quotient has dimensions $1,d,d^2$ there and the series imposes no
condition on the $\rho_\ell$.  In degree three the ideal is the
span $\langle\rho_1,\ldots,\rho_d\rangle_{\Fp}$ itself, so
\[
  \dim_{\Fp}\bigl(\mathcal A/(\rho_1,\ldots,\rho_d)\bigr)_3
  =d^3-\dim_{\Fp}\langle\rho_1,\ldots,\rho_d\rangle_{\Fp},
\]
and comparison with the coefficient $d^3-d$ makes
$\dim_{\Fp}\langle\rho_1,\ldots,\rho_d\rangle_{\Fp}=d$, that
is, $\rk_{\Fp}T=d$ \eqref{eq:R3}, necessary for strong freeness,
since linearly dependent $\rho_\ell$ are never strongly free.  With the
basis independence noted after the definition, the sequence
$\rho_1,\ldots,\rho_d$ is strongly free precisely when
$\rk_{\Fp}T=d$ and the complete cubic initial relation space
$R_3=\langle\rho_1,\ldots,\rho_d\rangle_{\Fp}$ is strongly free.
\end{remark}

A \emph{multiplicative monomial order} is a total order on the words
in the letters $\mgen_1,\ldots,\mgen_d$ (the monomial basis of
$\mathcal A$) that is compatible with concatenation on both sides
and satisfies $1<\mgen_i$ for all $i$.  Fix such an order.  For a
nonzero $\rho\in\mathcal A$ we write $\HT(\rho)$ for its \emph{high
term}, the largest of the words occurring in $\rho$ with nonzero
coefficient.  A \emph{factor}
of a word is a contiguous segment of it, a \emph{prefix} an initial
and a \emph{suffix} a terminal factor.  A finite
sequence $w_1,\ldots,w_m$ of words is called \emph{combinatorially
free} if the following two conditions hold.
\begin{enumerate}[label=\textup{(\roman*)}]
\item\label{cf:subword} For $i\ne j$ there is no factorization
  $w_i=u\,w_j\,v$ with possibly empty words $u,v$.  In particular the
  words are pairwise distinct.
\item\label{cf:overlap} No proper nonempty prefix of any $w_i$ is a
  proper nonempty suffix of any $w_j$.  Here $i=j$ is allowed, so
  overlaps of a word with itself are excluded as well.
\end{enumerate}
Anick's criterion \cite[Theorem~3.2]{Anick} (proved there for
degree-lexicographic orders) says that if the high terms of a
sequence of
homogeneous polynomials are combinatorially free, then the sequence
is strongly free.  We use the formulation in
\cite[Theorem~3.5]{Gaertner}, stated for an arbitrary multiplicative
monomial order.  See also Forr\'e \cite{Forre} for a direct
treatment of strongly free sequences and pro-$p$ groups of cohomological
dimension two.

The mildness theorem below (\cref{prop:Gaertner}) draws the
group-theoretic conclusion $\cd G=2$ from cubic initial forms that
are strongly free over $\Fp$.  The transverse criterion of
\cref{subsec:transverse-rank-one} will establish strong freeness
after a coordinate change that exists in general only over a finite
extension $k|\Fp$.  The following lemma carries strong freeness from
$k$ down to $\Fp$ and thereby connects the two.

\begin{lemma}[Strong freeness under scalar extension]\label{lem:descent}
Let $k|\Fp$ be a finite extension, let $V\subseteq\mathcal A$ be
a finite-dimensional graded subspace, and let
$I\subseteq\mathcal A$ and $J\subseteq\mathcal A\otimes_{\Fp}k$
be the two-sided ideals generated by $V$ and by its scalar extension
$V\otimes_{\Fp}k$.
Then, for every $n\ge0$,
\[
  \dim_k\bigl((\mathcal A\otimes_{\Fp}k)/J\bigr)_n
  =\dim_{\Fp}\bigl(\mathcal A/I\bigr)_n ,
\]
so the two quotients have the same Hilbert series.  In particular,
$V$ is strongly free over $\Fp$ if and only if $V\otimes_{\Fp}k$
is strongly free over $k$.
\end{lemma}

\begin{proof}
Choose a homogeneous basis $\rho_1,\ldots,\rho_m$ of $V$, of
degrees $h_1,\ldots,h_m$.  It is also a $k$-basis of
$V\otimes_{\Fp}k$.  In every degree $n$, the graded piece $I_n$ is
spanned over $\Fp$ by the products $u\rho_\ell v$ in which $u,v$
are words with $\deg u+h_\ell+\deg v=n$, and $J_n$ is spanned
over $k$ by the same products.  Hence
$J_n=I_n\otimes_{\Fp}k$ inside $\mathcal A_n\otimes_{\Fp}k$,
and taking dimensions gives the displayed equality.  The right-hand
side of \eqref{eq:hilbert-series} is the same series over both
fields, so the identity defining strong freeness holds over $\Fp$
if and only if it holds over $k$.
\end{proof}

The mildness theorem is due to Labute \cite{Labute2006} for initial
forms lying in the free Lie subalgebra of $\mathcal A$ generated by
$\mgen_1,\ldots,\mgen_d$, and to G\"artner
\cite[Theorem~2.12]{Gaertner} for arbitrary homogeneous initial forms
in $\operatorname{gr}F$, in particular for forms containing the
restricted powers $\mgen_i^{[p]}$, which Labute's setting does not
reach (see \cite[Remarks~2.13]{Gaertner} for the comparison).  We
state the special case used in this paper.

\begin{proposition}[Mildness from strongly free cubic initial forms]
\label{prop:Gaertner}
Let $1\to R\to F\to G\to1$ be a minimal presentation with
$R\subseteq F_{(3)}$, and let $\rho_1,\ldots,\rho_d$ be the cubic
initial forms of a minimal set of relators.  If these forms are
strongly free, then $G$ is mild with respect to the Zassenhaus
filtration, and hence
\[
  \cd G=2.
\]
\end{proposition}

\begin{proof}
Since $R\subseteq F_{(3)}$, the relators have initial forms of
degree three, namely $\rho_1,\ldots,\rho_d$, viewed in
$\operatorname{gr}F$ through the identification recalled at the
beginning of this section.  Their strong freeness makes the
presentation strongly free with respect to the Zassenhaus filtration
in the sense of \cite[Definition~2.11]{Gaertner}, all weights being
equal to one.  The definition also asks the generator number to equal
$h^1(G)$, which holds by minimality.  Thus $G$ is mild, and
\cite[Theorem~2.12(i)]{Gaertner} gives $\cd G=2$.
\end{proof}

\begin{remark}[The role of \texorpdfstring{$\rk T=d$}{rank T = d}]
\label{rem:rank-three}
\Cref{prop:Gaertner} asks $\rho_1,\ldots,\rho_d$ to be the initial
forms of the relators, that is, to be nonzero.  If
$\rho_1,\ldots,\rho_d$ are proved strongly free, this
hypothesis needs no separate check: strong freeness forces
$\rk_{\Fp}T=d$ by itself (\cref{rem:span-strongly-free}), so
\cref{prop:Gaertner} applies to them as they stand.
The conclusion that $G$ is \emph{not} mild, by contrast, needs
the condition $\rk_{\Fp}T=d$: it is what makes the failure of
strong freeness of $R_3$ a property of $G$ rather than of a
chosen presentation.  Indeed, since $\dim_{\Fp}\overline R=d$,
the condition is equivalent, by \eqref{eq:R3}, to injectivity
of $\vartheta_3$, so every relation with nonzero class in
$\overline R$ has nonzero cubic initial form and hence
Zassenhaus initial degree exactly three.  The classes of any
minimal system of relators then form a basis of $\overline R$
by the normal-generation criterion \cite[(3.9.3)]{NSW}, their
initial forms form a basis of $R_3$, and a change of the free
generators acts on $\mathcal A$ by a graded algebra
automorphism preserving \eqref{eq:hilbert-series}.  The claim
follows, and this is what the non-mildness proofs for the $26$
fields of \cref{sec:example-p3} rest on.
\end{remark}

Strong freeness can also be decided, in both directions, by
counting words.  The words avoiding the high terms of the elements of
the relation ideal form a basis of the quotient, so
\eqref{eq:hilbert-series} becomes a degreewise comparison of word
counts, which terminating Gr\"obner-basis completions render finite.
This route settles $11\,731$ of the $12\,243$ mild fields of
rank three at $p=3$, and the field of rank four.  Its
statements are developed in \cref{app:word-counts}
(\cref{lem:normal-words}, \cref{cor:word-counts},
\cref{lem:diamond}), and carried out by \cref{alg:completion}.

\subsection{A transversality criterion}
\label{subsec:transverse-rank-one}

Let $k|\Fp$ be a finite extension.  Using the scalar extension defined
after \cref{lem:D-image-orthogonal}, we regard $D_x$ for
\[
  x\in H^1(G,\Fp)\otimes_{\Fp}k
\]
as a $k$-linear map
\[
  D_x\colon \Cl(K)[p]\otimes_{\Fp}k
  \longrightarrow
  \Cl(K)/p\otimes_{\Fp}k.
\]
Over $k$, the annihilator is defined by the same formula,
\[
  x^\perp=
  \{c\in\Cl(K)/p\otimes_{\Fp}k:x(c)=0\},
\]
now a $k$-subspace of $\Cl(K)/p\otimes_{\Fp}k$.

As in the discussion following \cref{lem:D-image-orthogonal}, the
field $k$ is kept out of the notation: $D_x$, $\Delta D(x,v)$,
$M(x,y,z)$, $\langle\cdot,\cdot\rangle_{\AV}$ and $x^\perp$
denote the scalar-extended objects.  No ambiguity arises, and no
identity needs a new proof.  Each of these maps is multilinear
(quadratic in the case of $x\mapsto D_x$), so it has a unique
$k$-multilinear extension, and an identity of such maps holds over
$k$ once it holds over $\Fp$, both sides being determined by their
values at $\Fp$-rational arguments.  The one exception is the
Bockstein map at $p=3$.  We write $\beta_k$ for the $k$-linear
extension of $\beta$, keeping the field in the notation because the
cubic diagonal $x\mapsto M(x,x,x)$ has the same vanishing locus
over $k$ but is not $k$-linear (see the discussion following
\cref{lem:diagonal-p3}).

\begin{definition}[Rank and cone conditions]
\label{def:cone-condition}
An element $x\in H^1(G,\Fp)\otimes_{\Fp}k$ satisfies the
\emph{rank condition} if $\rk D_x=d-2$, and the
\emph{cone condition} if $M(x,x,x)=0$.
\end{definition}

The name anticipates the Bockstein cone $C_\beta$ of
\cref{def:bockstein-cone}.  For nonzero $x$ the cone condition says
precisely that the class $[x]$ lies on $C_\beta$.

Since the quadratic scaling law gives $D_0=0$, the rank
condition forces $x\ne0$.

Under the rank and cone conditions the annihilator $x^\perp$ is a
hyperplane, of dimension $d-1$, the image $\operatorname{im}D_x$
has codimension one in $x^\perp$, since the argument of
\cref{lem:D-image-orthogonal} applies verbatim at $x$, and the
kernel $\ker D_x$ is two-dimensional.  These last two dimensions do
not depend on $d$; only the source of the map $\Theta_x$ below
grows with it.

\begin{lemma}\label{lem:Theta-well-defined}
For any
$x\in H^1(G,\Fp)\otimes_{\Fp}k$ satisfying the rank and cone
conditions, there is a well-defined $k$-linear map
\begin{equation}\label{eq:Theta}
  \Theta_x\colon 
  \bigl(H^1(G,\Fp)\otimes_{\Fp}k\bigr)/kx
  \longrightarrow
  \Hom_k\!\left(
    \ker D_x,\,x^\perp/\operatorname{im}D_x
  \right)
\end{equation}
given, for every $v\in H^1(G,\Fp)\otimes_{\Fp}k$ with class
$\overline v$, by
\[
  \Theta_x(\overline v)\colon
  \ker D_x\longrightarrow x^\perp/\operatorname{im}D_x,
  \qquad
  e\longmapsto
  \overline{-\Delta D(x,v)(e)}.
\]
\end{lemma}

\begin{proof}
If $e\in\ker D_x$, then, by \eqref{eq:MfromD} and \eqref{eq:Ddef},
\[
  x\bigl(\Delta D(x,v)(e)\bigr)
  =\langle M(x,x,v),e\rangle_{\AV}
  =v(D_x(e))=0,
\]
so $\Delta D(x,v)(e)\in x^\perp$.  Replacing $v$ by $v+\lambda x$
with $\lambda\in k$ gives, by the bilinearity of $\Delta D$ and
the quadratic scaling law \eqref{eq:quadratic},
\[
\begin{aligned}
  -\Delta D(x,v+\lambda x)
  &=-\Delta D(x,v)-\lambda\,\Delta D(x,x)\\
  &=-\Delta D(x,v)-\lambda\bigl(2D_x-D_{2x}\bigr)
  =-\Delta D(x,v)+2\lambda\,D_x.
\end{aligned}
\]
The map
\[
  -\Delta D(x,v)\colon\Cl(K)[p]\otimes_{\Fp}k\longrightarrow
  \Cl(K)/p\otimes_{\Fp}k
\]
therefore changes by $2\lambda\,D_x$, which vanishes on
$\ker D_x$.  Its restriction to $\ker D_x$ is unchanged, so
$\Theta_x(\overline v)$ depends only on the class $\overline v$.
Linearity follows from the bilinearity of $\Delta D$.
\end{proof}

\begin{definition}[Transverse element]\label{def:transverse-rank-one}
Let $x\in H^1(G,\Fp)\otimes_{\Fp}k$ satisfy the rank and cone
conditions.  We call $x$ \emph{transverse} if the map $\Theta_x$ of
\cref{lem:Theta-well-defined} is surjective.
\end{definition}

\begin{remark}[Transversality is projective]
\label{rem:transverse-projective}
Transversality is a property of the projective point $[x]$ alone.
Replacing the representative $x$ by $\lambda x$ with
$\lambda\in k^\times$ multiplies $D_x$ by $\lambda^2$.  Hence
the rank condition is unchanged.  The projective point $[x]$, and
with it the cone condition, is also unchanged.  The same
replacement leaves $\ker D_x$,
$\operatorname{im}D_x$, $x^\perp$ and the quotient
$\bigl(H^1(G,\Fp)\otimes_{\Fp}k\bigr)/kx$ unchanged.  By the
bilinearity of $\Delta D$, it rescales $\Theta_x$ by the nonzero
scalar $\lambda$.  Hence $x$ is transverse if and only if
$\lambda x$ is.
\end{remark}

The geometric meaning of the terminology is established in
\cref{prop:rank-drop-geometry}.  The picture behind the name is
spelled out in \cref{rem:two-equations} there.

\begin{theorem}[Transversality criterion]\label{thm:transverse-rank-one}
Let $K$ be an imaginary quadratic number field, let $p$ be an odd
prime, put $d=d_p\Cl(K)$, and assume $d\ge3$.  If there exist a
finite extension $k|\Fp$ and a transverse element
\[
  0\ne x\in H^1(G,\Fp)\otimes_{\Fp}k,
\]
then $G_{\varnothing}(K)(p)$ is mild with respect to the Zassenhaus
filtration, and hence
\[
  \cd G_{\varnothing}(K)(p)=2.
\]
\end{theorem}

\begin{proof}
We first work over $k$.  Any basis $\eta_1,\ldots,\eta_d$ of
$H^1(G,\Fp)\otimes_{\Fp}k$ and any basis $w_1,\ldots,w_d$ of
$\Cl(K)[p]\otimes_{\Fp}k$ together determine the following data.
Let $X_1,\ldots,X_d$ be the Magnus variables dual to
$\eta_1,\ldots,\eta_d$, ordered degree-lexicographically with
$X_d<\cdots<X_1$.  The $k$-linear extension of the isomorphism
\[
  \Cl(K)[p]\longrightarrow\overline R,
  \qquad
  e\longmapsto
  -\operatorname{tg}^*\bigl(\langle\,\cdot\,,e\rangle_{\AV}\bigr),
\]
sends $e_\ell$ to $\overline{r_\ell}$ by
\eqref{eq:negative-orientation}, and it carries the basis
$w_1,\ldots,w_d$ to a basis
$\overline{s_1},\ldots,\overline{s_d}$ of
$\overline R\otimes_{\Fp}k$.  Let $\rho_\ell$ be the cubic initial
form of $\overline{s_\ell}$.  Applying \eqref{eq:rho-from-T} over $k$, with
$\eta_1,\ldots,\eta_d$ and $w_1,\ldots,w_d$ in place of
$\chi_1,\ldots,\chi_d$ and $e_1,\ldots,e_d$, shows that the
coefficient of the word $X_aX_bX_c$ in $\rho_\ell$ is
$\langle M(\eta_a,\eta_b,\eta_c),w_\ell\rangle_{\AV}$.

The $2d$ largest words of degree three are, in decreasing
order,
\[
  X_1X_1X_1>\cdots>X_1X_1X_d>X_1X_2X_1>\cdots>X_1X_2X_d.
\]
We shall choose
the two bases, with $\eta_1=x$, so that the coefficients of
these words in $\rho_1,\ldots,\rho_d$ become
\[
\renewcommand{\arraystretch}{1.15}
\begin{array}{r|c|cc}
  & \rho_\ell,\ \ell\le d-2 & \rho_{d-1} & \rho_d\\
\hline
X_1X_1X_1 & 0 & 0 & 0\\
X_1X_1X_2 & 0 & 0 & 0\\
X_1X_1X_{j+2},\ 1\le j\le d-2 & \delta_{j\ell} & 0 & 0\\
\hline
X_1X_2X_1 & 0 & 0 & 0\\
\hdashline
X_1X_2X_m,\ 2\le m<k_1 & \ast & 0 & 0\\
X_1X_2X_{k_1} & \ast & 1 & 0\\
X_1X_2X_m,\ k_1<m<k_2 & \ast & \ast & 0\\
X_1X_2X_{k_2} & \ast & \ast & 1\\
X_1X_2X_m,\ k_2<m\le d & \ast & \ast & \ast
\end{array}
\]
for some indices $2\le k_1<k_2\le d$.  The entries marked
$\ast$ are arbitrary.  Suppose such bases have been chosen.
Degree-lexicographic order runs through the rows of the table
from top to bottom, and the first nonzero entry of each column
gives the high term of the corresponding $\rho_\ell$, so
\[
  \HT(\rho_j)=X_1X_1X_{j+2}\ (1\le j\le d-2),
  \qquad
  \HT(\rho_{d-1})=X_1X_2X_{k_1},
  \qquad
  \HT(\rho_d)=X_1X_2X_{k_2}.
\]
The $d$ high terms are pairwise distinct words of the same length,
so none occurs inside another and condition \ref{cf:subword} holds.
Each begins with $X_1$ and none ends with it, since $j+2\ge3$ and
$k_i\ge2$.  Their proper nonempty prefixes are $X_1$, $X_1X_1$ and
$X_1X_2$.  Their proper nonempty suffixes are single letters $X_m$
with $m\ge2$ and the two-letter words $X_1X_m$ with $m\ge3$ and
$X_2X_{k_i}$.  No suffix equals the prefix $X_1$.  A suffix
$X_1X_m$ differs from $X_1X_1$ and $X_1X_2$ in its second letter,
and a suffix $X_2X_{k_i}$ differs from both in its first
letter.
Condition \ref{cf:overlap} holds, and the high terms are
combinatorially free.  Anick's criterion therefore shows that
$\rho_1,\ldots,\rho_d$ are strongly free over $k$.

It remains to construct bases producing the table.

For the first block of the table, construct a basis
$c_1,\ldots,c_d$
of $\Cl(K)/p\otimes_{\Fp}k$ adapted to $x$: choose a basis
$c_3,\ldots,c_d$ of $\operatorname{im}D_x$, extend it by $c_2$ to
a basis of the hyperplane $x^\perp$, and choose $c_1$ with
$x(c_1)=1$.  Take
$\eta_1,\ldots,\eta_d$ dual to $c_1,\ldots,c_d$.  Then
$\eta_1=x$, since $x(c_1)=1$ and $x$ vanishes on
$c_2,\ldots,c_d\in x^\perp$, and
\[
  \operatorname{im}D_x=\langle c_3,\ldots,c_d\rangle,
  \qquad \eta_2(\operatorname{im}D_x)=0.
\]
Choose $w_1,\ldots,w_{d-2}\in\Cl(K)[p]\otimes k$ with
$D_x(w_j)=c_{j+2}$ for $j=1,\ldots,d-2$, and a basis
$w_{d-1},w_d$ of the two-dimensional $\ker D_x$.  Then
$w_1,\ldots,w_d$ is a basis of $\Cl(K)[p]\otimes_{\Fp}k$.  By \eqref{eq:Ddef} the coefficient of
$X_1X_1X_m$ in $\rho_\ell$ is
\[
  \langle M(x,x,\eta_m),w_\ell\rangle_{\AV}
  =\eta_m\bigl(D_x(w_\ell)\bigr).
\]
With $m$ as the row index and $\ell$ as the column index, the
first block of the table is therefore the matrix of $D_x$ with
respect to
$w_1,\ldots,w_d$ and $c_1,\ldots,c_d$.  Its first row vanishes
by the cone condition, which holds identically for $p>3$ and
reads $\beta_k(x)=0$ for $p=3$.  Its second row vanishes because
$\eta_2$ annihilates $\operatorname{im}D_x$, and its last two
columns because $w_{d-1},w_d\in\ker D_x$.  For $j\le d-2$ the column of
$\rho_j$ is the coordinate vector of $D_x(w_j)=c_{j+2}$, that is,
$\eta_m(c_{j+2})=\delta_{m,j+2}$.  This gives the first block of the table.

In the second block, the row of $X_1X_2X_1$ vanishes.  The
shuffle identities give $M(x,\eta_2,x)=-2M(x,x,\eta_2)$, so its
coefficients are scalar multiples of those of $X_1X_1X_2$ and
vanish with them.  To describe the remaining rows, associate
with each $v\in\{\eta_2,\ldots,\eta_d\}$ the functional
\[
  \varphi_v\colon\ker D_x\longrightarrow k,
  \qquad
  w\longmapsto\langle M(x,\eta_2,v),w\rangle_{\AV},
\]
so that the coefficients of $X_1X_2X_m$ in $\rho_{d-1}$ and
$\rho_d$ are $\varphi_{\eta_m}(w_{d-1})$ and
$\varphi_{\eta_m}(w_d)$.  For $w\in\ker D_x$,
\[
  \varphi_v(w)
  =\eta_2\bigl(\Delta D(x,v)(w)\bigr)
  =\eta_2\Bigl(\,\overline{\Delta D(x,v)(w)}\,\Bigr)
  =-\eta_2\bigl(\Theta_x(\overline v)(w)\bigr),
\]
where the first equality is \eqref{eq:MfromD}.  The second holds
because $\Delta D(x,v)(w)$ lies in $x^\perp$
(\cref{lem:Theta-well-defined}) and $\eta_2$ descends to
$x^\perp/\operatorname{im}D_x$, as
$\eta_2(\operatorname{im}D_x)=0$.  The third is
\eqref{eq:Theta}.  Hence
$\varphi_v=-\eta_2\circ\Theta_x(\overline v)$.  The functional induced by $\eta_2$ is an isomorphism
$x^\perp/\operatorname{im}D_x\to k$, since the class of $c_2$
spans the quotient and $\eta_2(c_2)=1$.
In the table, the values $\varphi_{\eta_m}(w_\ell)$ are the
entries below the dashed line in the last two columns.  Write them
as the $(d-1)\times2$ matrix
\[
  A=\bigl(\varphi_{\eta_m}(w_\ell)\bigr)
  _{2\le m\le d,\ \ell\in\{d-1,d\}}.
\]
Row $m$ of $A$ is the coordinate vector of $\varphi_{\eta_m}$ in
the basis of $(\ker D_x)^*$ dual to $w_{d-1},w_d$.  Since
$\overline{\eta_2},\ldots,\overline{\eta_d}$ is a basis of
$\bigl(H^1(G,\Fp)\otimes_{\Fp}k\bigr)/kx$ and $\Theta_x$ is
surjective, the functionals
$\varphi_{\eta_2},\ldots,\varphi_{\eta_d}$ span the
two-dimensional space $(\ker D_x)^*$.  The rows of $A$ therefore
span $k^2$, and $\rk A=2$.  A change of the basis $w_{d-1},w_d$
of $\ker D_x$ acts on $A$ by invertible column operations.
Such a change affects only the last two columns of the table,
and within them only the entries below the dashed line, because
the entries above vanish for every basis of $\ker D_x$.  Choose
the basis so that $A$ is in column echelon form.  There are then
indices $2\le k_1<k_2\le d$ satisfying
\[
\begin{gathered}
  A_{k_1,d-1}=1,\qquad A_{m,d-1}=0\ \text{for }m<k_1,\\
  A_{k_2,d}=1,\qquad A_{m,d}=0\ \text{for }m<k_2,\qquad
  A_{k_1,d}=0 .
\end{gathered}
\]
These identities give the second block of the table.  The high
terms are as announced, and $\rho_1,\ldots,\rho_d$ are strongly
free over $k$.

The strongly free sequence just obtained is defined only over
$k$ (the element $x$ need not be $\Fp$-rational), so we first
transport
strong freeness to the scalar extension of the original forms.
Recall from the beginning of \cref{sec:relations} that the
variables $\mgen_1,\ldots,\mgen_d$ of $\mathcal A$ are the
Magnus variables dual to the fixed basis $\chi_1,\ldots,\chi_d$
of $H^1(G,\Fp)$.  The cubic initial forms of
\cref{prop:T-is-relations} and the space $R_3$ are written in
these variables, and $X_1,\ldots,X_d$ play the same role over $k$
for the basis $\eta_1,\ldots,\eta_d$.
Write
\[
  \vartheta_3\otimes k\colon
  \overline R\otimes_{\Fp}k\longrightarrow
  \mathcal A_3\otimes_{\Fp}k
\]
for the scalar extension of $\vartheta_3$.  Write
\[
  \psi\colon
  \mathcal A\otimes_{\Fp}k\longrightarrow
  \mathcal A\otimes_{\Fp}k
\]
for the graded algebra automorphism induced by the change of
basis from $\chi_1,\ldots,\chi_d$ to $\eta_1,\ldots,\eta_d$,
which replaces each $\mgen_i$ by its expression in the variables
$X_1,\ldots,X_d$.  By construction
$\rho_\ell=\psi\bigl((\vartheta_3\otimes k)(\overline{s_\ell})\bigr)$.
Since $\overline{s_1},\ldots,\overline{s_d}$ is a basis of
$\overline R\otimes_{\Fp}k$ and
$R_3=\operatorname{im}(\vartheta_3)$ \eqref{eq:R3},
\[
  \langle\rho_1,\ldots,\rho_d\rangle_k
  =\psi\bigl(R_3\otimes_{\Fp}k\bigr).
\]
Since strong freeness depends only on the span
(\cref{rem:span-strongly-free}) and is preserved by the graded
automorphism $\psi$, the space $R_3\otimes_{\Fp}k$ is strongly
free.  Its dimension is $d$, because the $\rho_\ell$, being
strongly free, are linearly independent.  Hence $\rk_{\Fp}T=d$
\eqref{eq:R3}.

Finally, \cref{lem:descent} allows strong freeness to descend
from $k$ to $\Fp$.  Since $\rk_{\Fp}T=d$, the cubic initial forms of the relators
form a basis of $R_3$.  By \cref{rem:span-strongly-free}, they
are strongly free, and \cref{prop:Gaertner} gives the asserted
mildness and cohomological dimension.
\end{proof}

For the criterion to be usable, transversality must be verifiable by a
finite computation.  The next definition encodes $\Theta_x$ in explicit
coordinates.  The resulting rank test, a determinant test for $d=3$,
is the form of the criterion used in all later computations.

\begin{definition}[Adapted data and transversality matrix]
\label{def:transversality-matrix}
Let $x\in H^1(G,\Fp)\otimes_{\Fp}k$ satisfy the rank and cone
conditions.  An
\emph{adapted choice datum} for $x$ is a tuple
\[
  \mathcal B=(f_2,f_3;v_1,\ldots,v_{d-1};y)
\]
such that $f_2,f_3$ is a basis of $\ker D_x$, the classes
$\overline{v_1},\ldots,\overline{v_{d-1}}$ form a basis of
$\bigl(H^1(G,\Fp)\otimes_{\Fp} k\bigr)/kx$, and
$y\in H^1(G,\Fp)\otimes_{\Fp}k$ annihilates
$\operatorname{im}D_x$ but not all of $x^\perp$.  The
\emph{transversality matrix} associated with $\mathcal B$ is
\begin{equation}\label{eq:Bx}
  B_{x,\mathcal B}=
  \Bigl(-y\bigl(\Delta D(x,v_i)(f_j)\bigr)\Bigr)
  _{1\le i\le d-1,\ j\in\{2,3\}}\in M_{(d-1)\times2}(k).
\end{equation}
\end{definition}

For the bases constructed in the proof of
\cref{thm:transverse-rank-one}, the tuple
$(w_{d-1},w_d;\ \eta_2,\ldots,\eta_d;\ \eta_2)$ is an adapted
choice datum.  Up to sign and a shift of the index sets, the
matrix $A$ in that proof is the associated transversality
matrix.

\begin{corollary}[Matrix form of transversality]\label{cor:determinant-form}
Let $d\ge3$ and let $x\in H^1(G,\Fp)\otimes_{\Fp}k$
satisfy the rank and cone conditions.  Then $x$ is transverse
if and only if $\rk B_{x,\mathcal B}=2$ for one, equivalently
every, adapted choice datum $\mathcal B$.  For $d=3$ the matrix is
square and the condition reads $\det B_{x,\mathcal B}\ne0$.  In this case
$G_{\varnothing}(K)(p)$ is mild, and hence has cohomological
dimension $2$.
\end{corollary}

\begin{proof}
Adapted choice data exist.  Under the rank condition the space
$\ker D_x$ is two-dimensional and
$\bigl(H^1(G,\Fp)\otimes_{\Fp}k\bigr)/kx$ has dimension $d-1$,
and since $\operatorname{im}D_x$ has codimension one in
$x^\perp$, some $y\in H^1(G,\Fp)\otimes_{\Fp}k$
annihilates $\operatorname{im}D_x$ but not $x^\perp$.

Fix an adapted choice datum
$\mathcal B=(f_2,f_3;v_1,\ldots,v_{d-1};y)$ and write
\[
  \overline y\colon x^\perp/\operatorname{im}D_x\longrightarrow k
\]
for the functional induced by $y$, an isomorphism of one-dimensional
$k$-vector spaces.  Since $\overline{v_1},\ldots,\overline{v_{d-1}}$
is a basis of the source of $\Theta_x$, $f_2,f_3$ is a basis of
$\ker D_x$, and $\overline y$ is an isomorphism, the map
$\Theta_x$ is surjective if and only if the matrix
\[
  \Bigl(\overline y\bigl(\Theta_x(\overline{v_i})(f_j)\bigr)
  \Bigr)_{1\le i\le d-1,\,j\in\{2,3\}}
\]
has rank two.  By \eqref{eq:Theta} its $(i,j)$ entry is
$-y\bigl(\Delta D(x,v_i)(f_j)\bigr)$, the corresponding entry of
$B_{x,\mathcal B}$ \eqref{eq:Bx}.  Hence, for every adapted choice
datum $\mathcal B$, the element $x$ is transverse if and only if
$\rk B_{x,\mathcal B}=2$.  As adapted choice data exist, the
condition $\rk B_{x,\mathcal B}=2$ holds for one adapted choice
datum if and only if it holds for every adapted choice datum.  The final assertion follows from
\cref{thm:transverse-rank-one}.
\end{proof}

\begin{remark}[Direct use of the complete relation space]\label{rem:direct-anick}
The transversality criterion is sufficient, not necessary.
When it does not apply, no further arithmetic is required.  The
$d(d+1)/2$ secondary norm operators already reconstruct the full
relation matrix $T$, and strong freeness can still be established from $T$
directly, by an invertible change of variables and relation
basis whose high terms are combinatorially free, or by the
terminating word counts developed in \cref{app:word-counts} and
discussed after \cref{rem:rank-three}.  The negative direction
rests on \cref{rem:rank-three}.
\end{remark}

\subsection{The norm-degeneracy scheme and the geometric mildness
criterion}
\label{subsec:rank-drop-geometry}

We now interpret transversality geometrically, applying
\cref{lem:rank-one-tangent} and \cref{lem:syzygy-smooth} of
Appendix~\ref{app:rank-one-lemmas} to the secondary norm
family; \emph{reduced} and \emph{isolated} for closed points
are used in the sense fixed there.
Besides the quadratic dependence of the family on $x$ and its
scalar extension \eqref{eq:scalar-extended-D}, two arithmetic
facts about this family are used:
\begin{enumerate}[label=\textup{(I\arabic*)}]
\item\label{input:image}
  $\operatorname{im}D_x\subseteq x^\perp$ for every $x$ satisfying
  the cone condition.  For $p>3$ this is every $x$
  (\cref{lem:D-image-orthogonal}), and for $p=3$ every $x$ with
  $\beta_k(x)=0$ (\cref{lem:diagonal-p3}).  In both cases the
  statement persists under scalar extension because the underlying
  shuffle identities are multilinear.
\item\label{input:derivative}
  $\Delta D(x,v)(e)\in x^\perp$ for every $v$ and every
  $e\in\ker D_x$ (\cref{lem:Theta-well-defined}).
\end{enumerate}
No further property of the arithmetic situation enters this
subsection.

The operators $D_x$ assemble into a matrix of quadratic forms as
follows.  Three bases enter, in different roles.  The basis
$\chi_1,\ldots,\chi_d$ of $H^1(G,\Fp)$ is the one fixed at the
beginning of this section.  It indexes the family
\eqref{eq:minimal-family} and provides the coordinates
$x_1,\ldots,x_d$ of a character $x=\sum_ix_i\chi_i\in H^1(G,\Fp)$.  The bases of
$\Cl(K)[p]$ and of $\Cl(K)/p$ are auxiliary.  Fix one of each and
represent each operator $D_x$ by its matrix
$\mathcal D_x\in M_d(\Fp)$ with respect to these two bases.
\Cref{def:rank-drop-scheme} will not depend on this choice.  Read
in these bases, the reconstruction formula
\eqref{eq:D-reconstruct} becomes the identity of matrices
\begin{equation}\label{eq:D-matrix}
  \mathcal D_x=\sum_i x_i^2\,\mathcal D_{\chi_i}
   +\sum_{i<j}x_ix_j
    \bigl(\mathcal D_{\chi_i+\chi_j}-\mathcal D_{\chi_i}
      -\mathcal D_{\chi_j}\bigr),
\end{equation}
and entry by entry it exhibits every entry of $\mathcal D_x$ as a
homogeneous quadratic form in the coordinates $x_1,\ldots,x_d$.  By
\eqref{eq:scalar-extended-D} the same forms compute $\mathcal D_x$
over every extension.  The $(d-1)\times(d-1)$ minors of
this $d\times d$ matrix of quadratic forms are homogeneous of degree
$2(d-1)$ and therefore define a closed subscheme of projective
space.

\begin{definition}[Norm-degeneracy scheme]\label{def:rank-drop-scheme}
The \emph{norm-degeneracy scheme}
\[
  \Sigma_D\subseteq\mathbb P\bigl(H^1(G,\Fp)\bigr)
\]
is the closed subscheme defined by the $(d-1)\times(d-1)$ minors of
this matrix of quadratic forms, the locus where the secondary norm
operators degenerate.  Since invertible row and column
operations do not change the ideal generated by the
$(d-1)\times(d-1)$ minors
(Cauchy--Binet), this subscheme does not depend on the chosen bases of
$\Cl(K)[p]$ and $\Cl(K)/p$.  Changing the basis of $H^1(G,\Fp)$
acts by a linear substitution of the coordinates $x_1,\ldots,x_d$ and
defines the same subscheme of $\mathbb P\bigl(H^1(G,\Fp)\bigr)$.
It is an invariant of the field $K$ itself.
For every finite extension $k|\Fp$ its
$k$-points are
\[
  \Sigma_D(k)
  =\bigl\{[x]:0\ne x\in H^1(G,\Fp)\otimes_{\Fp}k,\ 
    \rk D_x\le d-2\bigr\},
\]
where $[x]\in\mathbb P\bigl(H^1(G,\Fp)\bigr)(k)$ denotes the
line $kx$ spanned by $x$.  Here and throughout we identify $X(k)$
with $\bigl(X\otimes_{\Fp}k\bigr)(k)$ for an $\Fp$-scheme $X$.
Thus $[x]$ is also a $k$-rational point of
$\mathbb P\bigl(H^1(G,\Fp)\bigr)\otimes_{\Fp}k$, and when
$\rk D_x\le d-2$ a $k$-rational point of $\Sigma_D\otimes_{\Fp}k$.
\end{definition}

For $p>3$ one has $\operatorname{im}D_x\subseteq x^\perp$ for
every $x$, so $\rk D_x\le d-1$, and $\rk D_x=d-1$ exactly when
$[x]\notin\Sigma_D(k)$.  For $p=3$ the same bound holds for the
$x$ with $\beta_k(x)=0$, the locus of the Bockstein cone
introduced next.  Without this condition the rank can equal $d$.

\begin{definition}[Bockstein cone]\label{def:bockstein-cone}
For $p=3$, the \emph{Bockstein cone} is the reduced linear subspace
\[
  C_\beta=\mathbb P(\ker\beta)
  \subseteq\mathbb P\bigl(H^1(G,\mathbb F_3)\bigr),
\]
with $\beta$ the linear diagonal of
\cref{lem:diagonal-p3}\ref{diag:linear}.  The name refers to
\cref{lem:diagonal-p3}\ref{diag:bockstein}.  By
\cref{lem:diagonal-p3}\ref{diag:dual}, the codimension of
$C_\beta$ equals the
number of cyclic factors of order exactly $3$ in the $3$-part of
$\Cl(K)$.  For $p>3$ we set
$C_\beta=\mathbb P\bigl(H^1(G,\Fp)\bigr)$, the whole space.  The
diagonal vanishes identically there, so for every odd $p$ the cone
condition of \cref{subsec:transverse-rank-one} on a character $x$ is
the condition $[x]\in C_\beta$.
\end{definition}

\begin{remark}[The reduced cone]\label{rem:reduced-cone}
At $p=3$, over a finite extension $k$ the diagonal evaluations are
the \emph{cubes} of linear forms:
$\langle M(x,x,x),e_\ell\rangle_{\AV}=(\sum_iB_{\ell i}x_i)^{3}$ by
the scalar extension of \eqref{eq:diagonal-cube}.  The equations
$\langle M(x,x,x),e_\ell\rangle_{\AV}=0$ therefore define a
Frobenius thickening of the linear subspace, nonreduced exactly
when $0<\rk\beta<d$.  Only the
underlying points and the tangent spaces of $\Sigma_D$ enter the
arguments below, so we work throughout with the reduced linear cone
$C_\beta$.
\end{remark}

\Cref{prop:rank-drop-geometry} below and the remark following it
use a standard chart, which we set up first.  Fix a finite extension
$k|\Fp$ and a nonzero $x\in H^1(G,\Fp)\otimes_{\Fp}k$.
Coordinates refer here and below to the basis
$\chi_1,\ldots,\chi_d$ of $H^1(G,\Fp)$ fixed at the beginning of
this section.  Choose a coordinate $\ell$ with $\ell(x)\ne0$,
and let $U$ be the standard affine chart $\{\ell\ne0\}$ of
$\mathbb P\bigl(H^1(G,\Fp)\bigr)\otimes_{\Fp}k$; it contains
$[x]$, and each of its points has a unique representative with
$\ell=1$, its \emph{normalized representative}.  Replacing $x$ by $\ell(x)^{-1}x$ does not change
$[x]$, and by \cref{rem:transverse-projective} it preserves the
rank and cone conditions and transversality.  We therefore assume
$\ell(x)=1$: the given $x$ is the normalized representative
of $[x]$.  On
$U$ the scalar-extended secondary norm family
\eqref{eq:scalar-extended-D} defines
a morphism
\[
  \Phi\colon U\longrightarrow\Hom_k\bigl(\Cl(K)[p]\otimes_{\Fp}k,\
  \Cl(K)/p\otimes_{\Fp}k\bigr).
\]
Here, as in \cref{app:rank-one-lemmas}, the notation $\Hom_k$
stands for the affine $k$-scheme whose $k$-points form the vector
space used so far.  Let $R$ be a $k$-algebra.  Every $R$-point of $U$
has a unique representative
\[
  x'=\sum_i x'_i\,\chi_i\in H^1(G,\Fp)\otimes_{\Fp}R
  \qquad\text{with}\quad x'_\ell=1.
\]
To this point $\Phi$ assigns the $R$-linear map
\[
  D_{x'}\colon\Cl(K)[p]\otimes_{\Fp}R\longrightarrow
  \Cl(K)/p\otimes_{\Fp}R
\]
given by
\[
  D_{x'}=\tfrac12\,b_R(x',x'),
\]
the definition \eqref{eq:scalar-extended-D} with the polarization
$b$ now extended to the $k$-algebra $R$.  Since $b$ is
$\Fp$-bilinear and $2$ is invertible, the assignment is defined for
every $R$ and compatible with morphisms of $k$-algebras.  After the
choice of the auxiliary bases of $\Cl(K)[p]$ and $\Cl(K)/p$, the
target of $\Phi$ is identified with the affine space
$\mathbb A^{d^2}_k$ of $d\times d$ matrices.  Under this
identification $\Phi$ is represented by the matrix of quadratic
forms of \cref{def:rank-drop-scheme}, with $\ell$ set equal
to $1$.  By construction the restriction of $\Sigma_D\otimes k$
to $U$ is the closed subscheme $Z\subseteq U$ defined by the
$(d-1)\times(d-1)$ minors of this matrix, as in
\cref{lem:rank-one-tangent}.

\begin{lemma}[The evaluation family]\label{lem:evaluation-family}
Let $p>3$, and let
\[
  \lambda\colon U\longrightarrow
  \Hom_k\bigl(\Cl(K)/p\otimes_{\Fp}k,\ k\bigr)
\]
be the morphism that assigns to every $R$-point $x'$ of $U$,
represented as above, the $R$-linear evaluation functional
\[
  \lambda(x')\colon\Cl(K)/p\otimes_{\Fp}R\longrightarrow R,
  \qquad
  e\longmapsto x'(e)=\sum_i x'_i\,\chi_i(e).
\]
Then $\lambda(x')\circ\Phi(x')=0$ for every $k$-algebra $R$
and every $R$-point $x'$ of $U$.  At the $k$-point $[x]$ the
morphism specializes to
\[
  \lambda([x])\colon\Cl(K)/p\otimes_{\Fp}k\longrightarrow k,
  \qquad
  e\longmapsto x(e),
\]
which is nonzero with kernel $x^\perp$.
\end{lemma}

\begin{proof}
The identity
$\lambda(x')\circ\Phi(x')=0$ reads
$x'\bigl(D_{x'}(e)\bigr)=0$ for every
$e\in\Cl(K)[p]\otimes_{\Fp}R$.  For $R=k$ this is
\cref{lem:D-image-orthogonal} after scalar extension.  The same
proof gives it for every $R$: the identities \eqref{eq:MfromD}
and \eqref{eq:diagonal-triple} underlying it are multilinear, and
$3$ is invertible in $R$.  This is where $p>3$ enters.  The value $\lambda([x])$ is the
stated functional.  It is
nonzero since $x\ne0$ and the evaluation pairing is perfect, and
its kernel is $x^\perp$ by the definition of the annihilator.
\end{proof}

\begin{proposition}[Geometric interpretation of transversality]\label{prop:rank-drop-geometry}
Let $k|\Fp$ be a finite extension, assume $d\ge3$, and let
$x\in H^1(G,\Fp)\otimes_{\Fp}k$ satisfy the rank and cone
conditions.  For $p>3$ the cone condition is vacuous.
\begin{enumerate}[label=\textup{(\alph*)}]
\item\label{rd:tangent} The tangent space of
  $\Sigma_D\otimes_{\Fp}k$ at $[x]$ is $\ker\Theta_x$.  Its
  dimension is at least $d-3$, and it equals $d-3$ exactly when
  $x$ is transverse.
\item\label{rd:smooth} For $p>3$, the element $x$ is transverse if
  and only if $\Sigma_D\otimes_{\Fp}k$ is smooth over $k$ of
  dimension $d-3$ at $[x]$.
\item\label{rd:isolated} For $d=3$ and every odd $p$, the element
  $x$ is transverse if and only if $[x]$ is a reduced isolated
  point of $\Sigma_D\otimes_{\Fp}k$.
\end{enumerate}
\end{proposition}

\begin{proof}
With $U$, $\Phi$ and $\lambda$ as constructed above for this $k$
and $x$, we apply \cref{lem:rank-one-tangent} at the $k$-point
$u=[x]$ of $U$ and compute its map
\[
\begin{aligned}
  d^\perp\Phi_{[x]}\colon T_{[x]}U&\longrightarrow
  \Hom_k\bigl(\ker D_x,\ \coker D_x\bigr),\\
  v&\longmapsto\pi_{D_x}\circ\bigl(d\Phi_{[x]}(v)\bigr)\big|_{\ker D_x}.
\end{aligned}
\]
In the notation of that lemma,
$\varphi=\Phi([x])=D_x$, of rank $d-2$ by hypothesis.  In the
standard description of tangent vectors by dual numbers
\cite[Proposition~6.7 and Remark~6.8]{GoertzWedhorn},
$T_{[x]}U$ is the set of $k[\varepsilon]$-points of $U$ reducing
to $[x]$ at $\varepsilon=0$.  Every such point has a unique
normalized representative $x+\varepsilon v$ with a direction
vector $v$ satisfying $\ell(v)=0$.  These normalized
representatives give an isomorphism of $k$-vector spaces
\[
  \bigl\{v\in H^1(G,\Fp)\otimes_{\Fp}k:\ell(v)=0\bigr\}
  \longrightarrow T_{[x]}U,
  \qquad v\longmapsto[x+\varepsilon v].
\]
Since $\ell(x)=1$, the hyperplane $\{\ell=0\}$ is a complement
of the line $kx$ in $H^1(G,\Fp)\otimes_{\Fp}k$, so passing to
classes gives a second isomorphism
\[
  \bigl\{v\in H^1(G,\Fp)\otimes_{\Fp}k:\ell(v)=0\bigr\}
  \longrightarrow
  \bigl(H^1(G,\Fp)\otimes_{\Fp}k\bigr)/kx,
  \qquad v\longmapsto\overline v.
\]
Together they identify
\[
  \bigl(H^1(G,\Fp)\otimes_{\Fp}k\bigr)/kx\;\simeq\;T_{[x]}U,
\]
the left-hand side being the source of $\Theta_x$
\eqref{eq:Theta}.  For the $k[\varepsilon]$-point
$[x+\varepsilon v]$, expanding
$D_{x+\varepsilon v}=\tfrac12\,b_{k[\varepsilon]}
(x+\varepsilon v,\,x+\varepsilon v)$ bilinearly gives
\[
  D_{x+\varepsilon v}
  =D_x-\varepsilon\,\Delta D(x,v)+\varepsilon^2D_v
  =D_x-\varepsilon\,\Delta D(x,v),
\]
so
\begin{equation}\label{eq:dPhi-polarization}
  d\Phi_{[x]}(\overline v)=-\Delta D(x,v),
\end{equation}
and
\[
  d^\perp\Phi_{[x]}(\overline v)
  =\pi_{D_x}\circ\bigl(-\Delta D(x,v)\bigr)\big|_{\ker D_x}.
\]

Now insert the two arithmetic ingredients, both available at the
selected point since $x$ satisfies the cone condition.
By \ref{input:image} the inclusion
$x^\perp\subseteq\Cl(K)/p\otimes_{\Fp}k$ induces an injection
\[
  \iota\colon \Hom_k\bigl(\ker D_x,\ x^\perp/\operatorname{im}D_x\bigr)
  \lhook\joinrel\longrightarrow
  \Hom_k\bigl(\ker D_x,\ \coker D_x\bigr),
\]
and by \ref{input:derivative} the map $d^\perp\Phi_{[x]}$ takes values in the
image of $\iota$.  With the representative $v$ chosen above,
$\Theta_x(\overline v)(e)$ is by definition the class of
$-\Delta D(x,v)(e)$ in $x^\perp/\operatorname{im}D_x$, and
$d^\perp\Phi_{[x]}(\overline v)(e)$ is the class of the same element in
$\coker D_x$.  Hence
\[
  d^\perp\Phi_{[x]}=\iota\circ\Theta_x,
\]
with $\Theta_x$ as in \cref{def:transverse-rank-one}.  Since $\iota$ is
injective, $\ker d^\perp\Phi_{[x]}=\ker\Theta_x$.

Part \ref{rd:tangent} now follows from
\cref{lem:rank-one-tangent}(a): the tangent space of $Z$ at $[x]$ is
$\ker d^\perp\Phi_{[x]}=\ker\Theta_x$.  The source of $\Theta_x$
has dimension $d-1$ and its target dimension two, so
$\dim\ker\Theta_x=(d-1)-\rk\Theta_x\ge d-3$, with equality
exactly when $\Theta_x$ is surjective, that is, when $x$ is
transverse.

For part \ref{rd:isolated} let $d=3$.  Source and target of
$\Theta_x$ are then both two-dimensional, so $\Theta_x$ is surjective
if and only if it is injective, that is, if and only if
$d^\perp\Phi_{[x]}$ is injective.
\Cref{lem:rank-one-tangent}(b) identifies this with $[x]$ being a
reduced isolated point of $Z$.

For part \ref{rd:smooth} let $p>3$.  If $Z$ is smooth of dimension
$d-3$ at $[x]$, then its tangent space at $[x]$ has dimension
$d-3$, and $x$ is transverse by part \ref{rd:tangent}.  For the converse suppose $x$ is transverse.  Apply
\cref{lem:syzygy-smooth} to $\Phi$ and the evaluation family
$\lambda$ of \cref{lem:evaluation-family} on the chart $U$; its
hypotheses hold at $u=[x]$, since the chart is an affine space
and hence smooth over $k$, the rank of $\varphi=D_x$ is $d-2$,
and $\lambda([x])$ is nonzero by \cref{lem:evaluation-family}.
By \cref{lem:syzygy-smooth}\ref{sy:smooth}, which here yields
smoothness of codimension $c=2$, that is, of dimension $d-3$,
it then suffices to show that the map $\theta_{[x]}$ of
\cref{lem:syzygy-smooth} is surjective, and we do so by
comparing it with $\Theta_x$.  As $\ker\lambda([x])=x^\perp$, the two maps fit into the
triangle
\[
\begin{tikzcd}[column sep=large]
T_{[x]}U \arrow[r, "\theta_{[x]}"]
& \Hom_k\bigl(\ker D_x,\ x^\perp/\operatorname{im}D_x\bigr) \\
\bigl(H^1(G,\Fp)\otimes_{\Fp}k\bigr)/kx
\arrow[u, "\simeq"] \arrow[ur, "\Theta_x"'] &
\end{tikzcd}
\]
with the identification above on the left.  The triangle commutes.  For $e\in\ker D_x$,
\[
  \theta_{[x]}(\overline v)(e)
  =\overline{d\Phi_{[x]}(\overline v)(e)}
  =\overline{-\Delta D(x,v)(e)}
  =\Theta_x(\overline v)(e)
\]
in $x^\perp/\operatorname{im}D_x$, the first equality by the
description of $\theta_u$ in \cref{lem:syzygy-smooth}, the
second by \eqref{eq:dPhi-polarization}, and the third by
\eqref{eq:Theta}.
Hence $\theta_{[x]}$ is surjective, since $\Theta_x$ is.
\end{proof}

\begin{remark}[Two equations and the local picture]
\label{rem:two-equations}
Let $p>3$ and let $[x]$ be a $k$-point of $\Sigma_D\otimes_{\Fp}k$
with $\rk D_x=d-2$.  By \cref{lem:evaluation-family} the
hypotheses of \cref{lem:syzygy-smooth} hold at $[x]$.
Part~\ref{sy:present} therefore gives, near $[x]$, a
presentation of $\Sigma_D\otimes_{\Fp}k$ by two equations
rather than the $d^2$ minors of \cref{def:rank-drop-scheme}.  Two consequences follow.  First, every
irreducible component of $\Sigma_D\otimes_{\Fp}k$ through $[x]$
is, near $[x]$, an irreducible component of the vanishing scheme
of the two equations.  It therefore has codimension at most
two and dimension at least $d-3$
\cite[Propositions~5.30 and~5.35]{GoertzWedhorn}.  Second, by
part~\ref{sy:present}, the element $x$ is transverse exactly when
the two equations define hypersurfaces smooth at $[x]$ that cross
with distinct tangent hyperplanes.  This is the picture behind
the name.  At $d=3$ the two hypersurfaces are curves in the plane
$\mathbb P\bigl(H^1(G,\Fp)\bigr)\otimes_{\Fp}k$.  If they
cross with distinct tangents, the tangent lines meet only in
zero.  The intersection then has tangent space zero at $[x]$ and
is, near $[x]$, the reduced isolated point of
part~\ref{rd:isolated}.
\end{remark}

\begin{lemma}[Points of small rank are never smooth]
\label{lem:rank-zero-excluded}
Let $k|\Fp$ be a finite extension, possibly trivial, let $\xi$ be a
closed point of $\Sigma_D\otimes_{\Fp}k$, and let
$x\in H^1(G,\Fp)\otimes_{\Fp}\kappa(\xi)$ be a representative of
$\xi$.  If $\rk D_x\le d-3$ (a condition independent of this
choice by \cref{rem:transverse-projective}), then the tangent space
of $\Sigma_D\otimes_{\Fp}k$ at $\xi$ has the full dimension $d-1$.
In particular $\xi$ is not a smooth point of dimension $d-3$, and
for $d=3$ it is not a reduced isolated point.
\end{lemma}

\begin{proof}
We choose an affine chart of projective space containing $\xi$
and dehomogenize; the entries of the matrix of quadratic forms
of \cref{def:rank-drop-scheme} then become regular functions on
the chart, forming a $d\times d$ matrix $Q$.  We work in the
local ring of projective space at $\xi$, with maximal ideal
$\mathfrak m$.  Every $(d-2)\times(d-2)$ minor of $Q$ vanishes
at $\xi$, its value there being the corresponding minor of the
matrix of $D_x$, of rank at most $d-3$.  A
$(d-1)\times(d-1)$ minor of $Q$ is a polynomial in the entries
of the chosen $(d-1)\times(d-1)$ submatrix, and its partial
derivatives with respect to these entries are, up to sign, the
$(d-2)\times(d-2)$ minors of that submatrix, by cofactor
expansion.  Its differential at $\xi$ therefore vanishes, and
the minor lies in $\mathfrak m^2$.
Hence the ideal $I$ defining $\Sigma_D\otimes_{\Fp}k$ satisfies
$I\subseteq\mathfrak m^2$, and the cotangent space at $\xi$ is
\[
  \mathfrak m/(\mathfrak m^2+I)=\mathfrak m/\mathfrak m^2,
\]
of dimension $d-1$.  No smooth point of dimension $d-3$ and no
reduced isolated point has that cotangent space.
\end{proof}

\begin{corollary}[Geometric mildness criterion]
\label{cor:geometric-mildness}
Let $K$ be an imaginary quadratic number field, let $p$ be an odd
prime, put $d=d_p\Cl(K)$, and assume $d\ge3$.
\begin{enumerate}[label=\textup{(\alph*)}]
\item\label{gm:three} For $d=3$ the following are equivalent:
\begin{enumerate}[label=\textup{(\roman*)}]
\item there exist a finite extension $k|\Fp$ and a transverse element $x\in H^1(G,\Fp)\otimes k$;
\item $\Sigma_D$ has a reduced isolated closed point lying on
      $C_\beta$.
\end{enumerate}
For $k$ one may take the residue field of such a point, and for
$p>3$ the condition on $C_\beta$ is vacuous.
\item\label{gm:general} For $p>3$ and arbitrary $d\ge3$ the
following are equivalent:
\begin{enumerate}[label=\textup{(\roman*)}]
\item there exist a finite extension $k|\Fp$ and a transverse element $x\in H^1(G,\Fp)\otimes k$;
\item $\Sigma_D$ has a closed point at which it is smooth of
      dimension $d-3$.
\end{enumerate}
For $k$ one may again take the residue field of such a point.
\end{enumerate}
If these conditions hold, then $G_{\varnothing}(K)(p)$ is mild with
respect to the Zassenhaus filtration, and hence
$\cd G_{\varnothing}(K)(p)=2$.
\end{corollary}

\begin{proof}
For part \ref{gm:three} let $d=3$.  Assume (i) and let $\xi$ be the
image of $[x]$ in
$\Sigma_D$.  By \cref{prop:rank-drop-geometry}\ref{rd:isolated},
transversality of
$x$ makes $[x]$ a reduced isolated point of $\Sigma_D\otimes k$.
By \cref{lem:reduced-isolated-descent}\ref{bc:local}, $\xi$ is
a reduced isolated closed point of $\Sigma_D$.  It lies on
$C_\beta$: the transverse element $x$ satisfies the cone
condition by definition, so the $k$-point $[x]$ factors through
the reduced closed subscheme $C_\beta$, and its image $\xi$
lies on $C_\beta$.  This proves
\textup{(i)}$\Rightarrow$\textup{(ii)}.

Conversely, assume (ii) and let $\xi$ be a reduced isolated closed point on
$C_\beta$.  Put $k=\kappa(\xi)$ and let
$x\in H^1(G,\Fp)\otimes_{\Fp}k$ be a representative of $\xi$.
It exists because \cref{lem:reduced-isolated-descent}\ref{bc:points}
provides the $k$-rational point $[x]$ of $\Sigma_D$, and over the
field $k$ every point of $\mathbb P\bigl(H^1(G,\Fp)\bigr)(k)$ is
the class of a nonzero vector.
The defining equations of $C_\beta$ lie in the prime ideal of
$\xi$, hence vanish at $x$, so $x$ satisfies the cone
condition.  On $\Sigma_D$ the rank of $D_x$ is at most one, and
rank zero is excluded by \cref{lem:rank-zero-excluded}.  Hence
$\rk D_x=1=d-2$.  The representative defines a $k$-rational point
$[x]$ of $\Sigma_D\otimes k$ lying over $\xi$, and $[x]$ is
reduced and isolated by
\cref{lem:reduced-isolated-descent}\ref{bc:local}.  Now
\cref{prop:rank-drop-geometry}\ref{rd:isolated} makes $x$
transverse.  This proves
\textup{(ii)}$\Rightarrow$\textup{(i)}.

For part \ref{gm:general} let $p>3$, so that the cone condition is
vacuous.  If $x$ is transverse, then
\cref{prop:rank-drop-geometry}\ref{rd:smooth} shows that
$\Sigma_D\otimes_{\Fp}k$ is smooth of dimension $d-3$ at the
$k$-rational point $[x]$.  By
\cref{lem:reduced-isolated-descent}\ref{bc:local}, $\Sigma_D$ is
smooth of dimension $d-3$ at the image $\xi$ of $[x]$, which is
(ii).  Conversely, let $\xi$ be such a point and put $k=\kappa(\xi)$.
Choose a representative $x$ of $\xi$ as in the proof of
\ref{gm:three}.  Then $[x]$ is a $k$-rational point of
$\Sigma_D\otimes k$ over $\xi$, and $\Sigma_D\otimes k$ is
smooth of dimension $d-3$ at $[x]$ by
\cref{lem:reduced-isolated-descent}\ref{bc:local}.  On
$\Sigma_D\otimes k$ the rank of $D_x$ is at most $d-2$.  By \cref{lem:rank-zero-excluded}, rank at most $d-3$ would
force the tangent space at $[x]$ to have the full dimension
$d-1$, incompatible with smoothness of dimension $d-3$.  Hence $\rk D_x=d-2$.  The tangent space at the point has dimension $d-3$ by
smoothness, so $x$ is transverse by
\cref{prop:rank-drop-geometry}\ref{rd:tangent}.

The final assertion is \cref{thm:transverse-rank-one}.
\end{proof}

\begin{remark}[Effectivity]
\label{rem:effectivity}
Suppose $d=3$, and that $\Sigma_D$ is zero-dimensional.  This is
the case for every field computed in
\cref{subsec:uniform-computation}: the exact Gr\"obner-basis
computations recorded in \cite{Repository} show that the quotient
of $\Fp[x_1,x_2,x_3]$ by the nine minors has Krull dimension at
most one, while a curve inside $\Sigma_D$ would make the affine
vanishing locus two-dimensional.
Then $\Sigma_D$ has only finitely many closed points, each defined over
an explicitly computable finite extension of $\Fp$, and the transversality criterion becomes a finite procedure.  One determines the
closed points of $\Sigma_D$ from the nine quartic minors and applies
the rank test of \cref{cor:determinant-form}, a determinant test
for $d=3$, to each point at which the rank is one.
Every closed point has residue degree at most
$\dim_{\Fp}H^0(\Sigma_D,\mathcal O_{\Sigma_D})$, so explicit
a priori bounds on the degrees to be inspected could be obtained
by standard methods, for instance from a B\'ezout-type estimate
for two general linear combinations of the quartic minors.  We
have not pursued this, since in our examples the search succeeds
directly, in degree at most $6$.  At $p=3$ the same procedure applies
after intersecting with the cone.  The closed points of $\Sigma_D$ on
the line or point $C_\beta$, whose form is determined from the
$3$-class group by \cref{lem:diagonal-p3}, are inspected in the same
way.  There the cone itself bounds the degrees a priori: a cone
point is rational, and on a cone line the minors restrict to
binary quartic forms, so the residue degree is at most four.  The
inspection carried out in \cref{subsec:uniform-computation} is
therefore complete.
\end{remark}

\section{Two algorithms: secondary-norm evaluation and
transversality search}\label{sec:computation}

Let $K$ be, as in \cref{sec:setup}, an imaginary quadratic field,
$p$ an odd prime, and assume $d_p\Cl(K)=3$.  The following data
are fixed once and for all:
\begin{itemize}
\item generators $g_1,\ldots,g_r$ of $\Cl(K)$ whose orders
  $n_1,\ldots,n_r$ are the elementary divisors; exactly three of
  the orders are divisible by $p$, say those of
  $g_{i_1},g_{i_2},g_{i_3}$;
\item the basis $e_j=(n_{i_j}/p)\,g_{i_j}$, $j=1,2,3$, of
  $\Cl(K)[p]$;
\item the basis of $\Cl(K)/p$ formed by the images of
  $g_{i_1},g_{i_2},g_{i_3}$, as in
  \cref{subsec:rank-drop-geometry}.
\end{itemize}
The computation behind
\cref{sec:results} consists of two algorithms.
\Cref{alg:matrices} evaluates the six operators of the family
\eqref{eq:minimal-family},
\[
  D_x\colon\Cl(K)[p]\longrightarrow\Cl(K)/p,
\]
through norm witnesses, and returns their matrices
$\mathcal D_x$ with respect to these bases.  The six characters
are $\chi_1$, $\chi_2$, $\chi_3$ and the sums $\chi_1{+}\chi_2$,
$\chi_1{+}\chi_3$, $\chi_2{+}\chi_3$; at $p=3$, the sum
$\chi_1{+}\chi_2{+}\chi_3$ replaces $\chi_2{+}\chi_3$
(\cref{subsec:computation-p3}).  \Cref{alg:transverse} then
decides the transversality criterion on the reconstructed family.
\Cref{subsec:construct-verify} is formulated for $p>3$;
\cref{subsec:computation-p3} records the changes at $p=3$.  Their
validity rests on \cref{prop:AC}, its $p=3$ variant
\cref{prop:AC-p3}, and \cref{lem:norm-sign}.  How class fields and
norm witnesses are found does not enter.  Candidates may come from
any method, because every property the argument uses is verified
inside the algorithm (\cref{app:verification}).  The searches actually used are
described in \cref{subsec:searches} and recorded, with programs
and per-field commands, in \cite{Repository}.

\subsection{Computing the six matrices}\label{subsec:construct-verify}

For a basis class $e_j$ represented by
$(a'_j,J_j)$ as in \eqref{eq:pair-aJ}, the two equations of
\cref{def:norm-witness} for the norm witness
$(t_{x,j},I'_{x,j})$ read
\begin{align}
  I_{x,j}'^{\,(1-\sigma_x)^2}\,(t_{x,j})\,J_j\mathcal O_{L_x}
   &=\mathcal O_{L_x},\label{eq:checkAC1}\\
  N_x(t_{x,j})&=a'_j.\label{eq:checkAC2}
\end{align}

\begin{algo}[Evaluation of the secondary norm family]
\label{alg:matrices}\leavevmode
\begin{description}[font=\normalfont\itshape,topsep=2pt,itemsep=0pt]
\item[Input:] the field $K$, the fixed bases of $\Cl(K)[p]$ and
  $\Cl(K)/p$, and the six characters listed above.
\item[Output:] the matrices $\mathcal D_x$ of the six input
  characters, every entry backed by a verified norm witness.
\end{description}
\begin{enumerate}[label=\textup{Step~\arabic*.},leftmargin=*]
\item For each character $x$, construct the everywhere unramified
  cyclic degree-$p$ field $L_x$ attached to $\ker x$ by class
  field theory, and normalize a generator $\sigma_x$ of
  $\Gal(L_x| K)$ to the character $x$ through the Artin symbols of
  the fixed class-group generators, as in
  \eqref{eq:sigma-character}.
\item For each basis class $e_j$, represented by $(a'_j,J_j)$ as
  in \eqref{eq:pair-aJ}, produce a candidate for the norm witness
  $(t_{x,j},I'_{x,j})$.  This step is nondeterministic.  Any method
  may propose the candidate; in practice $t_{x,j}$ arrives in
  factored form, as a product of powers of algebraic numbers.
\item Verify \eqref{eq:checkAC1} as an equality of fractional
  ideals.  Then decide \eqref{eq:checkAC2} through
  \cref{lem:norm-sign}.  The residue of $N_x(t_{x,j})/a'_j$ modulo
  a prime ideal of odd residue characteristic is computed factor
  by factor and fixes the sign, so $t_{x,j}$ is never expanded
  (\cref{subsec:unconditionality}).  If a check fails, return to
  Step~2.
\item Compute $N_x(I'_{x,j})$, take its class, reduce it modulo
  $p$, and record the coordinate vector as the $j$-th column of
  $\mathcal D_x$.
\end{enumerate}
\end{algo}

If \cref{alg:matrices} terminates, its output is correct.  Step~3
verifies exactly the hypotheses of \cref{prop:AC}, so each
recorded column is the value the proposition asserts.  Termination is not claimed.  The expensive
class-group and unit-group computations in the degree-$2p$ fields
$L_x$ all sit in the nondeterministic Step~2.  For the fields of
\cref{sec:results}, candidates were found for every character and
every basis class, by the routes of \cref{subsec:searches}.

\begin{lemma}\label{lem:norm-sign}
Let $(a',J)$ represent an element of $\Cl(K)[p]$ as in
\eqref{eq:pair-aJ}, and suppose that
$t\in L_x^\times$ and a fractional ideal $I'$ of $L_x$ satisfy
\eqref{eq:checkAC1}.  Then $N_x(t)=\pm a'$.  For a prime ideal
$\mathfrak q$ of $K$ of odd residue characteristic, $N_x(t)=a'$
if and only if $N_x(t)/a'\equiv1\bmod{\mathfrak q}$.
\end{lemma}

\begin{proof}
The ideal norm is invariant under $\Gal(L_x| K)$, so
\[
  N_x\bigl(I^{(1-\sigma_x)^2}\bigr)
  =N_x(I)\,N_x(I)^{-2}\,N_x(I)=\mathcal O_K
\]
for every fractional ideal $I$ of $L_x$, while
$N_x\bigl((t)\bigr)=\bigl(N_x(t)\bigr)$ and
$N_x\bigl(J\mathcal O_{L_x}\bigr)=J^{\,p}$.  Taking ideal norms in
\eqref{eq:checkAC1} therefore gives
$\bigl(N_x(t)\bigr)J^{\,p}=\mathcal O_K$, and comparison with
$(a')\,J^{\,p}=\mathcal O_K$ yields the equality of principal ideals
$\bigl(N_x(t)\bigr)=(a')$; hence $N_x(t)/a'$ is a unit of $K$.  Since
$p| h_K$, the field $K$ is neither $\mathbb Q(i)$ nor
$\mathbb Q(\sqrt{-3})$, the only imaginary quadratic number fields
with units other than $\pm1$; hence $\mathcal O_K^\times=\{\pm1\}$.
Finally, $N_x(t)/a'=\pm1$, and $1\not\equiv-1\bmod{\mathfrak q}$
because the residue characteristic of $\mathfrak q$ is odd.
\end{proof}

\subsection{The transversality computation}\label{subsec:rank-one-computation}

After \cref{alg:matrices}, no further number-field computation is
required.  Everything below is linear algebra over finite fields,
carried out on the matrices $\mathcal D_x$ of
\cref{subsec:rank-drop-geometry}.

\begin{algo}[Transversality search]\label{alg:transverse}\leavevmode
\begin{description}[font=\normalfont\itshape,topsep=2pt,itemsep=0pt]
\item[Input:] the six matrices of \cref{alg:matrices} and a degree
  bound $m_0\ge1$.
\item[Output:] a transverse element
  $x\in H^1(G,\Fp)\otimes_{\Fp}k$ for an extension $k| \Fp$ of
  degree at most $m_0$ (so that $G_{\varnothing}(K)(p)$ is mild
  by \cref{thm:transverse-rank-one}), or the report that no
  closed point of residue degree at most $m_0$ is transverse,
  which leaves mildness open.
\end{description}
\begin{enumerate}[label=\textup{Step~\arabic*.},leftmargin=*]
\item List the closed points $[x]$ of the norm-degeneracy scheme
  $\Sigma_D$ (\cref{def:rank-drop-scheme}) of residue degree at
  most $m_0$; at $p=3$, keep only the points on the Bockstein
  cone $C_\beta$ (\cref{def:bockstein-cone}), which at $p>3$ is
  the whole plane.
\item For each listed point, compute $\mathcal D_x$ by
  substituting the coordinates of $x$ into the matrix identity
  \eqref{eq:D-matrix}, and keep the points with
  $\rk\mathcal D_x=1$.
\item For each of these, choose an adapted datum
  $\mathcal B=(e_2,e_3;v_1,v_2;y)$ (a basis $e_2,e_3$ of
  $\ker\mathcal D_x$, representatives $v_1,v_2$ of a basis modulo
  $kx$, and a functional $y$ vanishing on
  $\operatorname{im}\mathcal D_x$) and form the transversality
  matrix
  \begin{equation}\label{eq:Bx-computation}
  B_{x,\mathcal B}=
  \begin{pmatrix}
  -y\bigl(\Delta\mathcal D(x,v_1)\,e_2\bigr)&
  -y\bigl(\Delta\mathcal D(x,v_1)\,e_3\bigr)\\
  -y\bigl(\Delta\mathcal D(x,v_2)\,e_2\bigr)&
  -y\bigl(\Delta\mathcal D(x,v_2)\,e_3\bigr)
  \end{pmatrix},
  \quad
  \Delta\mathcal D(x,v)=\mathcal D_x+\mathcal D_v-\mathcal D_{x+v}.
  \end{equation}
\item If $\det B_{x,\mathcal B}\ne0$ for a listed point, return
  the transverse element $x$.
\end{enumerate}
\end{algo}

A returned element is transverse by \cref{cor:determinant-form},
and the mildness statement of the output is
\cref{thm:transverse-rank-one}.  The
negative report is a proved statement about $\Sigma_D$, not about
the tower group.  Mildness stays open.  Termination is clear, since
each degree carries finitely many closed points.  How much the
negative report covers depends on $m_0$.  When $\Sigma_D$ is
zero-dimensional, the residue degrees of all its closed points are
bounded a priori (\cref{rem:effectivity}), and any $m_0$ at least
this bound makes the list in Step~1 exhaustive.

\subsection{The case \texorpdfstring{$p=3$}{p=3}}
\label{subsec:computation-p3}

Two points change at $p=3$.
\begin{itemize}
\item The identities of \cref{subsec:construct-verify} are
  verified unchanged, but by \cref{prop:AC-p3} the ideal $J_j$
  of the representing pair enters the output class.  Step~4 of
  \cref{alg:matrices} records
  $D_x(e_j)=\bigl[N_x(I'_{x,j})\,J_j\bigr]$.
\item The reconstruction runs through the space of
  shuffle-symmetric forms.  By \cref{lem:four-evaluations}, four
  evaluations, at the basis characters and at
  $\chi_1{+}\chi_2{+}\chi_3$, already determine the complete
  product.  This is why the input list of \cref{alg:matrices}
  carries $\chi_1{+}\chi_2{+}\chi_3$ in place of
  $\chi_2{+}\chi_3$.  The two remaining sums are computed as
  well.  They overdetermine the matrix $T$ of \eqref{eq:T}, and the shuffle identities
  \eqref{eq:shuffle} are checked on the assembled values
  (\cref{subsec:certificate-verification}).  At $p>3$ the identity
  $D_{2x}=4D_x$ distinguishes a character from the other
  generators of its line, since $D_{\lambda x}=\lambda^2D_x$ and
  $\lambda^2\ne1$ for some $\lambda\in\Fp^{\times}$; at $p=3$
  every $\lambda^2=1$, so $D_{\lambda x}=D_x$ carries no such
  information, and the shuffle checks on the overdetermined
  values provide the consistency test instead.
\end{itemize}

\section{Mild class tower groups for
\texorpdfstring{$|D_K|<2^{30}$}{|D\_K|<2\^{}30}}\label{sec:results}

This section works in the range $|D_K|<2^{30}$, which contains
$326\,378\,404$ imaginary quadratic number fields.  The tables of
Mosunov--Jacobson \cite{MosunovJacobson} determine the class
groups of all imaginary quadratic number fields with
$|D_K|<2^{40}$ unconditionally, so the population of the range,
and of each stratum below, is read off from them.  The odd primes
$p$ for which fields of $p$-class rank three exist in the range
are exactly $3$, $5$, and $7$.  There are $204$ such fields at
$p=5$, three at $p=7$, and $12\,749$ at $p=3$, and this section
treats every one of these $12\,956$ fields.  A single field
of the range has $p$-class rank exceeding three; it appears in
\cref{subsec:exceptional-fields}.  Wherever fields are
enumerated in this section, they appear in order of increasing
absolute discriminant.

Three routes lead from the secondary norm operators of a field to
mildness of its tower group, and all three read the same input, the
six matrices $\mathcal D_x$ returned by \cref{alg:matrices}.
\begin{itemize}
\item The \emph{transversality criterion} decides whether the
  norm-degeneracy scheme $\Sigma_D$ has a reduced isolated closed
  point on the Bockstein cone $C_\beta$.  Such a point makes
  $G_{\varnothing}(K)(p)$ mild by \cref{cor:geometric-mildness},
  while its absence in the residue degrees inspected is a
  statement about $\Sigma_D$ and leaves mildness undecided.
\item An \emph{Anick witness}, an invertible change of variables
  after which the three cubic initial forms have combinatorially
  free high terms, proves the cubic initial relation space
  strongly free, and \cref{prop:Gaertner} then makes the group
  mild (\cref{rem:direct-anick}).  The notion is available over
  any coefficient field; the searches of
  \cref{subsec:uniform-computation} ran over $\Fthree$ and over
  $\Fnine$, and every witness they found is rational.
\item The \emph{word counts} of \cref{alg:completion} concern the
  two-sided ideal $I$ generated by the three cubic initial forms,
  in the given coordinates or after an invertible linear
  substitution over $\Fp$ or a finite extension of $\Fp$.  Strong
  freeness is invariant under both (the remark following
  \cref{prop:Gaertner} and \cref{lem:descent}).  They compare the
  number of $I$-normal words in each degree, the words with no high
  term of the completed basis of $I$ as a subword, with the
  coefficients of $1/(1-3z+3z^3)$.  Once the completion terminates,
  agreement in all degrees is decided by finitely many comparisons
  and proves strong freeness.  A deviation in a degree whose overlaps
  have been resolved disproves it, and a completion that does not
  terminate says nothing.  For a field with $\rk_{\Fp}T=3$ the
  failure is one of the group (\cref{rem:rank-three}), so this is
  the only route that decides against mildness as well as for it.
\end{itemize}
Each field is counted once in \cref{tab:range}, for the first
route in this order that decides it, so the columns are disjoint
by convention and not by exclusion.  This is a counting
convention, not a description of the runs.  The criterion ran on
every field, and the witness searches ran on every field the
criterion left undecided, so the last two route columns are
disjoint in fact, since the exhausted searches show that no
word-count field has an Anick witness.  A criterion field was not passed
on, and the proof of the criterion itself exhibits an Anick
witness over the residue field of its point
(\cref{thm:transverse-rank-one}).

\begin{table}[!htp]\centering\small
\begin{tabular}{l r r r r r r r}
\toprule
$p$ & fields & criterion & Anick & word count & mild & not mild & undecided\\
\midrule
$5$ & $204$ & $203$ & $0$ & $1$ & $204$ & $0$ & $0$\\
$7$ & $3$ & $3$ & $0$ & $0$ & $3$ & $0$ & $0$\\
$3$ & $12\,749$ & $505$ & $7$ & $11\,731$ & $12\,243$ & $26$ & $480$\\
\bottomrule
\end{tabular}
\caption{The $12\,956$ imaginary quadratic number fields of
$p$-class rank three with $|D_K|<2^{30}$, by prime and by the route
that decides them.  The criterion column counts the fields with a
reduced isolated closed point of $\Sigma_D$ on $C_\beta$, the Anick
column the seven fields settled by a rational Anick witness
(\cref{subsec:exceptional-fields}), and
the word-count column those settled by a terminating completion
(\cref{alg:completion}).  Each field is counted for the first route that decides it, so
the three columns are disjoint and add up to the mild column.  The last
two columns count the fields whose tower group is provably not mild,
and those for which none of the three routes yields a mildness
verdict.}
\label{tab:range}
\end{table}

All three routes read stored data of one kind.  For every one of the
$12\,956$ fields, a certificate in the sense of
\cref{def:certificate} records the six input characters of
\cref{alg:matrices} (with $\chi_1{+}\chi_2{+}\chi_3$ in place of
$\chi_2{+}\chi_3$ at $p=3$) against the three basis classes
$e_1,e_2,e_3$, and all $12\,956$ certificates verify in the sense of
\cref{def:verifies}, in runs on two machines with identical outcomes
on all $233\,208$ recorded values.  By \cref{prop:soundness} the six
assembled matrices are then the matrices of the secondary norm
operators of these characters and determine the complete Massey
product, at $p=3$ through the four evaluations of
\cref{lem:four-evaluations}, and with \cref{prop:T-is-relations}
they determine the cubic initial relation space of
$G_{\varnothing}(K)(p)$.  The checks, the sense in which the results
are unconditional, and the public procedure that repeats them are the
subject of \cref{app:verification}.  Every computation in the remainder of this section reads the
verified matrices and nothing else, and each is exact and
finite.  The complete
tables of the range are Parts~W5 and~W6 of \cite{Repository}.

Mildness is the one property proved field by field.  The other
assertions of the theorems below follow from it uniformly.

\begin{lemma}[Mildness and the tower group]\label{lem:mild-tower}
Let $K$ be an imaginary quadratic number field, let $p$ be an odd
prime, assume $d_p\Cl(K)\ge3$, and write $G=G_{\varnothing}(K)(p)$.
If $G$ is mild with respect to the Zassenhaus filtration, then $G$
is an infinite FAb pro-$p$ group with $\cd G=2$.
\end{lemma}

\begin{proof}
A mild pro-$p$ group has cohomological dimension two
\cite[Theorem~2.12(i)]{Gaertner}.  The group $G$ is FAb by \cref{prop:FAb}.
Finally $\dim_{\Fp}H^1(G,\Fp)=d_p\Cl(K)\ge3$ by \eqref{eq:ranks}, so $G$ is
nontrivial, while a nontrivial finite pro-$p$ group has infinite
mod-$p$ cohomological dimension.  With $\cd G=2$ this makes $G$
infinite.
\end{proof}

\noindent
The infinitude in \cref{lem:mild-tower} is not new.  At $d=r=3$
the quadratic Golod--Shafarevich inequality $r>d^{2}/4$ is
satisfied and excludes nothing.  All relations sit in degree at
least three, however, so the weighted bound applies, and
Koch--Venkov derived the infinitude from it for every imaginary
quadratic field of $p$-class rank at least three
\cite{VenkovKoch}.  What the route through mildness adds is the
cohomological dimension.

\subsection{The computations}\label{subsec:uniform-computation}

\paragraph{The criterion: closed points and transversality.}
For every field of the range, the closed points of $\Sigma_D$
with residue degree up to a bound fixed in advance were
determined from the verified matrices alone.
The bound is three at $p=3$ and at $p=7$, and six at $p=5$.  These bounds are search choices, not the a priori bound
of \cref{rem:effectivity}.  \Cref{alg:transverse} then evaluated the transversality
matrix
\eqref{eq:Bx-computation} at each of them.  This is the exact linear
algebra of \cref{subsec:rank-one-computation}, with no further
number-field computation.  Wherever the search returns a transverse
element, \cref{thm:transverse-rank-one} applies.

\paragraph{The criterion at \texorpdfstring{$p=3$}{p=3}: the
Bockstein cone.}
At $p=3$ the criterion can apply only at closed points lying on the
Bockstein cone $C_\beta$ (\cref{def:bockstein-cone}).  The
cone, unlike $\Sigma_D$, involves no secondary norm operators.
By \cref{lem:diagonal-p3}\ref{diag:dual} the rank of the natural map
$\delta\colon\Cl(K)[3]\to\Cl(K)/3$ is the number of cyclic factors
of order exactly $3$ in the $3$-part of $\Cl(K)$.  It partitions
the $12\,749$ fields into four strata, on which $C_\beta$ is the
whole plane ($1$ field), a projective line ($220$), a single
rational point ($5183$), or empty ($7345$), and it decides in
advance where the criterion can apply at all.

The strata are inspected as follows.
\begin{itemize}
\item For the single field whose cone is the whole plane, the
  search in bounded residue degree returned a transverse point.
\item On a cone line there is a dichotomy.  Either some
  $2\times2$ minor of the matrix of quadratic forms restricts to
  a nonzero binary quartic form on the line.  Then the closed
  points of $\Sigma_D$ on the line lie among the at most four
  zeros of that form, and their residue degree is at most four.
  Or every minor vanishes identically on the line.  Then the line
  lies inside $\Sigma_D$, no point of it is an isolated point of
  $\Sigma_D$, and the criterion cannot apply along it.  For each
  of the $220$ cone lines of the range the first case occurs, so
  every closed point of $\Sigma_D$ on a line has residue degree
  at most four.  The search at $p=3$ stops at degree three.  A
  separate pass over the degree-four points of the lines returned
  no transverse point beyond those already found.
\item In the stratum whose cone is a single rational point, that
  point is the only candidate, and the search inspects it.
\item In the stratum whose cone is empty there is no candidate
  at all.
\end{itemize}
For every field at $p=3$ not counted in the criterion column of
\cref{tab:strata-p3}, the whole cone has thereby been inspected,
so $\Sigma_D$ has no reduced isolated closed point on $C_\beta$,
and by \cref{cor:geometric-mildness} no transverse element exists
over any extension of $\Fthree$.  The reports are Part~W6 of
\cite{Repository}.

\paragraph{Anick witnesses.}
This route and the next were invoked only at $p=3$; at $p=5$ and
$p=7$ the criterion had decided every field except the single
$p=5$ field recorded undecided (\cref{sec:example}).  A search over
all $11\,232$ elements of $\mathrm{GL}_3(\Fthree)$ decides
whether a field has a rational Anick witness, and a search over
all $663\,390$ ordered projective bases of $\Fnine$ decides the
same over $\Fnine$.  By \cref{rem:flag} the high terms depend only
on the flag, so the $52$ flags over $\Fthree$ and the $910$ flags
over $\Fnine$ would have sufficed.  The searches were run over the
full groups, which is exhaustive a fortiori.  Both searches ran on every field the
criterion left undecided, and both runs are recorded in
\cite{Repository}.  They agree: exactly the seven fields of
\cref{tab:range} carry an Anick witness, every witness is
already rational, and for every other number field the exhausted
searches prove that no witness exists over $\Fthree$ or
$\Fnine$.  For the single field with $\rk_{\Fp}T=2$ the records
note instead that linearly dependent forms admit no witness
(\cref{subsec:exceptional-fields}).

\paragraph{Word counts in the cubic relation space.}
On the fields the criterion left undecided we compute in the
cubic relation space itself, by \cref{alg:completion} applied to the
three cubic initial forms reconstructed from the verified matrices,
first in the given coordinates with the degree bound thirteen
throughout, then, where that leaves the field undecided, after a
change of variables.  The two directions of
the verdict rest on different parts of \cref{lem:diamond}.

A completion that terminates below the bound exhibits a finite set
of words generating the ideal of high terms, and this set determines
the number of $I$-normal words in every degree
(\cref{lem:diamond}\ref{dl:full} and
\cref{cor:computing-normal-words}\ref{wc:finite-basis}), so that the
comparison with $1/(1-3z+3z^3)$ becomes the finite one of
\cref{alg:completion}.  Agreement of the counts with the
coefficients of that series proves the three cubic initial forms
strongly free by \cref{cor:word-counts}, and \cref{prop:Gaertner}
applies.  Agreement in degree three also forces the three forms to be
linearly independent, so that they span the complete cubic
initial relation space (\cref{rem:span-strongly-free}).

The negative direction needs no terminating completion.  In every
degree whose overlaps have been resolved, the word count is the
exact dimension of that graded piece
(\cref{lem:diamond}\ref{dl:truncated}), so a coefficient deviating
there from $1/(1-3z+3z^3)$ disproves strong freeness by
\cref{cor:word-counts}.  A completion abandoned at the bound whose
counts agree throughout says nothing about the higher degrees and
leaves the field undecided.  Every strong-freeness verdict carries two independent
completions, and the pairs agree on all $379$ verdicts: $352$
strongly free and $27$ not (\cref{subsec:software}).  Of the
$27$, the $26$ fields of \cref{tab:nonmild-p3} have
$\rk_{\Fp}T=3$ and their groups are not mild; the one field with
$\rk_{\Fp}T=2$ is taken up in \cref{subsec:exceptional-fields}.

On the $11\,859$ fields that the completion in the given
coordinates left undecided we repeat it after an invertible linear
substitution, running through the flags of \cref{rem:flag} over
$\mathbb F_{3^e}$ for $e\le8$ with the degree bound nine or eleven,
and stop at the first terminating completion.  This succeeds on
$11\,379$ fields,
$9\,975$ over $\Fthree$ itself and $1\,404$ over an extension of
degree at most eight.  Part~W7 of \cite{Repository} records for each
of them the substitution, the coefficient field and the high terms of
the completed basis, each recomputed by a second, independent
implementation.  For the remaining $480$ fields no flag inspected
terminates below the bound, a statement about the search, not about
the field.

\subsection{The fields at \texorpdfstring{$p=5$}{p=5}}\label{sec:example}

At $p=5$ the criterion decides $203$ of the $204$ fields.  A
transverse element exists over $\Ffive$ itself for $151$ of them,
and over an extension of degree between two and six for a further
$52$, so the freedom of a nontrivial residue field in
\cref{thm:transverse-rank-one} is needed.  The one remaining field
fails the hypotheses of the criterion for a provable reason,
recorded after the theorem.  The smallest absolute discriminant among the $204$ fields at
$p=5$, $-11203620$, was identified by Buell as the first imaginary
quadratic discriminant for which the $5$-Sylow subgroup of the class
group has rank three \cite{Buell}.  His computation gives
$\Cl(K)\simeq(\mathbb Z/10\mathbb Z)^3$, which our computation
reproduces.  It is the principal example of this paper.  Its
complete data are Parts~W1 and~W2 of \cite{Repository}, with the
transversality computation carried out in full in Part~W2, and the
closed points of its norm-degeneracy scheme are Part~W3.

\begin{theorem}\label{thm:mainexample}
Let $K$ be an imaginary quadratic number field with $5$-class rank
three and $|D_K|<2^{30}$.  Then the
$5$-class tower group
\[
  G_{\varnothing}(K)(5)=\Gal(K_{\varnothing}(5)| K)
\]
is an infinite mild FAb pro-$5$ group.  In particular,
\[
  \cd G_{\varnothing}(K)(5)=2.
\]
\end{theorem}

\begin{proof}
Of the $204$ fields, $203$ carry a transverse element.  For each of
them Part~W5 of \cite{Repository} records such an element over an
extension of $\Ffive$ of degree at most six, recomputed from the
verified matrices by \cref{alg:transverse}, so
$G_{\varnothing}(K)(5)$ is mild by \cref{thm:transverse-rank-one}.
For the remaining field, $D_K=-781922404$, Part~W7 of \cite{Repository}
records a change of variables over $\Ftwentyfive$ after which the
completion of \cref{alg:completion} terminates and the word counts
agree with $1/(1-3z+3z^3)$, so the cubic initial forms are strongly
free and \cref{prop:Gaertner} makes $G_{\varnothing}(K)(5)$ mild.
In each case \cref{lem:mild-tower} gives the remaining assertions.
\end{proof}

The field with $D_K=-781922404$, that is,
$K=\mathbb Q(\sqrt{-195480601})$, is the one field at $p=5$ to which
the criterion provably cannot apply.  Its norm-degeneracy scheme
consists of a single non-reduced closed point of degree two, the
chart-wise radical computation recorded in \cite{Repository} showing
that there are no further closed points.  Hence $\Sigma_D$ has no
reduced isolated closed point, so by \cref{cor:geometric-mildness} no
transverse element exists over any extension of $\Ffive$ and
\cref{thm:transverse-rank-one} cannot apply.  The word counts of
\cref{alg:completion} in the given coordinates leave the field
undecided as well.  After a change of variables over $\Ftwentyfive$
the completion terminates, and the field is mild
(\cref{thm:mainexample}).

\subsection{The fields at \texorpdfstring{$p=7$}{p=7}}\label{sec:example-p7}

At $p=7$ the range contains exactly three fields, and the procedure
of \cref{sec:computation} applies to them without modification.  The
fields are
\[
\begin{aligned}
  K_1&=\mathbb Q(\sqrt{-501510767}), &
  \Cl(K_1)&\simeq\mathbb Z/378\times\mathbb Z/7\times\mathbb Z/7,\\
  K_2&=\mathbb Q(\sqrt{-648153647}), &
  \Cl(K_2)&\simeq\mathbb Z/294\times\mathbb Z/7\times\mathbb Z/7,\\
  K_3&=\mathbb Q(\sqrt{-931506071}), &
  \Cl(K_3)&\simeq\mathbb Z/840\times\mathbb Z/7\times\mathbb Z/7.
\end{aligned}
\]
The criterion proves all three mild.  A rational transverse
point exists for $D_K=-501510767$ and for $D_K=-931506071$; for
$D_K=-648153647$ the transverse point of smallest degree has
residue degree two.

\begin{theorem}\label{thm:p7}
Let $K$ be an imaginary quadratic number field of $7$-class rank
three with $|D_K|<2^{30}$.  Then the $7$-class tower group
$G_{\varnothing}(K)(7)$ is an infinite mild FAb pro-$7$ group, and in
particular
\[
  \cd G_{\varnothing}(K)(7)=2.
\]
\end{theorem}

\begin{proof}
\Cref{thm:transverse-rank-one}, applied to the transverse element
recorded in \cite{Repository} for each of the three fields, shows
that $G_{\varnothing}(K)(7)$ is mild.  \Cref{lem:mild-tower} gives
the remaining assertions.
\end{proof}

The field with $D_K=-648153647$ is the one field of the range whose
norm-degeneracy scheme has no rational point at all.  Its single
closed point of residue degree at most three has degree two, and
this point is transverse, so the freedom of a nontrivial residue
field in \cref{thm:transverse-rank-one} is needed at $p=7$ as well.

\subsection{The fields at \texorpdfstring{$p=3$}{p=3}}\label{sec:example-p3}

At $p=3$ the four cone strata of \cref{subsec:uniform-computation}
partition the range, and the three routes reach them unevenly.
\Cref{tab:strata-p3} gives the whole range at once.

\begin{table}[!htp]\centering\small
\begin{tabular}{l r r r r r r r}
\toprule
cone $C_\beta$ & fields & criterion & Anick & word count & mild & not mild & undecided\\
\midrule
the plane & $1$ & $1$ & $0$ & $0$ & $1$ & $0$ & $0$\\
a line & $220$ & $79$ & $1$ & $128$ & $208$ & $1$ & $11$\\
a point & $5183$ & $425$ & $6$ & $4550$ & $4981$ & $11$ & $191$\\
empty & $7345$ & $0$ & $0$ & $7053$ & $7053$ & $14$ & $278$\\
\midrule
all fields & $12\,749$ & $505$ & $7$ & $11\,731$ & $12\,243$ & $26$ & $480$\\
\bottomrule
\end{tabular}
\caption{The fields of $3$-class rank three with $|D_K|<2^{30}$,
stratified by the rank of $\delta$, equivalently by the shape of the
Bockstein cone.  The columns are those of \cref{tab:range}.  The
single field with a degenerate cubic relation space, of
discriminant $-873204895$, is counted undecided and discussed in
\cref{subsec:exceptional-fields}.}
\label{tab:strata-p3}
\end{table}

The criterion proves mildness for $505$ of the $12\,749$ fields, for
$496$ of them at a rational point of the cone, for six at a point of
residue degree two and for three at a point of residue degree three.
Part~W6 of \cite{Repository} records the point field by field.  Seven
further fields carry a rational Anick witness
(\cref{subsec:exceptional-fields}), and the word counts settle
$11\,731$ more, $7\,053$ of them in the stratum on which the cone is
empty, where no criterion that requires a closed point of $C_\beta$
can apply.  At $p=3$ the criterion thus decides only a small fraction of
the mild fields, and it isolates those on which mildness becomes
visible on $\Sigma_D$.  For $26$ fields the
same computation rules mildness out (\cref{prop:nonmild-p3}), and
one further field has a degenerate cubic relation space
(\cref{subsec:exceptional-fields}).

\begin{theorem}[Mild fields at $p=3$]\label{thm:mild-p3}
Let $K$ be any of the $12\,243$ imaginary quadratic number fields
of $3$-class rank three with $|D_K|<2^{30}$ listed in Parts~W6
and~W7 of \cite{Repository}.  Then the $3$-class tower group
$G_{\varnothing}(K)(3)$ is an infinite mild FAb pro-$3$ group, and in
particular
\[
  \cd G_{\varnothing}(K)(3)=2.
\]
\end{theorem}

\begin{proof}
For each of the $12\,243$ fields, Parts~W6 and~W7 of
\cite{Repository} record the route that decides it and the datum
that does so, a transverse element, a rational Anick witness, or a
substitution with a terminating completion, recomputed from the
verified matrices of the field by \cref{alg:transverse} or
\cref{alg:completion}.  The matrices determine the cubic initial
relation space of $G_{\varnothing}(K)(3)$, restricted cubes included
(\cref{prop:soundness}, \cref{prop:T-is-relations}).  A transverse
element makes $G_{\varnothing}(K)(3)$ mild by
\cref{cor:geometric-mildness}, applied over the residue field of the
point, and an Anick witness or a terminating completion with word
counts $1/(1-3z+3z^3)$ makes the three forms strongly free
(\cref{rem:direct-anick}, \cref{cor:word-counts}), so
\cref{prop:Gaertner} applies.  \Cref{lem:mild-tower} gives the
remaining assertions.
\end{proof}

\begin{proposition}\label{prop:nonmild-p3}
For the $26$ fields listed in \cref{tab:nonmild-p3} the tower
group $G_{\varnothing}(K)(3)$ is not mild with respect to the
Zassenhaus filtration.
\end{proposition}

\begin{proof}
In each of these fields the number of $I$-normal words, for the
two-sided ideal $I$ generated by the three cubic initial forms,
exceeds the coefficient of $1/(1-3z+3z^3)$ in a degree whose
overlaps the completion has resolved, first in the degree
recorded in \cref{tab:nonmild-p3}, so the cubic initial relation space is not
strongly free by \cref{cor:word-counts}.  Each of the fields has
$\rk_{\Fp}T=3$, so by \cref{rem:rank-three} the failure is one of
the group and not of a chosen presentation.
\end{proof}

\noindent
For these $26$ tower groups the cohomological dimension remains
open, since mildness is sufficient for
$\cd G_{\varnothing}(K)(3)=2$ but not necessary.

\begin{table}[!htp]\centering\small
\begin{tabular}{@{}l@{\quad}l@{\quad}>{\raggedright\arraybackslash}p{0.62\textwidth}@{}}
\toprule
degree & fields & discriminants\\
\midrule
$4$ & $17$ & $-54962868$, $-123927735$, $-135819467$, $-252545119$,
$-259111511$, $-297353236$, $-306860520$, $-572463199$,
$-591612196$, $-718886391$, $-764214755$, $-835614552$,
$-896366299$, $-896646495$, $-1032299263$, $-1052572404$,
$-1070351128$\\
\addlinespace
$9$ & $2$ & $-257444104$, $-790383263$\\
\addlinespace
$10$ & $3$ & $-400506295$, $-526671699$, $-698938715$\\
\addlinespace
$12$ & $4$ & $-176214724$, $-522302531$, $-623631988$,
$-1030775848$\\
\bottomrule
\end{tabular}
\caption{The $26$ fields of the range whose cubic initial relation
space is provably not strongly free, by the first degree in which
the number of $I$-normal words exceeds the coefficient of
$1/(1-3z+3z^3)$.  The completions behind the verdicts are recorded
under \texttt{records/p3/results/} of \cite{Repository}.}
\label{tab:nonmild-p3}
\end{table}

The remaining $480$ fields of the range are undecided.  Failure of
the criterion, or of a search, proves nothing about mildness, and
neither does the absence of a deviation through the degree bound.
The mildness and the cohomological dimension of these fields remain
open.  No claim is intended about how often the criterion applies
beyond the bound $2^{30}$.

\subsection{Boundary cases}\label{subsec:exceptional-fields}

Nine fields at $p=3$ exhibit the three boundary phenomena of
the routes: a Bockstein cone that is the whole plane, cone
points at which the criterion fails while an Anick witness
exists, and a rank defect in the reconstructed matrix $T$.  The
range contains no other instances.

\paragraph{The plane stratum.}  The only field of the range on
which the Bockstein cone is the whole plane is the one with
$D_K=-364435991$.  Its $3$-class
group has no cyclic factor of order exactly $3$, so
$\delta=0$, hence $\beta=0$ by
\cref{lem:diagonal-p3}\ref{diag:dual} and $C_\beta$ imposes no
condition at all, and the criterion applies.  This is the single
field of the plane stratum in \cref{tab:strata-p3}.

\paragraph{The witness fields.}  Seven fields carry a rational
Anick witness: the fields of discriminant $-211248887$,
$-263780072$, $-435714452$, $-459350555$, $-622080024$,
$-639273539$, and $-746301223$.  In each of them the
cone holds a closed point at which the criterion fails; for the
first two fields the point is degenerate, with local ring of
$\Sigma_D$ of length three.  All seven are recorded in Part~W6
of \cite{Repository}.

\paragraph{The rank defect.}  The field with $D_K=-873204895$ is the
only one of the range with $\rk_{\Fp}T<3$.  Its cubic initial
relation space has dimension two, so by \cref{prop:T-is-relations}
one relator of the tower group has initial degree greater than
three, no three cubic initial forms can be strongly free, and the
field is counted undecided.  Mildness is not thereby excluded,
because the sequence of relator initial forms now contains one
form of degree at least four and could still be strongly free, a
case none of our routes tests.  The word count detects the defect in
degree three, where $27-\rk_{\Fp}T$ exceeds the coefficient $24$
(\cref{rem:span-strongly-free}).

\paragraph{The field of rank four.}  A single field of the
range has $p$-class rank exceeding three for an odd prime $p$:
the field with $D_K=-653329427$, of $3$-class rank four and
$3$-class group $(3,3,3,3)$, read off from the tables as for
the populations above.  The criterion of
\cref{thm:transverse-rank-one} applies to any rank $d\ge3$, and
the arithmetic of \cref{sec:computation}, adapted to rank four,
was carried out for this field in exact arithmetic and verified as
in \cref{app:verification} (\cref{rem:rank-d-verification}).
The
$d(d+1)/2=10$ secondary norm operators determine the Bockstein
matrix, which has rank four; the cone $C_\beta$ is therefore
empty, and the $p=3$ route admits no candidate point.  The
direct route decides the field.

\begin{theorem}[The field of rank four]\label{thm:rank-four}
The $3$-class tower group of the field with $D_K=-653329427$ is an
infinite mild FAb pro-$3$ group of $3$-class rank four, and
$\cd G_{\varnothing}(K)(3)=2$.
\end{theorem}

\begin{proof}
The arithmetic of \cref{sec:computation}, verified as in
\cref{app:verification} (\cref{rem:rank-d-verification}), determines
the four cubic initial forms (\cref{prop:T-is-relations} holds for any rank),
and Part~W7 of \cite{Repository} records a change of variables in
$\mathrm{GL}_4(\Fthree)$ after which the completion of
\cref{alg:completion} terminates and the word counts agree with
$1/(1-4z+4z^3)$.  The four forms are therefore strongly free
(\cref{cor:word-counts} holds for any number of generators and
relations), and \cref{prop:Gaertner} and \cref{lem:mild-tower}
give the assertions.
\end{proof}

\begin{remark}[Rank four beyond the range]\label{rem:rank-four-census}
We have also treated the fields of $3$-class rank four with
$|D_K|<2\cdot10^{10}$.  There are $62$ of them, enumerated
independently from the cubic-field counts of Belabas and from the
tables of Mosunov--Jacobson, and all $62$ have a mild $3$-class tower
group, four by the criterion, at a rational point of their Bockstein
cone, and the others by a word count after a change of variables in
$\mathrm{GL}_4(\Fthree)$.  The records are in \cite{Repository}.
\end{remark}

\section{Consequences of mildness}\label{sec:consequences}

Let $K$ be one of the fields of \cref{thm:mainexample},
\cref{thm:p7}, \cref{thm:mild-p3}, or \cref{thm:rank-four}, let
$p$ be the corresponding prime, let $d$ be the $p$-class rank of $K$,
and write $G=G_{\varnothing}(K)(p)$.  In each case $G$ is an
infinite mild FAb pro-$p$ group.  It has a presentation with $d$
generators and $d$ strongly free cubic initial relations, $d=3$
except for the field of rank four, and \eqref{eq:ranks} gives
$\dim_{\Fp} H^1(G,\Fp)=\dim_{\Fp} H^2(G,\Fp)=d$.  The
consequences recorded in this section apply uniformly to all
$12\,451$ groups.

\begin{corollary}\label{cor:nonanalytic}
The group $G$ is not a $p$-adic analytic pro-$p$ group.
\end{corollary}

\begin{proof}
Since $\cd G=2$ is finite, $G$ has no $p$-torsion and is
therefore torsion free.  If $G$ were $p$-adic analytic, Lazard's
theorem \cite[Th\'eor\`eme~V.2.5.8]{Lazard} would make it a
Poincar\'e group of dimension equal to its analytic dimension.  A
Poincar\'e group of dimension $n$ has $\cd=n$, so $n=2$, and then
$H^2(G,\Fp)\simeq\Fp$ \cite[(3.7.6)]{NSW}.  But
$\dim_{\Fp}H^2(G,\Fp)=d\ge3$ by \eqref{eq:ranks}.
\end{proof}

\begin{remark}\label{rem:FM}
The unramified Fontaine--Mazur conjecture asserts in particular that
$G_{\varnothing}(K)(p)$ has no infinite $p$-adic analytic quotient
for any number field $K$ \cite[Conjecture~5b]{FontaineMazur}.
\Cref{cor:nonanalytic} confirms this prediction for the group itself
in all our examples, at all three primes; it says nothing about proper
quotients, since cohomological dimension does not pass to quotients.
Non-analyticity as such is not new for this class of fields.  For an
imaginary quadratic number field with odd $p$ and $p$-class rank
$d\ge3$, the defining relations lie in $F_{(3)}$, so the minimal
presentation satisfies the Golod--Shafarevich condition in its
weighted form, with each relator counted by its Zassenhaus degree:
$1-dt+dt^{3}<0$ for some $t\in(0,1)$, for instance $t=3/5$.  By
Zelmanov's theorem such groups
contain nonabelian free pro-$p$ subgroups \cite{Zelmanov}.  In
particular they are not $p$-adic analytic.  The argument of
\cref{cor:nonanalytic} is of a different nature.  Given mildness it
is elementary, and it excludes analyticity through the cohomological
dimension rather than through subgroup structure.
\end{remark}

Mildness has two further quantitative consequences, recorded in the
remarks below.

\begin{remark}[Relation rank along the tower]
\label{rem:chi-ledger}
Let $L$ be a finite subextension of $K_{\varnothing}(p)| K$ and
put $U=\Gal(K_{\varnothing}(p)| L)$, an open subgroup of $G$ of
index $[L:K]$.  Recall first that
$L_{\varnothing}(p)=K_{\varnothing}(p)$.  The extension
$L_{\varnothing}(p)| K_{\varnothing}(p)$ is unramified and
pro-$p$, and so is every conjugate
$\sigma(L_{\varnothing}(p))=(\sigma L)_{\varnothing}(p)$ over
$K_{\varnothing}(p)$, since $\sigma L\subseteq K_{\varnothing}(p)$;
the compositum of the conjugates is therefore an everywhere
unramified pro-$p$ extension of $K$, and it lies in
$K_{\varnothing}(p)$.  Thus
$U=G_{\varnothing}(L)(p)$.

By the Artin isomorphism, the maximal
abelian everywhere unramified pro-$p$ extension of $L$ has group the
$p$-primary part of $\Cl(L)$, so
$H^1(U,\Fp)\simeq(\Cl(L)/p)^{*}$.  Each $H^i(U,\Fp)$ is finite,
since by Shapiro's lemma $H^i(U,\Fp)=H^i(G,\Fp[G/U])$ and $\Fp[G/U]$ is a
successive extension of trivial modules $\Fp$.  Since $\cd G=2$, the
Euler--Poincar\'e characteristic
\[
  \chi(U)=1-\dim H^1(U,\Fp)+\dim H^2(U,\Fp)
\]
is multiplicative \cite[Proposition~3.3.13]{NSW},
\[
  \chi(U)=[G:U]\,\chi(G),
\]
and $\chi(G)=1-d+d=1$.  Hence
\[
  \dim_{\Fp}H^2\bigl(G_{\varnothing}(L)(p),\Fp\bigr)
  \;=\;\dim_{\Fp}\Cl(L)/p\;+\;[L:K]-1
\]
for every finite subextension $L$ of the tower.
\end{remark}

\begin{remark}[Growth of the Zassenhaus graded pieces]
\label{rem:growth}
The mild presentation, with its $d$ generators and $d$ strongly
free cubic initial relations, also determines the dimensions of all
Zassenhaus graded pieces of $G$.  We write the case $d=3$ and note
the change for $d=4$ at the end.  Write $G_{(n)}$ for the Zassenhaus
filtration of $G$, in analogy with $F_{(n)}$.  For a mild presentation,
\[
  \operatorname{gr}\Fp[[G]]\;\simeq\;U(\operatorname{gr}G),
\]
where $U$ denotes the restricted universal envelope and $\Fp[[G]]$
carries the induced filtration.  This is Quillen's theorem
\cite{Quillen}, given for the mild pro-$p$ setting in
\cite[Theorem~2.12(iv)]{Gaertner}.  Moreover
\[
  \operatorname{gr}G\;\simeq\;
  (\operatorname{gr}F)/(\rho_1,\rho_2,\rho_3),
\]
the quotient of the restricted Lie algebra $\operatorname{gr}F$ by
the ideal generated by the initial forms
\cite[Theorem~2.12(iii)]{Gaertner}.  Combining the two identifications,
\[
 \operatorname{gr}\Fp[[G]]\;\simeq\;
 \Fp\langle X_1,X_2,X_3\rangle/(\rho_1,\rho_2,\rho_3);
\]
by strong freeness, the Hilbert series of this algebra is
$1/(1-3z+3z^3)$ \eqref{eq:hilbert-series}.
The restricted Poincar\'e--Birkhoff--Witt theorem
\cite[Chapter~V, Theorem~12]{Jacobson} factors the
series as
\[
  \frac1{1-3z+3z^3}
  \;=\;\prod_{n\ge1}
  \Bigl(\frac{1-z^{pn}}{1-z^{n}}\Bigr)^{c_n},
  \qquad c_n=\dim_{\Fp}G_{(n)}/G_{(n+1)},
\]
which determines every $c_n$.  Here the prime enters.  The first
values are
\[
 (c_n)_{n\ge1}=3,\,3,\,5,\,9,\,24,\,41,\,96,\,207,\,475,\,1077,\dots
 \qquad(p=5),
\]
\[
 (c_n)_{n\ge1}=3,\,3,\,5,\,9,\,21,\,41,\,99,\,207,\,475,\,1074,\dots
 \qquad(p=7),
\]
\[
 (c_n)_{n\ge1}=3,\,3,\,8,\,9,\,21,\,44,\,96,\,207,\,483,\,1074,\dots
 \qquad(p=3).
\]
The value of $c_3$ can also be seen directly.  The ideal
$(\rho_1,\rho_2,\rho_3)$ of $\operatorname{gr}F$ is generated in
degree three, so its degree-three part is the span of the three
initial forms, and these are linearly independent by strong
freeness (\cref{rem:span-strongly-free}).  Hence
\[
  c_3=\dim_{\Fp}(\operatorname{gr}F)_3-3,
\]
and $\dim_{\Fp}(\operatorname{gr}F)_3$ is $8$ for $p>3$ and $11$
for $p=3$ (\cref{lem:eleven-space}).  This gives $c_3=5$ and
$c_3=8$, in agreement with the values above.  For the field of rank
four the same argument starts from $1/(1-4z+4z^3)$ and gives
$(c_n)_{n\ge1}=4,\,6,\,20,\,44,\,140,\,430,\,1380,\,4456,\dots$ at $p=3$.
\end{remark}

\appendix

\section{Word counts and Gr\"obner bases}\label{app:word-counts}

The mildness proofs of \cref{sec:example-p3} use, besides Anick's
criterion, a route that counts words.  Unlike Anick's criterion it
can refute strong freeness as well as establish it.  Its
ingredients (normal words, the diamond lemma, Gr\"obner bases in
the free algebra) are not new, but the computation needs them
in an adapted form.  The completion stops at a degree bound, and
the comparison of Hilbert series has to be finite.  This appendix
assembles them in that form.

We keep the notation of \cref{subsec:strong-freeness}:
$\mathcal A=\Fp\langle\mgen_1,\ldots,\mgen_d\rangle$ is the free
associative $\Fp$-algebra on $\mgen_1,\ldots,\mgen_d$, a
multiplicative monomial order on its words is fixed, and
$\HT(\rho)$ denotes the high term of a nonzero
$\rho\in\mathcal A$.

Let $I\subseteq\mathcal A$ be a homogeneous two-sided ideal and
write $I_n=I\cap\mathcal A_n$.  Since $I$ is homogeneous, the
quotient $\mathcal A/I$ is graded and the projection induces an
isomorphism
\[
  \mathcal A_n/I_n\xrightarrow{\ \sim\ }(\mathcal A/I)_n
\]
in every degree.  We use this identification throughout.  Define the
\emph{high-term ideal} of $I$ by
\[
  \HT(I):=\bigl(\HT(g):0\ne g\in I\text{ homogeneous}\bigr).
\]
A word $w$ of $\mathcal A$ is called \emph{$I$-normal} if
$w\notin\HT(I)$, equivalently, if no factor of $w$, including $w$
itself, is the high term of a nonzero homogeneous element of $I$.

\begin{lemma}[Normal words compute the Hilbert series]
\label{lem:normal-words}
Let $I\subseteq\mathcal A$ be a homogeneous two-sided ideal.  In
every degree $n$, the $I$-normal words of $\mathcal A$ of degree $n$
project to an
$\Fp$-basis of $(\mathcal A/I)_n$.  In particular
$\dim_{\Fp}(\mathcal A/I)_n$ is the number of $I$-normal words of
$\mathcal A$ of degree $n$.
\end{lemma}

\begin{proof}
Fix $n$.  We first show that every word $w$ of $\mathcal A$ of
degree $n$ is congruent modulo $I_n$ to a linear combination of
$I$-normal words of $\mathcal A$.
First observe that if a word $w$ of $\mathcal A$ of degree $n$ is not
$I$-normal, then, after scaling a suitable nonzero homogeneous
$g\in I$, we can write
\[
  w=u\HT(g)v,
\]
where the high term of $g$ has coefficient one and $u,v$ are possibly
empty words of $\mathcal A$.  Every other word $w'$ of $\mathcal A$
occurring in $g$ satisfies
$w'<\HT(g)$, and hence
\[
  uw'v<u\HT(g)v=w
\]
because the order is multiplicative.  Since $ugv\in I_n$, the word
$w$ is congruent modulo $I_n$ to $w-ugv$, which is a linear combination
of words of $\mathcal A$ of degree $n$ strictly smaller than $w$.

We now argue by induction on $w$ in the finite ordered set of words
of $\mathcal A$ of degree $n$, the case of the least word being the
one where the induction hypothesis is empty.  So let $w$ be a word
of degree $n$ and assume the assertion for every word of
$\mathcal A$ of degree $n$ strictly smaller than $w$.  If $w$ is
$I$-normal, there is nothing to prove.  Otherwise the observation
above and the induction hypothesis express $w$ modulo $I_n$ as a
linear combination of $I$-normal words of $\mathcal A$.  This
completes the induction and proves that the images of the
$I$-normal words of $\mathcal A$ of degree $n$ span
$\mathcal A_n/I_n$, and hence $(\mathcal A/I)_n$.

It remains to prove that these images are linearly independent.
Suppose, to the contrary, that distinct $I$-normal words $w_1,\ldots,w_r$ of
$\mathcal A$ of degree $n$ and coefficients $c_1,\ldots,c_r\in\Fp$,
not all zero, satisfy
\[
  f:=\sum_{i=1}^r c_iw_i\in I_n.
\]
Then $f$ is a nonzero homogeneous element of $I$, and $\HT(f)$ is one
of the words $w_i$.  But $\HT(f)$ is the high term of $f\in I$ and is
a factor of itself, so it is not $I$-normal, contradicting the choice
of the $w_i$.  Thus the images of the $I$-normal words of $\mathcal A$
of degree $n$ are linearly independent.  Together with the spanning
result above, this proves that they form a basis of $\mathcal A_n/I_n$
and hence of $(\mathcal A/I)_n$.  The dimension formula follows.
\end{proof}

Anick works with the degree-lexicographic order, constructing an
order-ideal basis in \cite[Lemma~1.1]{Anick} and proving the
dimension inequality of \cite[Theorem~1.4]{Anick}.  The argument
above gives the exact count for every multiplicative monomial
order, which is what the computation uses.

Recall from \cref{subsec:strong-freeness} that homogeneous
$\rho_1,\ldots,\rho_m\in\mathcal A$ of degrees
$h_1,\ldots,h_m\ge1$ are called \emph{strongly free} precisely when
\[
  \operatorname{Hilb}_{\mathcal A/(\rho_1,\ldots,\rho_m)}(z)
  =\frac{1}{1-dz+z^{h_1}+\cdots+z^{h_m}}.
  \tag{\ref{eq:hilbert-series}}
\]

\begin{corollary}[Word counts decide strong freeness]
\label{cor:word-counts}
Let $\rho_1,\ldots,\rho_m\in\mathcal A$ be homogeneous of degrees
$h_1,\ldots,h_m\ge1$ and set $I=(\rho_1,\ldots,\rho_m)$.  The sequence
$\rho_1,\ldots,\rho_m$ is strongly free if and only if, for every
$n$, the number of $I$-normal words of $\mathcal A$ of degree $n$
equals the coefficient of $z^n$ in the right-hand side of
\eqref{eq:hilbert-series}.  In particular, failure of this equality
in a single degree disproves strong freeness.
\end{corollary}

\begin{corollary}[Computing normal-word counts]
\label{cor:computing-normal-words}
Let $\rho_1,\ldots,\rho_m\in\mathcal A$ be homogeneous of degrees
$h_1,\ldots,h_m\ge1$ and set $I=(\rho_1,\ldots,\rho_m)$.
\begin{enumerate}[label=\textup{(\roman*)}]
\item\label{wc:single-degree} For every fixed $n$, $I_n$ is spanned
  by the finite set of products $u\rho_\ell v$, where
  $u,v$ are possibly empty words of $\mathcal A$ and
  $\deg u+h_\ell+\deg v=n$.  Row reduction of their coefficient
  vectors in the word basis of $\mathcal A_n$ determines
  $\dim_{\Fp}I_n$, and
  \[
    \#\{\text{$I$-normal words of $\mathcal A$ of degree $n$}\}
    =d^n-\dim_{\Fp}I_n.
  \]
\item\label{wc:finite-basis} If $\HT(I)$ is generated by a finite set
  $W$ of words of $\mathcal A$, then a word of
  $\mathcal A$ is $I$-normal if and only if none of its factors lies
  in $W$.  The counts in all degrees are then determined by $W$.
\end{enumerate}
\end{corollary}

Part \ref{wc:finite-basis} determines the counts in all degrees,
so comparing them with \eqref{eq:hilbert-series} looks like an
infinite task.  It is not.  Both sequences obey linear recurrences,
and finitely many degrees then settle the comparison.

\begin{lemma}[Word counts satisfy a linear recurrence]
\label{lem:finite-comparison}
Let $W$ be a finite set of words of $\mathcal A$, let $c_n$ be the
number of words of degree $n$ with no factor in $W$, and let $s$ be
the number of proper prefixes of elements of $W$, the empty word
included.
\begin{enumerate}[label=\textup{(\roman*)}]
\item\label{fc:recurrence} The sequence $(c_n)_{n\ge0}$ satisfies a
  linear recurrence of order at most $s$.
\item\label{fc:comparison} Let $(b_n)_{n\ge0}$ be a sequence of
  rational numbers satisfying a linear recurrence of order at most
  $r$ with rational coefficients.  If
  $c_n=b_n$ for all $n\le s+r-1$, then $c_n=b_n$ for all $n$.
\end{enumerate}
\end{lemma}

\begin{proof}
Let $\mathcal P$ be the set of proper prefixes of elements of $W$,
including the empty word, so that $|\mathcal P|=s$.  For
$P\in\mathcal P$, let $c_{n,P}$ be the number of words of degree
$n$ with no factor in $W$ whose longest suffix in $\mathcal P$ is
$P$.

Let $u$ be such a word and $x$ a letter.  We claim that
\[
  \text{$ux$ has a factor in $W$}
  \quad\Longleftrightarrow\quad
  \text{$Px$ has a suffix in $W$},
\]
and that the suffixes of $ux$ lying in $\mathcal P$ are exactly
those of $Px$.

Since $P$ is a suffix of $u$, the word $Px$ is a suffix of $ux$.
Consequently, every suffix of $Px$ lying in $W$ is a factor of
$ux$, which proves the implication from right to left.  For the
implication from left to right, suppose that $ux$ has a factor
$w\in W$.  Since $u$ has no factor in $W$, the factor $w$
must end in the final letter $x$.  Thus $w=vx$ for some suffix $v$
of $u$.  As $v$ is a proper prefix of $w\in W$, we have
$v\in\mathcal P$.  The maximality of $P$ gives
$\deg v\le\deg P$.  Since both $v$ and $P$ are suffixes of $u$, it
follows that $v$ is a suffix of $P$, and hence that $w=vx$ is a
suffix of $Px$.

For the two sets of suffixes, one inclusion is again immediate.
As $Px$ is a suffix of $ux$, every suffix of $Px$ lying in
$\mathcal P$ is a suffix of $ux$.  For the other, let
$w\in\mathcal P$ be a suffix of $ux$.  If $w$ is
empty, then it is also a suffix of $Px$.  Otherwise $w=vx$ for
some suffix $v$ of $u$.  Since $w$ is a proper prefix of an
element of $W$, so is $v$, and hence $v\in\mathcal P$.  The same
maximality argument as above shows that $v$ is a suffix of $P$,
and $w=vx$ a suffix of $Px$.

Whether $ux$ still has no factor in $W$, and which its longest
suffix in $\mathcal P$ is, therefore depend only on $P$ and $x$.
Writing $m_{PQ}$ for the number of letters $x$ such that $Px$ has
no factor in $W$ and longest suffix $Q$ in $\mathcal P$, this
gives
\[
  c_{n+1,Q}=\sum_{P\in\mathcal P}m_{PQ}\,c_{n,P},
  \qquad
  c_n=\sum_{P\in\mathcal P}c_{n,P}.
\]
The numbers $m_{PQ}$ and $c_{n,P}$ are non-negative integers.  We
regard them as elements of $\mathbb Q$, so that the linear algebra
below takes place over a field.  Write
$v_n\in\mathbb Q^s$ for the row vector $(c_{n,P})_{P\in\mathcal P}$,
and let $M=(m_{PQ})_{P,Q\in\mathcal P}$ be the $s\times s$ matrix
over $\mathbb Q$ whose entry in row $P$ and column $Q$ counts the
letters leading from $P$ to $Q$.  The first identity above then
reads
\[
  v_{n+1}=v_nM,
\]
and $c_n$ is the sum of the entries of $v_n$.  Let
$t^s-a_{s-1}t^{s-1}-\cdots-a_0\in\mathbb Q[t]$ be the characteristic
polynomial of $M$.  By the Cayley--Hamilton theorem
\[
  M^s=a_{s-1}M^{s-1}+\cdots+a_0E,
\]
with $E$ the $s\times s$ identity matrix.  Multiplying on the left
by $v_n$
gives
\[
  v_{n+s}=a_{s-1}v_{n+s-1}+\cdots+a_0v_n,
\]
and summing the entries of each side gives
\[
  c_{n+s}=a_{s-1}c_{n+s-1}+\cdots+a_0c_n,
\]
a linear recurrence of order at most $s$.  This is
\ref{fc:recurrence}.

For \ref{fc:comparison}, write $S$ for the shift operator
$S(a_n)_n=(a_{n+1})_n$ on the $\mathbb Q$-vector space of
sequences of rational numbers, so that a linear recurrence of
order $e$ for a sequence $a$ is a monic polynomial
$f\in\mathbb Q[t]$ of degree $e$ with $f(S)a=0$.  Let $f_c$ and $f_b$ be such polynomials for $(c_n)$ and
$(b_n)$, of degrees at most $s$ and $r$.  Polynomials in $S$
commute, so $f_cf_b$ annihilates both sequences and hence their
difference, which satisfies a linear recurrence of order
at most $s+r$.
Vanishing in the degrees $0,\ldots,s+r-1$ makes it vanish
identically.
\end{proof}

It remains to produce a finite set $W$ as in
\cref{cor:computing-normal-words}\ref{wc:finite-basis}.  This is
what a two-sided Gr\"obner basis of $I$ does, and the computations
of \cref{subsec:uniform-computation} compute one (with Singular's
Letterplace subsystem \cite{Singular,LaScalaLevandovskyy})
by completing the defining relations.  Overlaps of high terms are
reduced, nonzero results are added, and the process repeats.  If the
completion terminates, the high terms of the resulting set generate
$\HT(I)$ (\cref{lem:diamond}\ref{dl:full} below).

A two-sided ideal of $\mathcal A$ need not have a finite
Gr\"obner basis, and completions that do terminate can still take
arbitrarily long.  The computation is therefore given a degree
bound in advance and is abandoned once the overlaps to be resolved
exceed it, which in \cref{subsec:uniform-computation} is the usual outcome.
\Cref{lem:diamond}\ref{dl:truncated} below says what an abandoned
computation still yields.  In every degree whose overlaps have been
resolved the word counts are exact.  This is what lets such a computation
prove that a relation ideal is \emph{not} strongly free.

Let $G$ be a finite set of homogeneous elements of $I$, each
normalized to high-term coefficient one.  A \emph{reduction step}
with respect to $G$ replaces a homogeneous element
$h\in\mathcal A$ in which a
word $u\HT(g)v$ occurs with coefficient $c\ne0$, where $g\in G$ and
$u,v$ are possibly empty words, by $h-c\,ugv$.  This removes the word
$u\HT(g)v$ in favour of strictly smaller words of the same degree.
Iterated reduction steps terminate.  Each step removes one word and
alters only strictly smaller ones, so the set of words occurring in
$h$ decreases strictly in the lexicographic comparison along the
finitely many words of its degree, taken from the largest word
downwards.  We say that $h$ \emph{reduces to zero} with
respect to $G$ if some sequence of reduction steps ends in $0$.
Reduction rewrites the single element $h$ and uses $G$ as a fixed
system of rewriting rules; it does not change $G$.

\begin{lemma}[Bases from resolved overlaps]
\label{lem:diamond}
Let $G$ be a finite set of homogeneous elements generating $I$, each
with high-term coefficient one, and set $W=\{\HT(g):g\in G\}$.
Assume that the high terms of distinct elements of $G$ are
distinct and that no element of $W$ is a factor of a different
element of $W$.  Call a pair of factorizations
\[
  \HT(g)=AB,\qquad \HT(g')=BC
\]
with $g,g'\in G$ and nonempty words $A,B,C$ an \emph{overlap} of
$G$.  Its \emph{degree} is $\deg ABC$, the degree of the overlapped
word.
\begin{enumerate}[label=\textup{(\roman*)}]
\item\label{dl:truncated} Let $N\ge0$, and suppose that for every
  overlap of degree at most $N$ the element $gC-Ag'$ reduces to
  zero with respect to $G$.  Then in every degree $n\le N$ the
  words of degree $n$ with no factor in $W$ project to an
  $\Fp$-basis of $(\mathcal A/I)_n$.
\item\label{dl:full} If this holds for every overlap, of any
  degree, then $\HT(I)=(W)$.
\end{enumerate}
\end{lemma}

Part \ref{dl:full} is the Buchberger criterion \cite[Theorem~5.9]{Mora}.
We prove both parts together, since the argument for the truncated
statement gives it along the way.

\begin{proof}
Fix $N$ as in \ref{dl:truncated}.  We apply the diamond lemma
\cite[Theorem~1.2]{Bergman} (see also \cite[\S1.4]{Rogalski})
not to $I$ itself but to the ideal
\[
  I+\bigoplus_{n>N}\mathcal A_n,
\]
which has the same graded pieces as $I$ in the degrees at most
$N$.  It is the ideal defined by the augmented reduction system
\begin{enumerate}[label=\textup{(R\arabic*)}]
\item for each $g\in G$, replace $\HT(g)$ by $\HT(g)-g$;
\item for each word $q$ of degree $N+1$, replace $q$ by zero.
\end{enumerate}

To apply the diamond lemma, we must check that the reductions are
compatible with a partial order satisfying the descending chain
condition and that every ambiguity is resolvable relative to its
ambiguity word \cite[Theorem~1.2(a$'$)]{Bergman}.
\begin{itemize}
\item \emph{Order and termination.}  Order words of the same degree
  by the fixed monomial order and leave words of different degrees
  incomparable.  This partial order is compatible with concatenation.
  An \textup{(R1)} rule replaces a word by smaller words of the
  same degree, and an \textup{(R2)} rule replaces a word by
  nothing at all, so both are compatible with this order.  A descending chain lies in one
  degree and is finite because that degree contains only finitely
  many words.

\item \emph{Ambiguities.}  The possible ambiguities are exhausted by
  the following three cases.
  \begin{itemize}
  \item An inclusion between two \textup{(R1)} rules cannot occur,
    since by assumption their left-hand sides are distinct and neither is
    a factor of the other.

  \item Suppose that the ambiguity word has degree greater than $N$.
    After either initial reduction, every word that remains has the
    same degree as the ambiguity word and hence contains a factor
    of degree $N+1$.
    Successive \textup{(R2)} rules therefore reduce both results to
    zero.  This covers every ambiguity involving an \textup{(R2)}
    rule and every overlap between two \textup{(R1)} rules in degree
    greater than $N$.

  \item The remaining ambiguities are overlaps between two
    \textup{(R1)} rules in degree at most $N$.  They have the form
    \[
      \HT(g)=AB,\qquad \HT(g')=BC,
    \]
    with $A,B,C$ nonempty.  The two reductions differ, up to sign,
    by $gC-Ag'$, which reduces to zero by hypothesis.  All words in
    $gC-Ag'$ are strictly smaller than the ambiguity word $ABC$, and
    subsequent reductions introduce only still smaller words.  Thus
    the ambiguity is resolvable relative to $ABC$.
  \end{itemize}
\end{itemize}
All hypotheses of the diamond lemma are verified, and it concludes
that the words in which neither an element of $W$ nor a word of
degree $N+1$ occurs as a factor, that is, the words of degree at
most $N$ with no factor in $W$, project to a basis of the
quotient by the displayed ideal.  In degrees $n\le N$ that quotient
has the same graded pieces as $\mathcal A/I$, proving
\ref{dl:truncated}.

For \ref{dl:full}, note first that $W\subseteq\HT(I)$ and hence
$(W)\subseteq\HT(I)$.  Since both ideals are generated by words, it
suffices to see that they contain equally many words in each
degree $n$.  Applying \ref{dl:truncated} with $N=n$, the words of
degree $n$ outside $(W)$ number $\dim_{\Fp}(\mathcal A/I)_n$, and
by \cref{lem:normal-words} so do the words of degree $n$ outside
$\HT(I)$.  Hence $\HT(I)=(W)$.
\end{proof}

The computation of \cref{subsec:uniform-computation} puts these statements
together as follows.

\begin{algo}[Deciding strong freeness by completion]
\label{alg:completion}\leavevmode
\begin{description}[font=\normalfont\itshape,topsep=2pt,itemsep=0pt]
\item[Input:] nonzero homogeneous
  $\rho_1,\ldots,\rho_m\in\mathcal A$, $m\ge1$, of degrees
  $h_1,\ldots,h_m\ge1$, and a degree bound $N\ge0$.
\item[Output:] \textsc{strongly free}, \textsc{not strongly free},
  or \textsc{undecided}.
\end{description}
\begin{enumerate}[label=\textup{Step~\arabic*.},leftmargin=*]
\item \emph{Completion.}  Normalize $\rho_1,\ldots,\rho_m$ to
  high-term coefficient one and inter-reduce them.  As long as the
  high term of one element occurs in another, reduce the latter by
  the former, discarding it if the result is zero and normalizing
  it again otherwise.  Let $G$
  be the resulting set, which satisfies the standing assumption of
  \cref{lem:diamond}.
  Process the overlaps of elements of $G$ in increasing degree.  For
  $\HT(g)=AB$, $\HT(g')=BC$ reduce $gC-Ag'$ with respect to $G$.
  If the result is nonzero, normalize it, add it to $G$ and
  inter-reduce again, so that $G$ keeps satisfying that
  assumption.  None of these
  operations changes the ideal
  generated by $G$, which therefore stays
  $I=(\rho_1,\ldots,\rho_m)$ throughout.  Stop
  when no unresolved overlap is left, or when every unresolved
  overlap has degree greater than $N$.  Write $W$ for the set of
  high terms of the final $G$.
\item \emph{Verdict.}  Write $c_n$ for the number of words of
  degree $n$ with no factor in $W$, and $b_n$ for the coefficient
  of $z^n$ in the right-hand side of \eqref{eq:hilbert-series}.
  \begin{itemize}
  \item If $c_n\ne b_n$ for some $n\le N$, return
    \textsc{not strongly free}.
  \item Otherwise, if Step~1 stopped at the bound, return
    \textsc{undecided}.
  \item Otherwise compare $c_n$ with $b_n$ for
    $n\le s+H-1$, where $s$ is the number of proper prefixes of
    elements of $W$ including the empty word and
    $H=\max_\ell h_\ell$.  Return \textsc{strongly free} if they
    agree throughout, and \textsc{not strongly free} otherwise.
  \end{itemize}
\end{enumerate}
\end{algo}

\begin{proposition}\label{prop:completion-correct}
The output of \cref{alg:completion} is correct.
\end{proposition}

\begin{proof}
Step~1 terminates.  At every stage the elements of $G$ are
homogeneous of degree at most $\max(N,h_1,\ldots,h_m)$.  Each
inter-reduction step lowers a high term in the word order, and
each element adjoined to $G$ strictly enlarges the monomial ideal
generated by the high terms seen so far, so in bounded degree
neither happens infinitely often.

Once resolved, an overlap remains resolvable with respect to the
final $G$.  Reduction to zero writes $gC-Ag'$ as a sum of products
$u\,g''\,v$ with $g''\in G$ and $u\HT(g'')v<ABC$, and each later
operation (replacing some $g''$ by $g''-c\,u'hv'$ with every
replaced word below $\HT(g'')$, discarding an element that has
become a combination of the others, or modifying $g$ and $g'$
themselves by terms that stay below $ABC$ after multiplication by
$C$ and $A$) rewrites the sum in the same form.  This is
resolvability of the ambiguity relative to $ABC$, the property the
proof of \cref{lem:diamond} derives from its hypothesis.

Suppose first that Step~1 leaves no unresolved overlap, and let $q$
be the largest degree of an element of the final $G$.  In an overlap
$\HT(g)=AB$, $\HT(g')=BC$ the word $B$ is nonempty, so
$\deg ABC\le 2q-1$.  There are finitely many overlaps, all of degree
at most $2q-1$, and all of them have been resolved.  Hence
\cref{lem:diamond}\ref{dl:full} gives $\HT(I)=(W)$ for
$I=(\rho_1,\ldots,\rho_m)$, so by
\cref{cor:computing-normal-words}\ref{wc:finite-basis} the number
$c_n$ of Step~2 is the number of $I$-normal words of degree $n$; so
\cref{cor:word-counts} turns $c_n=b_n$ for all $n$ into strong
freeness, and a single deviation into its failure.  Writing
$\bigl(\sum_nb_nz^n\bigr)\bigl(1-dz+\sum_\ell z^{h_\ell}\bigr)=1$
and comparing coefficients gives
\[
  b_{n+H}=d\,b_{n+H-1}-\sum_{\ell}b_{n+H-h_\ell}
  \qquad (n\ge0),
\]
a linear recurrence of order at most $H$, so the degrees compared
in Step~2 decide all of them by
\cref{lem:finite-comparison}\ref{fc:comparison}.

Suppose now that Step~1 stops at the bound.  All overlaps of degree
at most $N$ have been resolved, so
\cref{lem:diamond}\ref{dl:truncated} gives $c_n=\dim_{\Fp}
(\mathcal A/I)_n$ for every $n\le N$.  A deviation $c_n\ne b_n$ with
$n\le N$ therefore disproves strong freeness by
\cref{cor:word-counts}, while agreement up to $N$ says nothing
about the higher degrees, which is the verdict
\textsc{undecided}.
\end{proof}

\begin{remark}[Changes of variables and flags]\label{rem:flag}
The verdict of \cref{alg:completion} depends on the coordinates.
Let $I\subseteq\mathcal A$ be the two-sided ideal generated by the
input $\rho_1,\ldots,\rho_m$ of \cref{alg:completion}, let $k|\Fp$
be a finite extension with $q$ elements, let
$I_k\subseteq\mathcal A\otimes_{\Fp}k=k\langle\mgen_1,\ldots,\mgen_d\rangle$
be the ideal generated by $I$, and let $g\in\mathrm{GL}_d(k)$.  The
substitution $\varphi_g\colon\mgen_i\mapsto\sum_a g_{ia}\mgen_a$ is a
graded automorphism of $\mathcal A\otimes_{\Fp}k$, so it preserves
strong freeness and the Hilbert series of the quotient (the remark
following \cref{prop:Gaertner} and \cref{lem:descent}), but it
changes the ideal of high terms.  For the degree-lexicographic order
with $\mgen_1>\ldots>\mgen_d$, an upper triangular $g$ sends each
$\mgen_i$ to $g_{ii}\mgen_i$ plus smaller variables and therefore
fixes the high term of every polynomial, so the ideal of high terms
of $\varphi_g(I_k)$ depends only on the coset $gB$ of $g$ modulo the
upper triangular subgroup $B$ of $\mathrm{GL}_d(k)$, that is, on the
complete flag spanned by the columns of $g$.  There are
$\prod_{j=1}^{d}(q^j-1)/(q-1)$ such flags, $52$ for $d=3$ and
$q=3$.  A search for coordinates in which the completion terminates
is therefore a finite search through flags, and its failure is a
statement about those flags and that degree bound only.
\end{remark}

\section{Degeneracy loci and base change}
\label{app:rank-one-lemmas}

The geometric interpretation of transversality rests on three
general facts.  The arithmetic input is used only in
\cref{subsec:rank-drop-geometry}.  The first two facts concern
tangent spaces and smoothness of determinantal loci and are stated
over an arbitrary field $k$.  They are proved in full.  The third fact describes
closed points under finite extension of a finite base field.  It
is elementary commutative algebra, needed in a form that keeps
track of the residue fields.

Let $A,B$ be finite-dimensional $k$-vector spaces and write
$k[\varepsilon]=k[t]/(t^2)$.  We denote by $\Hom_k(A,B)$ also the affine
$k$-scheme representing the functor
$R\mapsto\Hom_R(A\otimes_kR,\,B\otimes_kR)$ on $k$-algebras.  Writing
$m=\dim_kA$ and $n=\dim_kB$, after a choice of bases this is the affine
space $\mathbb A^{mn}_k$, whose
$k$-points form the vector space $\Hom_k(A,B)$.  Accordingly, a morphism
from a $k$-scheme $U$ to $\Hom_k(A,B)$ is nothing but an $n\times m$
matrix of regular functions on $U$, the columns being the images of the
chosen basis vectors of $A$.  Equivalently, it is an
$\mathcal O_U$-linear map
$A\otimes_k\mathcal O_U\to B\otimes_k\mathcal O_U$.  Functorially, it
assigns to every $R$-point $u\in U(R)$ an $R$-linear map
$\Phi(u):A\otimes_kR\to B\otimes_kR$, compatibly with morphisms of
$k$-algebras $R$.  For a $k$-point $u$ of $U$, a tangent vector
$v\in T_uU$ is a $k[\varepsilon]$-point $u_v$ of $U$ reducing
to $u$, and the expansion
\[
  \Phi(u_v)=\Phi(u)+\varepsilon\,d\Phi_u(v)
\]
defines $d\Phi_u(v)\in\Hom_k(A,B)$, the entrywise derivative
of $\Phi$ at $u$ in the direction $v$.

For $\varphi\in\Hom_k(A,B)$ let
\[
  \pi_\varphi\colon B\longrightarrow\coker\varphi
\]
be the projection onto the cokernel.

We call a closed point $u$ of a $k$-scheme $Z$ \emph{isolated} if
some neighbourhood of $u$ contains no other point of $Z$, and
\emph{reduced} if the local ring $\mathcal O_{Z,u}$ contains no
nilpotent elements.  For $Z$ locally of finite type over $k$ the two
conditions together are equivalent to the local ring being the
residue field.  Indeed, by \cite[Tag~01TH]{StacksProject},
applied to the structure morphism $Z\to\Spec k$, a closed point
$u$ is isolated exactly when no other point of $Z$ specializes to
$u$, that is, when $\dim\mathcal O_{Z,u}=0$.  A zero-dimensional
local ring is reduced exactly when it is a field, its maximal
ideal being its nilradical.  The proofs
below use the conditions in this combined form.  A $k$-rational point of $Z$ has
as its image a closed point with residue field $k$.  We apply the
terminology and the notation $\mathcal O_{Z,u}$ to a $k$-rational
point $u$ through this image, so that reduced and isolated then
means $\mathcal O_{Z,u}=k$.

\begin{lemma}[Tangent spaces of determinantal loci]
\label{lem:rank-one-tangent}
Let $U$ be a scheme locally of finite type over $k$, let
$\Phi\colon U\to\Hom_k(A,B)$ be a morphism, let $r\ge1$, and let
$Z\subseteq U$ be the
closed subscheme defined by the $(r+1)\times(r+1)$ minors of the
matrix of
$\Phi$.  Let $u\in U(k)$ with $\rk\Phi(u)=r$ (so that in particular
$u\in Z(k)$), write $\varphi=\Phi(u)$, and let
\[
  d^\perp\Phi_u\colon T_uU\longrightarrow
  \Hom_k\bigl(\ker\varphi,\ \coker\varphi\bigr),
  \qquad
  d^\perp\Phi_u(v)=\pi_\varphi\circ\bigl(d\Phi_u(v)\bigr)\big|_{\ker\varphi}.
\]
Then:
\begin{enumerate}[label=\textup{(\alph*)}]
\item $T_uZ=\ker d^\perp\Phi_u$;
\item the following are equivalent:
  \begin{enumerate}[label=\textup{(\roman*)}]
  \item $d^\perp\Phi_u$ is injective;
  \item $Z$ is smooth over $k$ of dimension zero at $u$;
  \item $u$ is a reduced isolated point of $Z$.
  \end{enumerate}
\end{enumerate}
\end{lemma}

\begin{proof}
(a) A tangent vector $v\in T_uU$ lies in $T_uZ$ precisely
when the defining equations of $Z$ vanish at $u_v$, that is,
when all $(r+1)\times(r+1)$ minors of the matrix
\[
  \Phi(u_v)=\varphi+\varepsilon\,d\Phi_u(v)
\]
vanish in $k[\varepsilon]$.  Now, for a matrix
$\varphi$ of rank exactly $r$, the matrices
$\varphi+\varepsilon\psi$ with entries in $k[\varepsilon]$ whose
$(r+1)\times(r+1)$ minors all vanish are precisely those with
$\psi(\ker\varphi)\subseteq\operatorname{im}\varphi$.  In the
language of tangent spaces this is \cite[Lemma~2.2]{Jiang}.  So the
condition on $v$ is $d^\perp\Phi_u(v)=0$.

(b) If $d^\perp\Phi_u$ is injective, then $T_uZ=0$ by (a).  The
cotangent space $\mathfrak m_u/\mathfrak m_u^2$ of the noetherian local
ring $\mathcal O_{Z,u}$ then vanishes, so $\mathfrak m_u=0$ by Nakayama's
lemma and $\mathcal O_{Z,u}=k$.  A local ring equal to the base
field is regular of dimension zero, so $Z$ is smooth of dimension
zero at $u$ by the final implication of
\cite[Theorem~6.28]{GoertzWedhorn}, and $u$ is
reduced and isolated.  If $Z$ is smooth of dimension zero at $u$,
the same theorem gives $\dim T_uZ=0$, and $d^\perp\Phi_u$ is
injective by (a).  If $u$ is reduced and isolated, then
$\mathcal O_{Z,u}$ is a field, so again $T_uZ=0$ and (a) applies.
\end{proof}

\begin{proposition}[Corank two under an annihilating cosection]
\label{prop:cosection-abstract}\label{lem:syzygy-smooth}
Let $k$ be a field and let $U$ be a $k$-scheme, locally of
finite type and smooth over $k$ at a point $u\in U(k)$.  Let
\[
  \mathcal E\xrightarrow{\ \varphi\ }\mathcal F
  \xrightarrow{\ \lambda\ }\mathcal L
\]
be a complex of $\mathcal O_U$-modules with $\mathcal E$ and
$\mathcal F$ finite locally free of ranks $m$ and $n$, with
$n\ge3$ and $m\ge n-1$, and $\mathcal L$ invertible, and assume
$\rk\varphi(u)=n-2$ and
$\lambda(u)\ne0$.  Let $Z\subseteq U$ be the closed subscheme
defined by the ideal $I_{n-1}(\varphi)$ of $(n-1)$-minors, and
write $c=m-n+2$.  Fix a trivialization of the complex on a
neighbourhood of $u$, with $\varphi$ given by an $n\times m$
matrix $\Phi$ of regular functions.  Then
\[
  \theta_u\colon T_uU\longrightarrow
  \Hom_k\bigl(\ker\varphi(u),\
  \ker\lambda(u)/\operatorname{im}\varphi(u)\bigr),
  \qquad
  \theta_u(v)(w)=\overline{d\Phi_u(v)(w)},
\]
is a well-defined linear map, and:
\begin{enumerate}[label=\textup{(\alph*)}]
\item\label{ca:present}\label{sy:present} there are germs
  $\mathcal S_1,\dots,\mathcal S_{c}\in\mathfrak m_u$
  generating $I_{n-1}(\varphi)$ in $\mathcal O_{U,u}$, and
  $\theta_u$ is surjective if and only if the differentials
  $d(\mathcal S_1)_u,\dots,d(\mathcal S_c)_u$ are linearly
  independent;
\item\label{ca:smooth}\label{sy:smooth} if $\theta_u$ is surjective, then $Z$
  is smooth over $k$ at $u$, of codimension $c$ in $U$.
\end{enumerate}
In particular both assertions apply to a morphism
$\Phi\colon U\to\Hom_k(A,B)$ as in \cref{lem:rank-one-tangent}
with $r=n-2$, where $m=\dim_kA$ and $n=\dim_kB$, together with
a morphism $\lambda\colon U\to\Hom_k(B,k)$ satisfying
$\lambda(u)\ne0$ and $\lambda(x')\circ\Phi(x')=0$ for every
$k$-algebra $R$ and every $R$-point $x'$ of $U$: such a pair
defines a complex
$A\otimes_k\mathcal O_U\to B\otimes_k\mathcal O_U
\to\mathcal O_U$ as above, and the codimension in
\ref{sy:smooth} is $c=m-n+2=m-r$.
\end{proposition}

\begin{proof}
All assertions are local at $u$, so we may shrink $U$.  In
the fixed trivialization $\lambda$ is given by a row
$(\lambda_1,\dots,\lambda_n)$ of regular functions, with
$\lambda\Phi=0$ and $\lambda(u)\ne0$.  Evaluating $\lambda\Phi=0$ at $u$ gives
$\operatorname{im}\varphi(u)\subseteq\ker\lambda(u)$;
evaluating it at the $k[\varepsilon]$-point attached to
$v\in T_uU$ and taking the coefficient of $\varepsilon$ gives
$d\lambda_u(v)\,\varphi(u)+\lambda(u)\,d\Phi_u(v)=0$, so
$d\Phi_u(v)\bigl(\ker\varphi(u)\bigr)\subseteq\ker\lambda(u)$.
Hence $\theta_u$ is well defined, and it is linear in $v$ and
$w$.

\ref{ca:present}  Since $\lambda(u)\ne0$, after shrinking $U$
there is a commutative diagram
\[
\begin{tikzcd}
& & \mathcal O_U^m \arrow[d, "\Phi"] \arrow[dl, "\psi"'] & & \\
0 \arrow[r] & \mathcal K \arrow[r, "\iota"]
  & \mathcal O_U^n \arrow[r, "\lambda"]
  & \mathcal O_U \arrow[r] & 0
\end{tikzcd}
\]
with exact row, in which $\mathcal K$ is a subbundle of rank
$n-1$ with $\mathcal K(u)=\ker\lambda(u)$ and $\psi$ exists
because $\lambda\Phi=0$.  A frame $C$ of
$\mathcal K$ identifies $\psi$ with a morphism
$U\to\Hom_k\bigl(k^m,\ker\lambda(u)\bigr)$, the values of $C$
at $u$ serving as a basis of $\ker\lambda(u)$;
evaluating $\Phi=\iota\circ\psi$ at $u$ gives
$\iota(u)\circ\psi(u)=\varphi(u)$, so $\rk\psi(u)=n-2$.  In a
frame of $\mathcal O_U^n$ extending $C$ the matrix of
$\varphi$ becomes the matrix $\Psi$ of $\psi$ with a zero row
appended, and in matrix form the factorization reads
$\Phi=C\Psi$.  Since a change of frame multiplies the matrix by an
invertible matrix over $\mathcal O_U$,
$I_{n-1}(\Phi)=I_{n-1}(\Psi)$, the ideal of maximal minors of
the $(n-1)\times m$ matrix $\Psi$.

Since $\rk\psi(u)=n-2$, some $(n-2)$-minor of $\Psi$ is
invertible in $\mathcal O_{U,u}$, and Gauss elimination over
$\mathcal O_{U,u}$ yields invertible matrices $P$ and $Q$, of sizes
$(n-1)\times(n-1)$ and $m\times m$, with
\[
  P\Psi Q=\begin{pmatrix}I_{n-2}&0\\0&\mathcal S\end{pmatrix},
  \qquad
  \mathcal S=(\mathcal S_1,\dots,\mathcal S_{c}),
  \qquad
  c=m-n+2,
\]
the row $\mathcal S$ being the bottom-right block of this
$(n-1)\times m$ matrix, and $\mathcal S_j\in\mathfrak m_u$,
since a nonzero value at
$u$ would make $\rk\psi(u)$ exceed $n-2$.  Invertible
operations preserve ideals of minors, and the maximal minors of
the block matrix, each of them $\pm\mathcal S_j$ or $0$,
generate $(\mathcal S_1,\dots,\mathcal S_c)$.
Hence these germs generate $I_{n-1}(\Phi)$ in
$\mathcal O_{U,u}$, and $Z=V(\mathcal S_1,\dots,\mathcal S_c)$
near $u$.  Write
\[
  \alpha_u\colon T_uU\longrightarrow k^{c},
  \qquad
  v\longmapsto\bigl(d(\mathcal S_1)_u(v),\dots,
    d(\mathcal S_c)_u(v)\bigr).
\]
A $k[\varepsilon]$-point over $u$ lies on $Z$ exactly when
$\mathcal S_1,\dots,\mathcal S_c$ vanish on it, so
$T_uZ=\ker\alpha_u$.  On the other
hand
\cref{lem:rank-one-tangent}, applied to $\psi$, gives
$T_uZ=\ker d^\perp\psi_u$.  The maps $d^\perp\psi_u$ and
$\theta_u$ have the same target:
$\iota(u)\circ\psi(u)=\varphi(u)$ gives
$\ker\psi(u)=\ker\varphi(u)$ and
$\operatorname{im}\psi(u)=\operatorname{im}\varphi(u)$, so
$\coker\psi(u)=\ker\lambda(u)/\operatorname{im}\varphi(u)$.
Moreover $d^\perp\psi_u=\theta_u$: for
$w\in\ker\varphi(u)$ the product rule applied to
$\Phi=C\Psi$ gives
\[
  d\Phi_u(v)(w)
  =C(u)\,d\psi_u(v)(w)
  +dC_u(v)\,\underbrace{\psi(u)\,w}_{=\,0}
  =d\psi_u(v)(w),
\]
the second summand vanishing because
$\ker\psi(u)=\ker\varphi(u)$, and the last equality holding
because $C(u)$ is the inclusion of $\ker\lambda(u)$ into
$k^n$.  Hence
$\ker\alpha_u=\ker\theta_u$.  The targets of $\alpha_u$ and
$\theta_u$ both have dimension $c$, the second because
$\dim\ker\varphi(u)=c$ and
$\dim\bigl(\ker\lambda(u)/\operatorname{im}\varphi(u)\bigr)=1$.
So $\alpha_u$ and $\theta_u$ have equal rank, and $\theta_u$ is
surjective if and only if $\alpha_u$ is, that is, if and only if
$d(\mathcal S_1)_u,\dots,d(\mathcal S_c)_u$ are linearly
independent.

\ref{ca:smooth}  By \ref{ca:present} the germs
$\mathcal S_1,\dots,\mathcal S_c$ generate the ideal of $Z$ in
$\mathcal O_{U,u}$; if $\theta_u$ is surjective, their
differentials at $u$ are linearly independent.  Since
$T_uU=(\mathfrak m_u/\mathfrak m_u^2)^*$, their images in
$\mathfrak m_u/\mathfrak m_u^2$ are linearly independent.  The
local ring $\mathcal O_{U,u}$ is regular, since $U$ is smooth
at $u$ \cite[Lemma~6.26]{GoertzWedhorn}, so
$\mathcal O_{Z,u}
=\mathcal O_{U,u}/(\mathcal S_1,\dots,\mathcal S_c)$ is
regular of dimension $\dim\mathcal O_{U,u}-c$
\cite[Proposition~B.77~(3)]{GoertzWedhorn}.  As $\kappa(u)=k$, the final implication of
\cite[Theorem~6.28]{GoertzWedhorn} shows that $Z$ is smooth
over $k$ at $u$ of relative dimension
$\dim\mathcal O_{U,u}-c$, that is, of codimension $c$ in
$U$.

For the final assertion, the morphisms $\Phi$ and $\lambda$
define $\mathcal O_U$-linear maps whose values at every
$R$-point $x'$ are $\Phi(x')$ and $\lambda(x')$.  The entries
of the composite are regular functions on $U$ and vanish at
every $R$-point, hence vanish identically (evaluate at the
canonical point of each affine open), so the pair defines a
complex of the stated kind, with fibres at $u$ the given
$\varphi$ and $\lambda(u)$.
\end{proof}

\begin{lemma}[Closed points under base change]
\label{lem:reduced-isolated-descent}
Let $F$ be a finite field, let $Z$ be a scheme locally of finite
type over $F$, and let $k|F$ be a finite extension.
\begin{enumerate}[label=\textup{(\alph*)}]
\item\label{bc:points} The image $\xi$ in $Z$ of a $k$-rational
point of $Z$ is a closed point, and every closed point $\xi$ of
$Z$ is the image of a canonical $\kappa(\xi)$-rational point,
the canonical morphism $\Spec\kappa(\xi)\to Z$.
\item\label{bc:local} Let $u\in Z(k)$ with
image $\xi$ in $Z$, and regard $u$ also as a $k$-rational point
of $Z\otimes_Fk$, as in \cref{subsec:rank-drop-geometry}.  Then,
for every $n\ge0$, the scheme $Z\otimes_Fk$ is smooth of
dimension $n$ at $u$ if and only if $Z$ is smooth of dimension
$n$ at $\xi$.  Likewise $u$ is a reduced isolated point of
$Z\otimes_Fk$ if and only if $\xi$ is a reduced isolated point
of $Z$.
\end{enumerate}
\end{lemma}

\begin{proof}
\ref{bc:points}  By \cite[Proposition~3.8]{GoertzWedhorn}, a
morphism $\Spec k\to Z$ of $F$-schemes with image $\xi$ amounts
to an $F$-algebra homomorphism $\kappa(\xi)\to k$.  The residue
field of the image of a $k$-rational point therefore embeds into
$k$, hence is finite over $F$, and the image is closed
\cite[Proposition~3.33]{GoertzWedhorn}.  For a closed point
$\xi$, the residue field $\kappa(\xi)$ is a finite extension of
$F$, and the identity of $\kappa(\xi)$ provides the canonical
$\kappa(\xi)$-rational point with image $\xi$.

\ref{bc:local}  Finite fields are perfect, so the extension
$k|F$ is separable and $\Spec k\to\Spec F$ is étale
\cite[Tag~02GL]{StacksProject}.  Hence so is its base change
$Z\otimes_Fk\to Z$ \cite[Tag~02GO]{StacksProject}.  Over an
affine open neighbourhood $\Spec A\subseteq Z$ of $\xi$ this
morphism restricts to $\Spec(A\otimes_Fk)\to\Spec A$.  The ring
map $A\to A\otimes_Fk$ is therefore étale, and it is of finite
type over the noetherian ring $A$.  Hence the induced local
homomorphism
$\mathcal O_{Z,\xi}\to\mathcal O_{Z\otimes_Fk,\,u}$ is an étale
homomorphism of local rings \cite[Tag~039L]{StacksProject}.
Such a homomorphism leaves the dimension unchanged
\cite[Tag~039S]{StacksProject} and preserves and reflects
regularity \cite[Tag~025N]{StacksProject}.

The residue fields $\kappa(\xi)$ and $\kappa(u)$ are finite, so
the extensions $\kappa(\xi)|F$ and $\kappa(u)|k$ are separable.  At a point with separable
residue field, smoothness over the base field is equivalent to
regularity of the local ring \cite[Tag~00TV]{StacksProject}.
Consequently $Z\otimes_Fk$ is smooth at $u$ exactly when
$\mathcal O_{Z\otimes_Fk,\,u}$ is regular, and $Z$ is smooth at
$\xi$ exactly when $\mathcal O_{Z,\xi}$ is regular.  Since
regularity passes both ways along the \'etale homomorphism above,
the two smoothness statements are equivalent.  By
\cite[Theorem~6.28]{GoertzWedhorn}, smoothness of dimension $n$
at the closed point $u$, respectively $\xi$, means smoothness
there together with local dimension $n$.  As the two local
dimensions are equal, the smoothness statement follows.  A
reduced isolated point is one whose local ring is a field, as
recalled at the beginning of this appendix, that is, regular of
dimension zero.  Since regularity and dimension are unchanged
along $\mathcal O_{Z,\xi}\to\mathcal O_{Z\otimes_Fk,\,u}$, the
point $u$ is reduced and isolated if and only if $\xi$ is.
\end{proof}

\section{Verification and reproducibility}\label{app:verification}

This appendix records the checks behind the computer-assisted parts
of \cref{thm:mainexample,thm:mild-p3,thm:p7}.
\Cref{subsec:certificate-verification} defines the arithmetic
certificate, the stored output of \cref{alg:matrices} for one
field, and defines when a certificate \emph{verifies}: the checks
re-establish the claim carried by each stored datum by exact
arithmetic.  It also states in which sense the results are
unconditional.
\Cref{subsec:searches} records how the class fields and the norm
witnesses were found; nothing in the proofs depends on it.
\Cref{subsec:software} describes the released programs and the
public verification procedure.  The subsequent computations
(\cref{alg:transverse} and \cref{alg:completion})
read only the verified matrices; they are exact and deterministic,
and they are checked by direct recomputation.

\subsection{Certificates and their verification}
\label{subsec:certificate-verification}
\label{subsec:unconditionality}

The following definitions make the stored data and the checks
precise.
Fix an imaginary quadratic number field $K$ and an odd prime $p$ as
in \cref{sec:computation}, and write $g_1,\ldots,g_r$ for the
class-group generators fixed there, of orders $n_1,\ldots,n_r$.
Exactly $d$ of the orders are divisible by $p$, say those of
$g_{i_1},\ldots,g_{i_d}$.  The stored data refer to these
generators through three conventions.
\begin{itemize}
\item The basis element $e_j$ of $\Cl(K)[p]$ is the
  $(n_{i_j}/p)$-fold multiple of $g_{i_j}$.
\item A coordinate vector in $\Fp^{\,d}$ denotes a class of
  $\Cl(K)/p$ through the images of $g_{i_1},\ldots,g_{i_d}$.
\item A character vector $x\in\Fp^{\,d}$ denotes the character
  with $x(\overline{g_{i_j}})=x_j$ for $j=1,\ldots,d$.  Characters are read as
  functionals on $\Cl(K)/p$ as in \eqref{eq:sigma-character},
  so every character vanishes on classes of order prime to $p$.
\end{itemize}
In these
coordinates $\chi_1,\ldots,\chi_d$ are the unit vectors, and the
six input characters of \cref{alg:matrices} (with
$\chi_1{+}\chi_2{+}\chi_3$ in place of $\chi_2{+}\chi_3$ at
$p=3$) are called the \emph{stored characters}.

\begin{definition}\label{def:certificate}
A \emph{certificate} for $(K,p)$ consists of a header, stored
once, and a list of entries.  The header contains:
\begin{center}
\begin{tabular}{@{}l@{\quad}>{\raggedright\arraybackslash}p{0.40\textwidth}@{\quad}>{\raggedright\arraybackslash}p{0.34\textwidth}@{}}
\toprule
datum & stored form & role\\
\midrule
$K$ & a monic quadratic polynomial over $\mathbb Q$; an integer &
a defining polynomial of the base field and the discriminant
$D_K$\\
\addlinespace
$\Cl(K)$ & the elementary divisors $n_1,\ldots,n_r$;
representatives of the generators $g_1,\ldots,g_r$ & the class
group of $K$, presented by the generators fixed in
\cref{sec:computation}\\
\addlinespace
units & an order and a generator & the torsion units of
$K$\footnotemark\\
\addlinespace
$B_K$ & a list of two elements of $K$ & the integral basis of
$\mathcal O_K$, to which all base-field coordinates refer\\
\addlinespace
$T$ (at $p=3$) & a $3\times27$ matrix over $\Fp$ & a copy of the
matrix $T$ of \eqref{eq:T}, recorded at the search stage
(\cref{subsec:searches}) and called the \emph{expected
tensor}\footnotemark\\
\bottomrule
\end{tabular}
\end{center}
\addtocounter{footnote}{-1}%
\footnotetext{The quotient $N_x(t)/a'$ in \cref{lem:norm-sign}
is a unit of $\mathcal O_K$, hence a torsion unit.  The stored
group is the candidate list for that check, and the reduction
place must separate the candidates.}%
\stepcounter{footnote}%
\footnotetext{The name is that of the data format of
\cite{Repository}.  The asymmetry between the primes is
historical.  The $p=3$ search stage recorded the complete matrix
family of every field, so an independent record existed to embed;
at $p>3$ no record was embedded in the format.}%

An entry consists of the data listed below.  The stated roles
are claims, made precise and decided by the checks of
\cref{def:verifies}.
\begin{center}
\begin{tabular}{@{}l@{\quad}>{\raggedright\arraybackslash}p{0.40\textwidth}@{\quad}>{\raggedright\arraybackslash}p{0.34\textwidth}@{}}
\toprule
datum & stored form & role\\
\midrule
$x$, $e$ & a nonzero vector in $\Fp^{\,d}$; one of
$e_1,\ldots,e_d$ & the pair at which $D_x$ is
evaluated\\
\addlinespace
$f_{\mathrm{rel}}$, $f_{\mathrm{abs}}$ & monic polynomials, of
degree $p$ over $K$ and of degree $2p$ over $\mathbb Q$ & the
relative and the absolute defining polynomial of one field $L$\\
\addlinespace
$B_L$ & $2p$ elements of $\mathbb Q[T]/(f_{\mathrm{abs}})$, each
a polynomial in the image $\theta$ of $T$, a primitive element
of $L$ & the integral basis of $L$, to which all coordinates of
the entry refer\\
\addlinespace
$w$ & a vector in $\mathbb Q^{2p}$ & the coordinates in $B_L$ of
the image of $\theta$ under the normalized generator of
$\Gal(L\mid K)$\\
\addlinespace
$a'$, $J$ & an element of $K^\times$; an ideal of $\mathcal O_K$
in Hermite normal form & a pair as in \eqref{eq:pair-aJ}
representing $e$\\
\addlinespace
$t$, $I'$ & a formal product of elements of $L$ with integer
exponents; an ideal of $\mathcal O_L$ in Hermite normal form & a
norm witness for $x$ and $(a',J)$ (\cref{def:norm-witness})\\
\addlinespace
$\ell$, $\mathfrak q$ & an odd prime; a prime ideal of
$\mathcal O_K$ above it & the reduction place for
\eqref{eq:checkAC2}\\
\addlinespace
$v$ & a vector in $\Fp^{\,d}$ & the coordinate vector of
$D_x(e)$\\
\bottomrule
\end{tabular}
\end{center}
The table lists the mathematical content of an entry.  The
layout of the serialized data file, with its labels and field
order, is documented in \cite{Repository}.  The list contains
exactly one entry for each of the six
stored characters and each basis element $e_j$.
\end{definition}

Part~W4 of \cite{Repository} presents one certificate entry as
a worked example, item by item.

\begin{definition}\label{def:verifies}
A certificate \emph{verifies} if the conditions listed below
hold, each decided by exact arithmetic.  The conditions of group \textup{(V1)}
concern the base field and are checked once.  For every entry,
the conditions of group \textup{(V2)} establish the field $L$
and the automorphism $\sigma_L$, and those of group
\textup{(V3)} verify the arithmetic identities on them.
\begin{center}
\begin{tabular}{@{}l@{\quad}l@{\quad}>{\raggedright\arraybackslash}p{0.62\textwidth}@{}}
\toprule
check & datum & condition\\
\midrule
\textup{(V1)} & $D_K$ & the recomputed discriminant of $K$ is
the stored one\\
\addlinespace
\textup{(V1)} & $B_K$ & the recomputed integral basis is the
stored one\\
\addlinespace
\textup{(V1)} & $\Cl(K)$ & the class group, recomputed and
established unconditionally, has the stored elementary divisors
and generators\\
\addlinespace
\textup{(V1)} & units & the torsion units of $K$ are the stored
ones\\
\midrule
\textup{(V2)} & $f_{\mathrm{rel}}$, $f_{\mathrm{abs}}$ &
$f_{\mathrm{rel}}$ is irreducible over $K$ and defines a field
$L=K[T]/(f_{\mathrm{rel}})$ of degree $p$; the extension is
unramified everywhere, decided through the discriminant relation
$D_L=D_K^{\,p}$; $f_{\mathrm{abs}}$ is the absolute defining
polynomial derived from $f_{\mathrm{rel}}$, so that the two
polynomials present the same field\\
\addlinespace
\textup{(V2)} & $B_L$ & the change of basis between $B_L$ and a
recomputed integral basis is integral with determinant $\pm1$,
so that $B_L$ is itself an integral basis and the coordinate
vectors of the entry denote well-defined elements and ideals of
$\mathcal O_L$\\
\addlinespace
\textup{(V2)} & $w$ & the element with coordinate vector $w$ is
the image of $\theta$ under a power of an independently
recomputed generator of $\Gal(L\mid K)$; the automorphism so
determined, denoted $\sigma_L$, fixes $K$ and has exact order
$p$\\
\addlinespace
\textup{(V2)} & $x$ & the Artin symbols of $g_1,\ldots,g_r$,
written as powers of $\sigma_L$, have exponent vector $x$, so
that $L=L_x$ and $\sigma_L=\sigma_x$ in the normalization of
\eqref{eq:sigma-character}\\
\midrule
\textup{(V3)} & $a'$, $J$ & $(a')\,J^{\,p}=\mathcal O_K$, and
the class coordinates of $J$ with respect to the stored
generators are those of $e$\\
\addlinespace
\textup{(V3)} & $t$, $I'$ & identity \eqref{eq:checkAC1} holds
as an equality of ideals in Hermite normal form\\
\addlinespace
\textup{(V3)} & $\ell$, $\mathfrak q$ & the principal ideal
generated by $N_x(t)/a'$, formed factor by factor from the
stored representation of $t$, equals $\mathcal O_K$;
$\mathfrak q$ lies above the odd prime $\ell$, every factor of
the quotient is a unit at $\mathfrak q$, and exact reduction at
$\mathfrak q$ gives $N_x(t)/a'=1$\\
\addlinespace
\textup{(V3)} & $v$ & the recomputed class $[N_x(I')]$, at
$p=3$ the class $[N_x(I')\,J]$ of \cref{prop:AC-p3}, has
coordinate vector $v$\\
\bottomrule
\end{tabular}
\end{center}
\end{definition}

\begin{proposition}\label{prop:soundness}
If a certificate for $(K,p)$ verifies, then for every entry
\[
  D_x(e)=v \quad\text{in } \Cl(K)/p,
\]
with respect to the stored generators.  In particular the six
matrices assembled from the eighteen entries are the matrices of
the secondary norm operators of the six stored characters, and
they determine the complete Massey product.
\end{proposition}

\begin{proof}
By \textup{(V2)} the stored field is everywhere unramified, cyclic
of degree $p$, and its Artin symbols realize the prescribed
character; hence $L=L_x$, and the automorphism $\sigma_L$ determined by
$w$ satisfies the normalization
\eqref{eq:sigma-character}: $\sigma_L=\sigma_x$.  By \textup{(V1)}
the class coordinates in \textup{(V3)} refer to the stored
generators of $\Cl(K)$, and the entry's pair $(a',J)$ represents
$e$ as in \eqref{eq:pair-aJ}.  Identity \eqref{eq:checkAC1} holds by
\textup{(V3)}; \cref{lem:norm-sign} then forces $N_x(t)=\pm a'$,
and the exact reduction at $\mathfrak q$ excludes the sign $-1$,
which is identity \eqref{eq:checkAC2}.  Thus the hypotheses of
\cref{prop:AC}, at $p=3$ of \cref{prop:AC-p3}, are
verified, and $D_x(e)$ is the recomputed norm class of
\textup{(V3)}, whose coordinate vector is $v$.  The final
statements follow as stated: for $p>3$ by
\cref{thm:reconstruction} applied to the six assembled matrices,
and at $p=3$ by \cref{lem:four-evaluations} applied to the four
of them at the basis characters and at
$\chi_1{+}\chi_2{+}\chi_3$.
\end{proof}

\begin{remark}[Verification at higher rank]\label{rem:rank-d-verification}
For a field with $d_p\Cl(K)=d\ge3$ the same scheme applies with
exactly $d$ of the elementary divisors divisible by $p$, the basis
$e_1,\ldots,e_d$ of $\Cl(K)[p]$, the $d(d+1)/2$ characters of
\cref{alg:matrices} and the $d(d+1)/2$ assembled matrices with their
$d\cdot d(d+1)/2$ entries.  \Cref{def:certificate},
\cref{def:verifies} and the proof of \cref{prop:soundness} apply
verbatim, since nothing in them uses $d=3$ beyond notation.  The
record of the field of rank four follows this scheme, with ten
matrices and forty entries.
\end{remark}

The certificates of all three primes are checked by one program:
\texttt{verify\_certificate.c}, in the top-level directory
\texttt{verifier} of \cite{Repository}.  It checks the conditions
\textup{(V1)}--\textup{(V3)} directly, by exact arithmetic, and
a successful run establishes that its input verifies.  It computes no class group,
unit group, or ray class group of any of the fields $L_x$.
Entry by entry the chain
\[
  x \;\longrightarrow\; L_x\mid K \;\longrightarrow\; \sigma_x
  \;\longrightarrow\; \eqref{eq:checkAC1}, \eqref{eq:checkAC2}
  \;\longrightarrow\; [N_x(I')]
\]
is re-established from the stored data alone, independently of the
search that produced the data.  The checks change with the prime
only in the stored character family and, at $p=3$, in the
factor $J$ of \cref{prop:AC-p3} in \textup{(V3)}, which enters
every recomputed class.

\Cref{def:verifies} stops at the six matrices; reconstructing the
Massey product from them is linear algebra over $\Fp$, with no
further stored data entering.  The program carries this
reconstruction out as well.  From the verified matrices it
computes the three polarizations $\Delta D(\chi_i,\chi_k)$,
$i<k$, of \eqref{eq:DeltaD}.  At $p=3$, where $\chi_2+\chi_3$ is
not among the stored characters, it solves the expansion
$D_{\chi_1+\chi_2+\chi_3}=\sum_i D_{\chi_i}-\sum_{i<k}\Delta
D(\chi_i,\chi_k)$, bilinearity as in \cref{thm:reconstruction},
for the missing polarization, so all six stored matrices enter.
It then assembles the complete matrix $T$, with the coefficients
placed by the word-order convention of \eqref{eq:T}, and rejects
unless the shuffle identities \eqref{eq:shuffle} hold.  At $p=3$ the six stored
characters overdetermine the product (four would suffice by
\cref{lem:four-evaluations}), and these identities cross-check
the assembled data.  Verification uses the certificate
alone.\footnote{Some $p=3$ directories carry auxiliary
factorization hints, read alongside the certificate.  Every hint
is proved prime before use, so a hint can shorten a factorization
or cause rejection, but no hint can make a failing certificate
pass.}  Beyond it, the
program compares the reconstruction against the expected tensor
when the header contains one (at $p=3$), and against a result
record of the search stage when such a record is passed
alongside.  A mismatch causes rejection.  Neither comparison
enters \cref{def:verifies}; both are redundancy checks against
the search stage.

The proofs of \cref{sec:results} use three further
computations: the ranks and determinants of
\cref{subsec:rank-one-computation}, the transversality decision
of \cref{alg:transverse}, and the strong-freeness computations
of \cref{alg:completion}.  These are performed by separate
programs of \cite{Repository}, which read nothing of the
certificate beyond the matrices, or the matrix $T$ reconstructed
from them.

The repository \cite{Repository} provides a test program that
alters temporary copies of two stored certificates, one from the
$p=3$ collection and one from the $p=5$ collection.  Five of
its alterations change a stored arithmetic datum: the character
vector, the normalized automorphism, the multiplicity of an
entry, a norm-class vector, and the absolute defining polynomial
are changed separately; the remaining alterations malform the
file itself or alter the embedded expected tensor.  The program requires the verifier to reject
each altered copy with the message of the check that guards the
altered object, and it records the outcomes.  These rejections are
a check on the implementation, not on the mathematics.  They confirm
that the implementation ties each altered object to the stated
character, field, and norm computation, rather than checking the
displayed equations in isolation.

The verification uses PARI's exact arithmetic for number fields,
Galois actions, ideals, and residue fields, and no floating-point
comparison enters it.  The results of this paper are unconditional in the
following sense.  A generalized Riemann hypothesis could enter the
chain at exactly one point, namely the class- and unit-group
computation for the base field $K$, whose initial run uses
analytic bounds.
\textup{(V1)} removes this dependence by establishing the
class group unconditionally before any entry is checked.  No other
step of the verifier rests on an unproved hypothesis.  After
\textup{(V1)}, only exact relative-field, Galois, ideal, and
finite-field arithmetic is used, and everything after the verifier
is finite linear algebra over $\Fp$.  The searches are free to use
conditional or floating-point computations
(\cref{subsec:searches}); their output enters only as data.  What remains assumed is the correctness of
PARI's exact arithmetic, not any unproved hypothesis of number
theory.

\subsection{The oracles: finding the class fields and the norm
witnesses}\label{subsec:searches}

The verifier checks the stored class fields, normalized
automorphisms, and norm witnesses of
\cref{subsec:construct-verify}, but it does not construct them.
The class fields and the norm witnesses are found by searches and
then written into the certificate; the normalized automorphisms
are computed from the class fields.
The searches play the role of oracles.  They may propose data
obtained by any method whatsoever, because every claim the proofs
draw from a certificate is re-established by the verifier from
the stored data alone.  A defective oracle can fail to produce a
certificate that verifies, but cannot make a false claim pass.  Nothing in this paper depends on how
the oracles work, but this subsection records how they did.

The constructive part was carried out with Eric Ahlqvist's
\texttt{Massey-pari} program, which grew out of the computations
accompanying Ahlqvist--Carlson \cite{AhlqvistCarlson}.  We extended
it so that the character $x$ can be prescribed and the resulting
secondary norm operator is returned in fixed class-group
coordinates.

The class fields were first obtained by the Kummer-theoretic
route, which computes the class group and unit group of
$K(\zeta_p)$ and obtains $L_x$ from the ray class-field routines.
Its cost is dominated by $\Cl(K(\zeta_p))$.  It produced the
whole $p=3$ collection; at $p=5$ it carried the fields of
smallest absolute discriminant and then became impracticable.

At that point the searches were delegated to large language
models used as coding agents.  The oracle principle is what made
the delegation safe.  The models proposed and implemented an
analytic construction of the class fields: period sums of
Schertz-corrected double-eta values over the cosets of the
character kernel, followed by exact recognition of the defining
polynomial \cite{EngeSchertz,AtkinMorain}.  This is a
complex-analytic, floating-point computation.  Its rounded output,
the integral defining polynomial, is checked exactly against the
prescribed data (degree, ramification, and Artin symbols)
before a certificate entry is written.  With the degree
parametrized, this route carried the remainder of the $p=5$ range
and all three fields at $p=7$.

Finding a norm witness (\cref{def:norm-witness}) comes down to
solving one equation in $\Cl(L_x)$ and determining one unit of
$\mathcal O_{L_x}$, through PARI's class-group and unit-group
routines in these fields of degree $2p$.  For the third
field at $p=7$, with $D_K=-931506071$, a single such computation in
degree $14$ was infeasible even so.  Here the models located the norm-relation
method of
Biasse--Fieker--Hofmann--Page \cite{BiasseFiekerHofmannPage},
implemented it for the $D_7$-extensions in the Hecke system
\cite{Hecke}, and coordinated
the PARI and Hecke computations field by field.  The resulting
data enter the same certificate format and are verified like
every other.

The models' programs and choices are not part of the argument;
only their verified output is.

\subsection{Software and the public verification procedure}
\label{subsec:software}

\Cref{subsec:certificate-verification} isolated the one assumption on
which the verification rests, the correctness of PARI's exact
arithmetic.  In the released software that assumption is
PARI~2.17.4 \cite{PARI}, built with the one-file patch described
below, which removes three \texttt{static} qualifiers and alters no
computation.
Every arithmetic claim encoded in an individual certificate is checked
from that certificate alone.  The remaining claims are finite
computations from the resulting matrices and tensors, whose programs
and records are listed below.  All paths
quoted in this subsection refer to the archived version of
\cite{Repository}.

The table lists the top-level components of \cite{Repository} that
carry evidence, with the role each plays in the argument.
\begin{longtable}{@{}>{\raggedright\arraybackslash}p{0.22\textwidth}@{\quad}>{\raggedright\arraybackslash}p{0.33\textwidth}@{\quad}>{\raggedright\arraybackslash}p{0.31\textwidth}@{}}
\toprule
path & contents & role in the argument\\
\midrule
\endhead
\texttt{certificates/} & one collection per prime and one directory
per field, each with the data file \texttt{certificate.gp}; at $p=3$
some directories contain in addition a file of proven prime factors, at
$p=7$ the search matrices and transversality protocols & the stored
data of \cref{def:certificate}, on which
\textup{(V1)}--\textup{(V3)} are decided; the factor files only
speed up the factorizations, each factor being proved prime before
it is used\\
\addlinespace
\texttt{verifier/} & \texttt{verify\_certificate.c}, its
\texttt{Makefile}, the rejection tests \texttt{test\_rejections.py},
and their recorded outcomes & the exact arithmetic check of
\cref{def:verifies}, with the negative tests of
\cref{subsec:certificate-verification}\\
\addlinespace
\texttt{records/} & result records with run logs, source tensors,
strong-freeness verdicts with their second-engine crosschecks, the
criterion reports on the Bockstein cone at $p=3$, the exhaustive
Anick scans over $\Fthree$ and $\Fnine$, the
transversality records at $p=5$, Part~W7 with the substitution, the
coefficient field and the high terms for each field decided after a
change of variables, a standalone verifier for these completions,
and the records of the field of rank four &
the provenance of the search stage, the redundancy data compared
against the certificates, and the outcome protocols of the finite
computations\\
\addlinespace
\texttt{tools/} & the harnesses that run the verifier over a
collection, the drivers of the finite computations, the PARI build
script, and the generator of the human-readable companion of a
certificate & implements the finite computations of
\cref{subsec:rank-one-computation}, \cref{alg:transverse}, and
\cref{alg:completion}; the provenance of their matrix and tensor
inputs and the crosschecks are recorded under \texttt{records/}\\
\addlinespace
\texttt{src/}, \texttt{headers/} & the C search program, and the
declarations of the three PARI routines named below & produces the
stored data and takes no part in checking it, apart from the
declaration header, which the verifier includes as well\\
\addlinespace
\texttt{cm-constructor/}, \texttt{parallelization/} & the
complex-multiplication construction of the class fields, and the
per-character drivers used for the most expensive fields & further
programs of the search stage (\cref{subsec:searches}), on which no
proof depends\\
\addlinespace
\texttt{census/} & the discriminant list of each prime, the
provenance of the class-group tables, and a comparison against an
independent scan & the completeness of the range $|D_K|<2^{30}$\\
\addlinespace
\texttt{worksheets/} & Parts W1--W4 as one document, the tables W5
and W6 as documents of their own & human-readable restatements of
computations established elsewhere, no step of a proof\\
\addlinespace
\texttt{tests/} & regression tests of the discriminant scanner and of
the tools of the finite computations & checks on those
implementations, not on the mathematics\\
\addlinespace
\texttt{doc/}, \texttt{REPRODUCING.md} & the one-file PARI patch with
its rebuild steps; the principal build and rerun commands &
the software environment, and the reproduction of everything beyond
the minimal procedure below\\
\addlinespace
\texttt{MANIFEST.sha256} & one SHA-256 line per other tracked release
file &
identifies byte for byte the files to which the paths above refer\\
\bottomrule
\end{longtable}

The released verifier requires PARI~2.17.4.  Besides its exact
number-field, ideal, and finite-field arithmetic, the checks of
\textup{(V2)} use the internal relative-field routines
\texttt{rnfcycaut}, \texttt{allauts}, and \texttt{cyclicrelfrob},
which stock PARI~2.17.4 does not export.  None of the three computes
a relative class group or unit group.  The one-file patch that
exposes them is documented in \texttt{doc/}, and the script
\texttt{tools/build-patched-pari.sh} applies it and installs the
result into a private prefix.  From the repository root the minimal
public procedure consists of four commands.
\begin{verbatim}
tools/build-patched-pari.sh
make -C verifier PARI=$HOME/.local
verifier/verify_certificate certificates/p5/K-2800905-p5/certificate.gp
make -C verifier PARI=$HOME/.local check
\end{verbatim}
The build script ends by reporting \texttt{exported patch
symbols:\ 3 of 3}, and it stops if fewer than three are exported.  The second
command produces the verifier executable in \texttt{verifier/}.
The third prints one line per certificate entry, carrying the outcome
of every check of \cref{def:verifies} and the recomputed norm class,
then the six matrices and the reconstructed matrix $T$, and ends by
printing \texttt{CERTIFICATE VERIFIED}.  A failed check instead
prints \texttt{CERTIFICATE FAILURE} with the entry, where applicable,
and the failed check, and the program exits with a nonzero status.

The fourth command runs the implementation tests of
\cref{subsec:certificate-verification}.  It verifies three stored
certificates unchanged, one for each prime, and then requires the
verifier to reject fifteen altered copies of the $p=3$ and the $p=5$
fixture, eight and seven respectively.  It is a regression test on
these three fixtures, not a run over all $12\,956$ stored
certificates, which is a separate and far longer computation.
A single run of
\texttt{python3 tools/verify\_all\_certificates.py --all
--read-only} verified all $12\,956$ stored certificates in four
and a quarter hours on one core of an AMD Ryzen~9~9950X3D and
reported no failures.  A second run, on one core of an AMD
Ryzen~9~3950X, took nine hours and reported the same.

Verifying a stored certificate takes the single command above,
whereas regenerating one invokes the full search of
\cref{subsec:searches}.  Regeneration is therefore a separate task
and no part of the public procedure.  \texttt{REPRODUCING.md} of
\cite{Repository} documents it, together with the principal build and
rerun commands; the collection-specific \texttt{README}s document the
remaining computations.

Each of the $379$ strong-freeness verdicts of
\cref{sec:example-p3} is computed by two of three engines: the
implementation of \cref{alg:completion} in
\texttt{tools/strong-freeness/}, Singular's Letterplace subsystem
\cite{Singular,LaScalaLevandovskyy}, which recomputes the high terms
while the final word count is shared, and a second implementation of
\cite{Repository}.  The records under \texttt{records/p3/results/}
name the pair used for each collection, and the paired verdicts agree
on all $379$ fields.

\vspace{2\baselineskip}

\noindent\textsc{Universität Heidelberg, Institut für Mathematik,
Im Neuenheimer Feld 205, D-69120 Heidelberg, Deutschland}

\medskip

\noindent\textit{E-mail address}: \texttt{vogel@mathi.uni-heidelberg.de}

\end{document}